\documentclass{amsart}

\usepackage{latexsym,amssymb}
\usepackage{multicol}

\newtheorem{theo}{Theorem}[section] 
\newtheorem{lemma}[theo]{Lemma}    
\newtheorem{corol}[theo]{Corollary}
\newtheorem{propo}[theo]{Proposition}

\theoremstyle{definition}

\newtheorem*{notation}{Notation}
\newtheorem*{tab}{Table}
\newtheorem{remar}[theo]{Remark}
\newtheorem{proc}[theo]{Procedure}

\numberwithin{equation}{section}

 \newcommand{\ld}{,\ldots ,}
\newcommand{\ra}{ \rightarrow }
\newcommand{\Id}{\mathop{\rm Id}\nolimits}
 \newcommand{\Irr}{\mathop{\rm Irr}\nolimits}
 \newcommand{\CC}{\mathbb{C}}
 \newcommand{\ep}{\varepsilon}
\newcommand{\lam}{\lambda }
\newcommand{\si}{\sigma }
 \def\f{{following }}
\def\ii{{if and only if }}
 \def\ir{{irreducible }}
 \def\irt{{irreducible. }}
\def\irr{{irreducible representation }}
 \def\itf{{It follows that }}
\def\mult{{multiplicity }}
\def\rep{{representation }}
 \def\reps{{representations }}
 \def\rept{{representation. }}
 \def\syl{{Sylow $p$-subgroup }}
 \newcommand{\ZZ}{{\mathbb Z}}
\newcommand{\med}{\medskip}

\newcommand{\Tr}[2]{\overline{#2}_{#1}}

\newcommand{\tabA}[1]{\textsc{chevie} sets ${\bf X}_h$, their degrees and parameters for #1.}
\newcommand{\tabB}[1]{{Some $Syl_p$-decompositions of \ir characters for #1.}}
\newcommand{\tabC}[1]{The $Syl_p$-vanishing characters of degree $St(1)$ for #1.}
\newcommand{\tabD}[1]{The $p$-vanishing characters of degree $St(1)$ and their values for
#1.}
\newcommand{\tabDD}[1]{\noindent the character #1. This character differs from
$St$ only on the following classes:}

\begin{document}

\title[On characters of Chevalley groups]{On characters of Chevalley groups vanishing
at the non-semisimple elements}

\author{M.A. Pellegrini}
\address{Departamento de Matem\'atica, Universidade de Bras\'ilia, 
70910-900 Bras\'ilia - DF, 
Brazil}
\email{pellegrini@unb.br}
\thanks{The first author was supported by FEMAT and CNPq - Brazil}

\author{A.E. Zalesski}
\address{Universit\`a degli Studi di Milano-Bicocca, 
Dipartimento di Matematica e Applicazioni, 
via R.Cozzi 53,  20125 Milano,
Italy}
\email{alexandre.zalesski@gmail.com}
\thanks{The work of the second author is a part of 
the collaboration project ``Cohomology and representations'' between the 
University of Milano-Bicocca (Italy) and the University of Brasilia (Brazil).}

\dedicatory{Dedicated to Lino Di Martino on the occasion of his 65th birthday}


\begin{abstract}
Let $G$ be a finite simple group of Lie type. In this paper we study  characters of  $G$
that vanish at the non-semisimple elements and whose degree is equal to the order of a
maximal 
unipotent subgroup of $G$. Such characters can be viewed as a natural generalization of
the Steinberg character.
For groups $G$ of small rank we also determine  the characters of this degree 
vanishing only at the non-identity unipotent elements. 
\end{abstract}

\maketitle

\section{Introduction}

Let $G$ be a finite simple group of Lie type in defining characteristic
$p$, and let $|G|_{p}$ denote the order of a Sylow $p$-subgroup
of $G$.  There is a unique \irr of $G$  of
dimension $|G|_{p}$ over the complex numbers, and this is called {\it the Steinberg
representation}. We denote here by $St$ the character of the
Steinberg representation. The Steinberg character plays a
prominent role in the character  theory of groups of Lie type, so
this encourages one to look for generalizations 
of such character. In particular, in this paper we consider 
$p$-vanishing characters of degree $|G|_p$.

Let $H$ be any finite group and $p$ a prime. Elements $x\in H$ of  order divisible by $p$
are called $p$-singular.
  Characters (in
general reducible) that vanish at the $p$-singular elements of $H$ are
called $p$-{\it vanishing characters} in this paper. The characters of projective modules over a field of
characteristic $p$ are $p$-vanishing.  We think that $p$-vanishing characters are of significant interest due to the 
connection with projective modules. In the case where $H$ is $p$-solvable, the well known
result of Fong (1962) tells us  
that
every $p$-vanishing character is the character of a projective
module \cite[32.17]{CR1}. In general, this is not true.  To the best of our knowledge,
no publications on $p$-vanishing  characters appears since 1962.

As $St$ is the
character of a projective module, one may think of the
characters of  projective modules of dimension $|G|_p$ as
natural analogs of the Steinberg character. Malle and Weigel
\cite{MW} determined the projective modules of degree $|G|_p$
whose character has the trivial character  $1_G$ as a
constituent.  If $G$ is a group of Lie type in
characteristic $p$, the second named author showed in \cite{Z}
that the restriction on $1_G$ can be dropped. In other words, the Steinberg
character and the characters obtained in \cite{MW} are the only
characters of projective modules of $G$ of degree $|G|_p$. 
(Note that for other primes dividing $|G|$ and for other groups the
problem remains open.) In both the papers \cite{MW} and \cite{Z} the argument
relies on the $p$-modular representation theory. Our results on $p$-vanishing characters of degree $|G|_p$ are stronger,
and make no use of the theory of modular representations. Our main result is the \f theorem:

\begin{theo}\label{mt0}
Let $G$ be a
simple group of Lie type in defining characteristic $p$. If  $G\in \{A_n(q)$, $n>4$, ${}^2A_n(q)$,  $n>2$, $B_n(q)$, $q$ odd, $n>5$,
 $C_n(q)$, $n>2$, $D^+_n(q)$, $n>5$, $D^-_n(q)$, $n>3$,  ${}^3D_4(q)$,
$E_6(q)$, ${}^2E_6(q)$, $E_7(q)$, $E_8(q)$, $F_4(q)$, ${}^2F_4(q^2)$, $G_2(q)\}$
then the Steinberg character is the only  $p$-vanishing character of degree $|G|_p$ (in particular every projective module of this dimension is the Steinberg module).
\end{theo}

Naturally, one wishes to know whether the converse is true. Unfortunately, we have not been able to
manage with the groups $B_3(q)$, $B_4(q)$, $B_5(q)$, $D_4(q)$ and $D_5(q)$ for arbitrary $q$.    
For all other groups we have the full answer:

\begin{theo}\label{mt1}
Let $G$ be a
simple group of Lie type in defining characteristic $p$.
\begin{itemize}
\item[(1)] Groups $A_1(q)$,  $A_2(q)$, ${}^2B_2(q^2)$, ${}^2G_2(q^2)$
have reducible $p$-vanishing characters of degree $|G|_p$, 
where $q$ is any $p$-power in the former two cases, whereas in the latter
two cases $q^2$ is an odd $p$-power for $p=2,3$, respectively).
\item[(2)] Groups ${}^2A_2(q)$, $A_3(q)$, $A_4(q)$ have reducible $p$-vanishing characters
of degree $|G|_p$ \ii $q+1$ is divisible by $3$.
\item[(3)] Groups $C_2(q)\cong B_2(q)$ have reducible $p$-vanishing
characters of degree $|G|_p$ \ii $q+1$ is divisible by $7$.
\end{itemize}
\end{theo}

Moreover, we give an explicit decomposition of every  $p$-vanishing character as a sum of \ir ones, and this is 
used in the   proof of Theorem \ref{mt0}. For instance, in order to 
prove that $PSL(6,q)$ belongs to the list of  Theorem \ref{mt0}, we need to know all \ir constituents of 
the $p$-vanishing characters of degree $q^{10}$ for the group $PSL(5,q)$.

As long as we have obtained the decompositions of  the  characters  in items (1), (2), (3) of Theorem  \ref{mt1}
in terms  of \ir ones, we are able  to compute 
the  values of every $p$-vanishing character and compare it with those of the Steinberg character.  
 (Full information is tabulated 
at the end of the paper.) We find out that every such character has real values and does not vanish at the semisimple elements. An essential feature of our results is that they are uniform with respect of  $q$.
 
One of possible applications of our results to projective modules and the structure of the decomposition matrices can be outlined as follows.
Let $\Psi,\Phi$ be the characters of projective modules; we write  $\Psi>\Phi$ if $\Psi-\Phi$ is a proper character.
Looking at known decomposition matrices one observes a lot of examples of projective  indecomposable modules whose characters $\Psi,\Phi$ satisfy $\Psi>\Phi$. However, their dimensions cannot differ by $|G|_p$ for most groups:  

\begin{corol} Let $G$ be a group as in \emph{Theorem \ref{mt0}}. Let $\Psi>\Phi$ be the characters of two projective modules of $G$.
Then either $\Psi-\Phi=St$ or  $\Psi(1)-\Phi(1)\geq 2|G|_p$.\end{corol} 

If $\Psi,\Phi$ are the characters of indecomposable projective modules, then $\Psi>\Phi$ means that, for every row of the 
decomposition matrix, 
the $\Phi$-entry is not greater than the $\Psi$-entry. The above corollary implies $\Psi(1)-\Phi(1)\geq 2|G|_p$
(as the first option of the corollary cannot occur).

Our proofs of Theorems \ref{mt0} and \ref{mt1} is based on the contemporary theory of 
characters of groups of Lie type. Particularly important role is played by the
Gelfand-Graev characters and  
Harish-Chandra theory, together with regular and cuspidal characters. 
Note that if $\chi$ is a $p$-vanishing character, then so is
the Harish-Chandra restriction $\overline{\chi}_L$ of $\chi$ to every Levi subgroup $L$ of $G$
(see \cite{MW,Z} or Lemma \ref{x2a} below). Moreover,  $\chi(1)/|G|_p=\overline{\chi}_L(1)/|L|_p$. 
Let $L_u$ be the subgroup
of $L$ generated by the unipotent elements of $L$. Then, the  restriction $\tilde \chi_L$
of $\overline{\chi}_L$ to $L_u$
is a $p$-vanishing character of $L_u$ of the same degree. Therefore, through induction assumption we have information
on properties of  $\tilde \chi_L$ for all Levi subgroups $L$ of $G$. Therefore, in order to run induction, we need first to prove
Theorem \ref{mt1} 
for groups of BN-pair rank $1$. This case is dealt with in   Section
\ref{rk1}.
Surprisingly, only for $G=SL(2,q)$  this is easy. 

One could expect that these groups constitute the base of induction for our proof.
However, this is not the case as the situation is more complex. Most groups of rank $2$, especially, $Sp(4,q)$, ${}^3D_4(q)$ 
and ${}^2F_4(q^2)$
require quite a lot of computational work. Probably, performing this work manually 
is not realistic, but with use of \textsc{chevie}   we complete the analysis of groups of rank $2$
(see Section \ref{rk2}).

Our results on groups of rank $2$ are sufficient to run induction for unitary groups and $D^-_n(q)$. However, for the series $SL(n,q)$ we cannot run induction from
 $n<6$, and for the series  $Sp(2n,q)$, 
the induction 
starts from $n=3$. Groups $B_3(q)$ and $D_4^+(q)$ should also be handled by computations.
However, the \textsc{chevie} 
package does not contain enough data to deal with these groups. 
We are still unable to  complete the work for groups $B_n(q)$, $n=3,4,5$, and  $D_n^+(q)$, $n=4,5$.
If $n>5$, as well as for the groups $E_6(q),E_7(q),E_8(q)$, 
we do have the result stated in Theorem \ref{mt1}  due to the fact that we can use the result for $SL(n,q)$ for
induction purpose. 

For computer computations we use the program package \textsc{chevie}, created by a research group at University of Aachen. This
 is better adapted to deal with groups with connected center such as $GL(n,q)$ or $CSp(2n,q)$, 
the conformal symplectic group. In particular, we elaborate in details the cases 
$GL(6,q)$ and $CSp(6,q)$.
We expect that every $p$-vanishing character of a simple group $G$ of Lie type  in defining characteristic $p$
is the restriction to $G$ of a $p$-vanishing character of a suitable group with connected center. Unfortunately, 
we have not been able to prove this. Instead, we prove 
that every $p$-vanishing character of
$G$  is the restriction to $G$ of a $Syl_p$-vanishing character of a  group $H$ with
connected center (Proposition 
\ref{lu6}). 
A $Syl_p$-vanishing character of a finite group $G$ means a character vanishing at all
non-identity 
elements of a \syl of $G$.
This forces us
to pay significant attention to determining $Syl_p$-vanishing characters of degree $|G|_p$ for groups $G$ 
of small rank. In fact,
we determine all $Syl_p$-vanishing characters of degree $|H|_p$ for quasi-simple groups
$H$ of BN-pair rank $1$ and $2$ 
and the groups $H=CSp(6,q)$, $q$ odd, $GL(n,q)$, $n=4,5,6$, $U(n,q)$, $n=3,4,5$, 
and use certain elements of $H$ to rule out the characters that are not $p$-vanishing.
This gives us the list
of the $p$-vanishing characters of degree $|H|_p$ for $H$.
Our results on $Syl_p$-vanishing characters are auxiliary for the purpose of this paper, however, we expect that they will play some role in general theory.

In order  to perform the inductive step, we have to convert  information on $\tilde \chi_L$
for various Levi subgroups $L$ to a certain conclusion on $\chi$. Apart from   
Harish-Chandra induction
(see  \cite{CR2}), 
we use some basic theory of general
Gelfand-Graev characters.  We define a (general)
{\it Gelfand-Graev character}  to be $\nu^G$, where $\nu$ is a linear character of
a \syl of $G$ for $p$ being the defining characteristic of $G$. (The traditional definition 
of a Gelfand-Graev character (see \cite[14.29]{DM} or \cite[p. 254]{Ca})
 requires a certain non-degeneracy condition on $\nu$, which is dropped in our definition. Gelfand-Graev characters 
satisfying the non-degeneracy  condition are called here {\it non-degenerate} as in \cite{Ko}.)
Then we have $(\chi, \nu^G)=\chi(1)/|G|_p$
for every $p$-vanishing character $\chi$ (Lemma \ref{b1}). This allows us to take some control
of the \ir constituents  of $\chi$ that are common with $\nu^G$ for some $\nu$. 

For groups other than $SL(n,q)$ there are \ir characters that do not belong to any
$\nu^G$. We are unable  to control such constituents of $\chi$, and this is a source of
certain difficulties. Strictly speaking, our non-computational results only concern  with  the
"Gelfand-Graev part`` $\gamma(\chi) $ of $\chi,$ where $\gamma(\chi) $ is the sum of all \ir 
constituents of $\chi$ that occur in  $\nu^G$ for some $\nu$.

The formula   $(\chi, \nu^G)=\chi(1)/|G|_p$ allows us to bound
the multiplicities of the constituents common for $\chi$ and $\nu^G$ but does not help in identifying
them. In particular, if $G$ is with connected center of BN-pair rank $2$, and if $\chi$ is $p$-vanishing character of degree $|G|_p$
then the number of the constituents of $\chi$ in question is at most 4, however, even in these cases we cannot avoid
computations with explicit character tables. For projective module characters of degree $|G|_p$ this 
has been done manually, see \cite{MW} and \cite{Z}. For some groups  this can be done manually 
in our situation too, however, some groups of rank 2 such as ${}^2F_4(q^2)$ seem to require unbearable volume of
manual work. 

The induction step would  be easy if some Levi subgroup
of $G$ had no $p$-vanishing character of degree $|G|_p$. However,
this never happens, as the defect zero \ir characters of $L$ (in
particular, the Steinberg one) are $p$-vanishing of degree $|L|_p$.
In some cases of interest for every maximal Levi subgroup
$L$ every $p$-vanishing character of degree $|G|_p$ is of defect
$0$. This leads to the natural question:  is  a
$p$-vanishing character  \ir if the Harish-Chandra restriction
of it to every Levi subgroup is a sum of defect $0$ \ir characters?
In the case of projective indecomposable modules the
Brauer-Nesbitt correspondence  reduces  this to a similar question
for \ir modules, which are easy to handle (see \cite{Z}). In our case no
analog of the Brauer-Nesbitt correspondence is available. This
explains why one cannot mimic any machinery from the theory of
projective modules to deal with $p$-vanishing characters.

Now we continue our discussion of groups of small rank. As mentioned above, the task is to determine
all $Syl_p$-vanishing characters of degree $|G|_p$.
Technical difficulties are forced us to look for tools for 
simplifying computational work. We have found an advantage to 
determine first the characters $\chi$ of $G$ such that $\chi|_U$
is the  character of the regular representation of $U$, where $U$ is a \syl of $G$.
For groups of small rank the number of such characters is not large,
and we are able to check which of them are $p$-vanishing.
The matter is that the \textsc{chevie}  format of the character table is better adapted to
deal with $Syl_p$-vanishing characters, as \textsc{chevie}  organizes them in sets
such that all characters in the same set have the same restriction to $U$.
In this way we found another  phenomenon of  general interest. 
Specifically, we discover a considerable amount of formulas of type $\chi_1|_U=\chi_2|_U+ \chi_3|_U$, 
where $\chi_i$ $(i=1,2,3)$ are \ir characters of $G$. 
Therefore, if $\chi_1$ occurs in a decomposition of some $Syl_p$-vanishing character
in terms of \ir characters, then the substitution $\chi_2+ \chi_3$ for $\chi_1$
yields another decomposition, in which the degrees of $\chi_2$, $ \chi_3$ are less
than that of $\chi_1$. We did not try in this paper to obtain all formulas of this kind, 
but those we have got are collected in the tables at the end of the paper. 
Partially, we expose them  as tools helping  the reader to control 
our computations. Another purpose is to provide experimental material for theoretical analysis.  

Note that some formulas provided in the tables  involve more than three 
terms (see for instance Table \ref{T2B2}-A for $G={}^2B_2(q^2)$). 

\begin{notation}
${\mathbb C}$ is the complex number field, ${\mathbb Z}$ is the ring
of integers and $F_q$ is the finite field of $q$ elements. For a positive integer $m$
we define ${\mathbb Z}_m={\mathbb Z}/m{\mathbb Z}$.  We often view ${\mathbb Z}_m$ as the set of integers 
$\{1\ld m\}$. 

For a  finite group  $G$ we denote by $G'$, $Z(G)$ and $|G|$ the derived subgroup,  the center and  the
order of $G$, respectively. If $p$ is a prime then $|G|_p$ is the $p$-part of
$|G|$ and also the order of every \syl of $G$. A $p'$-group means
a finite group with no element of order $p$. If $g\in G$ then
$|g|$ is the order of $g$. All \reps and modules are over
${\mathbb C}$
(unless otherwise stated). 
Therefore, we take liberty to use the term `$G$-module'.  All
modules are assumed to be finitely generated.   The character of the regular \rep
of $G$ is denoted by $\rho_G^{\rm reg}$ and the trivial one-dimensional
\rep is denoted by $1_G$. We also use $1_G$ to denote the trivial
one-dimensional module and its  character. The inner product of 
class functions $\eta,\eta'$ is denoted by $(\eta,\eta')$.

The set of \ir characters of $G$ is denoted by $\Irr G$, and any integral linear
combination of
elements of  $ \Irr G$
 is called a generalized  character of $G$. 

Let $H$ be a subgroup of $G$ and $N$  a normal subgroup of $H$. Let $M$ be 
an $H$-module over a field.
Then $C_M(N)$ is an $H$-module, and when it is viewed as
$H/N$-module, it is denoted  by $\overline{M}_{H/N}$ (or
$\overline{M}$), and called the Harish-Chandra restriction of $M$
to $H/N$. Conversely, given an $H/N$-module $D$, one can view it
as an $H$-module with trivial action of $N$. Then the induced
$G$-module $D^G$ (when $D$ is viewed as an $H$-module) is denoted
by $D^{\# G}$ and  called Harish-Chandra induced from $D$. For
details see  \cite[p. 667, \S 70A]{CR2}, where these operations
are called generalized restriction and induction. They correspond
to similar operations on characters (or Brauer characters). The Harish-Chandra
restriction and induction extend by linearity to class functions
on $G$ with values in $\CC$. So if $\chi$ is a class function on $G$ then
$\overline{\chi}_{H/N}$ is the corresponding Harish-Chandra
restriction of $\chi$ to  $H/N$, and if $\lam$ is a class function
on $H/N$ then $\lam^{\# G}$ denotes the Harish-Chandra induced
class function on $G$. Let $\eta$ be a class function on $G$.  The
formula $(\lam^{\# G},\eta)=(\lam , \overline{\chi}_{H/N})$ is an
easy consequence of the Frobenius reciprocity and called the
Harish-Chandra reciprocity. (This is the formula 70.1(iii) in
\cite[p. 668]{CR2}.) Given a class function $f$ on $G$ define a
function on $H$ by $h\ra \frac{1}{|N|}\sum_{n\in N}f(hn)$ $(h\in  H)$.
This function is constant on the cosets $hN$ and therefore defines
a function on $H/N$ called the {\it truncation of $f$ to} $H/N$
with respect to $N$ (see \cite[70.4]{CR2} or \cite[pp. 262 - 283]{Ca}). 
Sometimes we also call the truncation the function $h\ra
\frac{1}{|N|}\sum_{n\in N}f(hn)$ on $H$ (which hopefully does not lead
to a confusion). 
Note that the truncation and the Harish-Chandra restriction of a character
coincide.

Let ${\mathbf H}$ be a reductive algebraic group in defining characteristic $p>0$. An algebraic group endomorphism  
$Fr:{\mathbf H}\ra {\mathbf H}$ is called {\it Frobenius} if 
the the subgroup ${\mathbf H}^{Fr}=\{h\in {\mathbf H}:Fr(h)=h\}$ is finite.
This subgroup is called a {\it finite reductive group}, and $p$ is referred as 
the defining characteristic of ${\mathbf H}^{Fr}$ too. Every finite reductive group  
is a group with BN-pair, and our use of terms `Borel, parabolic  and Levi subgroup, Weyl group' is as 
in the theory of groups with BN-pair \cite[\S 65]{CR2}. More detailed notation will be introduced later.   

If ${\mathbf H}$ is a simple algebraic group of universal type 
then we refer to 
${\mathbf H}^{Fr}$ as a {\it Chevalley group}. (Some authors  additionally assume that  $Fr$ acts 
trivially on the Weyl group of ${\mathbf H}$, but we prefer to refer to such groups as 
non-twisted Chevalley groups.)  

Let $H={\mathbf H}^{Fr}$ be a finite reductive group. We use some results on the Curtis
duality (called also the Curtis-Alvis duality) for the class
functions on $H$, see \cite[\S 71A]{CR2}. If  $\phi$ is a class function on $H$, we denote
by $D(\phi)$ the Curtis dual of $\phi$.  If $D(\phi)=\pm \phi$ then we call $\phi$ Curtis self-dual. 

Our notation for simple groups of Lie type are as in \cite{Ca}, however, for classical groups we also 
use the traditional notation such as $PSU(n,q)$ for the unitary group. Similarly, for quasi-simple groups and 
classical groups with connected center. For instance, $CSp(2n,q)$ denotes the conformal symplectic group.  

\end{notation}

\section{$p$-vanishing  characters and the
Harish-Chandra restriction}

We start with few observations valid for every finite group.

\subsection{$p$-vanishing and $Syl_p$-vanishing characters}
Let $\phi$ be a class function on a finite group $G$, and 
let $p$ a prime dividing $|G|$. We say that $\phi$ is
$p$-{\it vanishing} if $\phi(g)=0$ whenever $p$ divides $|g|$
$(g\in G)$. This also applies to characters and to generalized
characters. For instance, the character of the  regular module is
$p$-vanishing for every prime $p$. The characters of projective
$G$-modules over a field of characteristic $p$ are $p$-vanishing. It is well
known that every $p$-vanishing function is a linear combination of
the characters of projective indecomposable modules, and every $p$-vanishing generalized character
is an integer linear combination of them, see  \cite[Ch. IV,
Corollary 3.13]{Fe} or \cite[Theorem 18.26]{CR1}.
Furthermore, we say that $\phi$ is $Syl_p$-{\it vanishing}  if
$\phi(g)=0$ whenever $g\neq 1$ belongs to a \syl of $G$. Obviously,
$p$-{\it vanishing} class functions are $Syl_p$-vanishing, but the
converse is not always true. However, for \ir characters, the converse is
true (well known). (If $G=PSL(3,q)$, $q=p^k$, then there are
$Syl_p$-vanishing reducible characters
of degree $|G|_p$ that are not $p$-vanishing.)

If $\chi$ is a $Syl_p$-vanishing proper or generalized character
of $G$ then $\chi(1)$ is a multiple of $|G|_p$, and we set
$c_\chi=\chi(1)/|G|_p$.  
This notation allows us to express certain information in a
more compact form.

A proper character $\chi$ is $Syl_p$-vanishing \ii the restriction of $\chi$
to a \syl $U$ of $G$ is the character of a free $U$-module.

A proper $p$-vanishing (resp. $Syl_p$-vanishing) character $\chi$
is called {\it minimal} if whenever $\chi=\chi_1+\chi_2$, where
$\chi_1$, $\chi_2$ are proper characters of $G$, then neither
$\chi_1$ nor $\chi_2$ is $p$-vanishing (respectively, $Syl_p$-vanishing).

Obviously, \ir $p$-vanishing and $Syl_p$-vanishing characters are
minimal and have defect 0. The converse is not true, for instance,
every group $SL(2,q)$, $q=p^k$, has a reducible $p$-vanishing
character of degree $q$. Thus, minimal $p$-vanishing characters
are analogous to \ir characters of defect $0$.
Other examples of minimal $p$-vanishing and $Syl_p$-vanishing
characters $\chi$ are those with $c_\chi=1$.  
Note that $c_\chi=1$ means that $\chi|_U=\rho_U^{reg}$, where $U$
 is a \syl of $G$.

\begin{lemma}\label{bb1}
Let $G$ be a finite group and let $\chi$ be a minimal
$p$-vanishing character of $G$. Then all \ir constituents of
$\chi$ are in the same block.
 \end{lemma}

\begin{proof}
This follows from \cite[Ch.IV, Lemma 3.14]{Fe}. 
\end{proof}

Lemma \ref{bb1} allows one to assign a $p$-block to every minimal $p$-vanishing character.
Note that the lemma is not true for $Syl_p$-vanishing characters, even for $SL(2,q)$.

\begin{corol}\label{bb2} 
Let $G$ be a perfect finite 
group, and let $\chi$ be a minimal
$p$-vanishing character of $G$.
Let $z\in Z(G)$. Suppose that
$(\chi(1),|z|)=1$.
 Then $z$ belongs to  the kernel of $\chi$.
\end{corol}

\begin{proof}
Let $\phi$ be a \rep of $G$ with character $\chi$. 
It follows from Lemma
\ref{bb1} that $\phi (z)=\zeta(z)\cdot \Id$ for some $\zeta\in\Irr
Z(G)$.  Then $1=\det\phi(z)=\zeta(z)^n$, where $n=\chi(1)$.
This implies $\zeta(z) =1$.  
\end{proof}
 
\begin{lemma}\label{x1}
Let $G$ be a finite group with normal
subgroup $N$, and let $M$ be a free $G$-module of rank $r$. Then
$C_M(N)$ is a free $G/N$-module of rank $r$.

Consequently, if $\chi$ is a 
$Syl_p$-vanishing character then
$\overline{\chi}_{G/N}$ is
$Syl_p$-vanishing and $c_\chi=c_{\overline{\chi}_{G/N}}$.
\end{lemma}

\begin{proof}
The first claim  is obvious for $r=1$, and easily implies
the statement. To get the second claim, apply the first one to a
\syl $U$ of $G$ in place of $G$ and to $N\cap U$ in place of $N$.
\end{proof}

\begin{lemma}\label{b1}
Let $G$ be a finite group and
$U$ a \syl of $G$. Let $\eta$ be a $Syl_p$-vanishing
class function on $G$. Let $\tau$ be an \ir character of $U$. 
Then $(\eta, \tau^G)=\eta(1)\tau(1)/|U|$, and hence $c_\eta\cdot \tau(1)
=(\eta, \tau^G).$ 
\end{lemma}

\begin{proof}
We have $(\eta,
\tau^G)=(\eta|_U,\tau)= \eta(1)\tau(1)/ |U|$.  
\end{proof}

\subsection{The Harish-Chandra restriction}

Our standard assumption here is that $G$ is a finite group, $U$ a \syl
of $G$, $P\subset G$ is a subgroup containing $U$ and $L=P/O_p(P)$. (Later on, we specify $G$
to be  a
finite Chevalley group, $P$  a parabolic subgroup of $G$ and $L$ a Levi subgroup of $P$.)
Recall that for a character (or a class function) $\chi$ of $G$ we denote by
 $\overline{\chi}_{L}$ the truncation of $\chi$ to $L$.

The \f easy lemma is used in many situations.
\begin{lemma}\label{x2a}
Let $G,P,L$ be as above. 
\begin{itemize}
\item[(1)] Let $\chi$ be a $p$-vanishing  (resp., $Syl_p$-vanishing) character of
$G$. Then $\overline{\chi}_{L}$ is a $p$-vanishing
(resp., $Syl_p$-vanishing) character of $L$ and
$c_\chi=c_{\overline{\chi}_{L}}$. (In particular,  $\overline{\chi}_{L}\neq 0$.)
\item[(2)] Let $\eta$ be an \ir constituent of  $\overline{\chi}_{L}$.
Then $(\overline{\chi}_{L},\eta)=(\chi,\eta^{\# G})>0$.
\end{itemize}
\end{lemma}

\begin{proof}
(1) The statement on $Syl_p$-vanishing characters and the equality
follows from  Lemma \ref{x1} applied to $P$. Assume that $\chi$ is $p$-vanishing.
  Let $M$ be a $\CC P$-module with character $\chi$, and let
$R$ be the fixed point subspace of $U$ on $M.$ Note that $r\in R$
\ii $r=\frac{1}{|U|}\sum_{u\in U} um$ for some $m\in M$. Let $g\in
P$, $u\in U$. Suppose that the projection of  $g$, and hence of
$gu$, into $L$ is not a $p'$-element. 
 (Note that the projections of $gu$ and $g$ into $L$
coincide.) \itf $\overline{\chi}_{L}(g)=\frac{1}{|U|}\sum_{u\in
U}\chi(gu)=0$ by assumption, whence the first claim.

(2) follows from Harish-Chandra reciprocity.
\end{proof}

\begin{lemma}\label{er1} 
Let $G$ be a finite group and $U$ a \syl
of $G$. Let $\chi$ be a $Syl_p$-vanishing character of $G$. The \f are
equivalent:
\begin{itemize}
\item[(a)] $\chi(1)=|G|_p=|U|$;
\item[(b)] $\overline{\chi}_{N_G(U)/U}(1)=1$. In other words, the
truncation of $\chi$ over $U$ is a linear character of $N_G(U)/U$.
\end{itemize}
\end{lemma}

\section{Gelfand-Graev characters}

In this section
 $H$ is a finite reductive group in defining characteristic $p$, and $U$ a
\syl of $H$. Let $\lam$ be a character of $U$ with $\lam (1)=1$. Then
we call $\lam^H$ a {\it (general) Gelfand-Graev character}
 of $H$. The original Gelfand-Graev characters are called {\it non-degenerate} in
\cite{Ko}, and we use this term to specify the original Gelfand-Graev characters within
the set of (general) Gelfand-Graev characters. (If $H$ is not with connected center,  
there can be several non-degenerate Gelfand-Graev characters \cite{DM}.) 
Gelfand-Graev characters play a significant role in this paper.
The cause of this is partially explained by the
 \f lemma which is in fact is a special case of Lemma \ref{b1}.

\begin{lemma} \label{gg1}
Let $\chi$ be a $Syl_p$-vanishing character of $H$ of degree $|H|_p$.
 Then any  Gelfand-Graev character
has   exactly one common \ir
constituent with $\chi$, and this occurs with \mult $1$ in each of them.
In particular, $St$ is a constituent of every  Gelfand-Graev character.
\end{lemma}

Let $\chi$ be a $Syl_p$-vanishing character. Then we can write
$\chi=\gamma (\chi)+\gamma'(\chi)$, where $\gamma (\chi)$ is the
sum of all \ir constituents of $\chi$ that occur in some
 Gelfand-Graev character, and $\gamma'(\chi)$ is the
sum of the remaining terms. Clearly, $\gamma(\chi)$ is the sum of 
the \ir constituents $\tau$ of $\chi$ such that $\tau|_U$
contains a linear character of $U$. 

\begin{corol} The number of \ir constituents of $\gamma(\chi)$
does not exceed the number of distinct  Gelfand-Graev
characters.
\end{corol}

\begin{lemma}\label{k03}
Let  $\chi$ be a $Syl_p$-vanishing character of $H$ of degree
$|H|_p$.
 Then $\gamma(\chi)$ is \mult free.
\end{lemma}

\begin{proof}
This follows from the equality $\chi|_U=\rho_U^{reg}$, and
 is also an immediate consequence of Lemma \ref{gg1}.
\end{proof}

\begin{remar}\label{re0}
The character $\gamma'$ is not always \mult free. For instance,
for $H=PSU(3,5)$, in notation of Atlas \cite{Atl}, the character
$\chi_1+2\chi_2+\chi_4+\chi_5+\chi_6$ is $5$-vanishing 
of degree $|H|_5=125$.
\end{remar}

The  Gelfand-Graev characters can be characterized as follows:

\begin{lemma} [{\cite[Lemma 2.3]{Ko}}] \label{k00}
Suppose that $H$ is with connected center, with no component of type
   $ B_n(2), C_n(2)$,  $F_4(2)$, $G_2(2)$,
 $G_2(3), {}^2B_2(2), {}^2F_4(2),
 {}^2G_2(3)$.  Let $\lam$ be a linear character of  $U$.
 Then $\lam^H=\Gamma_L^{\#H}$, where $L$ is some Levi subgroup of $H$ and
$\Gamma_L$ is the non-degenerate Gelfand-Graev character of $L$. 
Conversely, if $\Gamma_L$ is the non-degenerate Gelfand-Graev character of $L$, then $\Gamma_L^{\#H}=\lam^H$ 
for some linear character $\lam$ of $U$.
\end{lemma}

\begin{remar}
The exclusions in the above lemma correspond to the cases 
where $U\neq U_1$ in notation of \cite[p. 350]{Ko}.
The statement in \cite{Ko} does not exclude the groups $H$ with components of type
${}^2B_2(2), {}^2F_4(2),
 {}^2G_2(3)$, however, we think that this is required.
\end{remar}

Let $\mu\in\Irr H$ and let $L$ be a Levi subgroup of $H$.  Suppose that
 $\mu$ is not a constituent of any Gelfand-Graev character of 
$H$. Then all \ir constituents of $\overline{\mu}_L$ belong to $\gamma'
(\overline{\chi}_L)$. Indeed, suppose the contrary, and let $\lam$ be an \ir constituent
of $\overline{\mu}_L$ belonging to $\gamma (\overline{\chi}_L)$.
This means that $\lam$ is a constituent of $\Gamma_K^{\#L}$ for some Levi subgroup $K$ of $L$
(and hence of $H$) and a non-degenerate Gelfand-Graev character $\Gamma_K$ of $K$. 
Note that  $(\mu, \lam^{\#H})=(\overline{\mu}_L,\lam)>0$ by Harish-Chandra reciprocity. 
Therefore, $\mu$ is a constituent of $\Gamma_K^{\#H}$ by transitivity of Harish-Chandra
induction. Hence, by Lemma \ref{k00}, we get a contradiction. 

Note that a  regular character $\si$ is cuspidal \ii $\si$  is semisimple.
Indeed,   
consider the Curtis dual $D(\si)$ of $\si$. As $\si$ is
regular, $\ep D(\si)$ is semisimple for some $\ep=\pm 1$
\cite[14.39]{DM}. However,  $\si$ is cuspidal \ii $ D(\si)=\pm \si$ \cite[p. 690]{CR2}, 
so the claim follows.

\begin{lemma}\label{ko3}
Let $H$ be as in Lemma $\ref{k00}$.  
 Then the number of distinct Gelfand-Graev 
characters equals the number of non-conjugate Levi subgroups of
$H$.
\end{lemma}

\begin{proof}
By Lemma \ref{k00}, every  Gelfand-Graev
character is of shape $\Gamma_L^{\# H}$, where $\Gamma_L$ is the
non-degenerate Gelfand-Graev character of $L$. (Here $L$ can
coincide with $H$.) Since $H$ is with connected center then so is $L$ (see
\cite[8.1.4]{Ca}),  and hence $L$ has a
unique non-degenerate Gelfand-Graev character. Clearly, if  Levi
subgroups $L,M$ of $H$ are conjugate then $\Gamma_L^{\# H}$
coincides with $\Gamma_M^{\# H}$.   (Note that $\Gamma_L^{\# H}$ does not depend on the choice of the parabolic subgroup  which defines $L$ as a Levi subgroup \cite[70.10]{CR2}.)   
Suppose that $L,M$ are not
conjugate. It is known that each  $L,M$ has a cuspidal regular character
$\eta_L, $ $\eta_M, $ say. 
(Indeed, if a maximal torus $T^*$ of $H^*$ contains an element $s$ in general position 
then the regular character $\rho_s$ corresponding to $s$ is cuspidal \ii 
$T^*$  contained in no proper Levi subgroup of $H^*$ (see \cite[9.3.2]{Ca}). 
It is shown by Veldkamp \cite[p. 391]{Ve} that a so called Coxeter torus $T^*$ is 
not contained in any proper Levi subgroup of $H^*$ and contains an element in general
position; 
it is also known that an element in general position is  contained in a unique maximal 
torus \cite[p. 390]{Ve}.) 
Then $\eta_L^{\# H}, $ $\eta_M^{\# H} $
are disjoint, and their \ir constituents are contained in $\Gamma_L^{\#
H}$, $\Gamma_M^{\# H}$, respectively. So 
$\Gamma_L^{\# H}\neq \Gamma_M^{\#
H}$, and the result follows. 
\end{proof}

\begin{remar}
One observes that this result is not valid for $H={}^2B_2(2), {}^2G_2(3), {}^2F_4(2)$. 
\end{remar}

The \f result is contained in  \cite[Propositions 2.4, 2.5, Corollary 2.6]{Ko}. 

\begin{lemma}\label{ko1} 
Let $H$ be as in Lemma $\ref{k00}$,
and $\tau\in\Irr H$. Then the \f are equivalent:
\begin{itemize}
\item[(1)] $\tau$ is a constituent of some Gelfand-Graev character of $H$;
\item[(2)] $\tau$ is a constituent of $\lam^{\# H}$, where $\lam$ is a
cuspidal regular character of a Levi subgroup of some parabolic
subgroup of $H$.
\end{itemize}
\end{lemma}

Observe that the character $\lam$ in Lemma \ref{ko1} is both regular and semisimple
because every cuspidal 
regular character is semisimple (see comments prior Lemma \ref{ko3}). 

Distinct  Gelfand-Graev characters  have common
\ir constituents.
Recall that 
a {\it Harish-Chandra series of} $H$ is the set of \ir
constituents of a Harish-Chandra induced character $\lam^{\#
H}$, where $L$ is a Levi subgroup of a parabolic subgroup $P$ of
$H$ (possibly $P=H$) and $\lam $ is a cuspidal character of $L$.
Two Harish-Chandra series either coincide or are disjoint, see
\cite[70.15A]{CR2}. Therefore, if $\lam $ is cuspidal then
the \ir constituents of $\lam^{\# H}$  do not
occur in the other  Gelfand-Graev characters.

In addition, if $\si$ is a regular cuspidal character of $L$
then the \ir constituents of $\sigma^{\#
H}$ are in bijection with $\Irr W_{\si}$, where $W_{\si}$
is a certain subgroup of $W$ generated by reflections. More precisely,
$W$ acts on the roots of the corresponding algebraic group, and $L$ is defined by a subset
$J$ of simple roots.
Then $W=\{w\in W: w(J)=J$ and $w(\si)=\si\}$.

\begin{propo}\label{d34}
Assume $H$ to be with connected center and  has no  component of type
$B_n(2)$, $C_n(2)$, ${}^2B_2(2), $ $F_4(2)$, ${}^2F_4(2)$, $G_2(2)$,
 $G_2(3) $ and ${}^2G_2(3) $.  Let $\chi$ be a $Syl_p$-vanishing character 
of $H$ of degree $|H|_p$. 
\begin{itemize}
\item[(1)] For every \ir constituent of $\tau$ of 
$\gamma(\chi)$ there are a unique Levi subgroup $L$ of $H$ and 
a  unique cuspidal regular character $\rho$ of $L$ such that 
$(\tau,\rho^{\#H})>0$. Moreover,
$(\tau,\rho^{\#H})=1$.
\item[(2)]  Let 
$l(\chi)$ be the number of non-conjugate Levi subgroups of $H$ such that 
$\overline{\chi}_L$ contains a cuspidal regular character of $L$.
Then the number of \ir constituents of  $\gamma(\chi)$ equals $l(\chi)$.
\end{itemize}
\end{propo}

\begin{proof}
Let $\{\nu_1\ld \nu_k\}$ be a maximal set of linear characters of $U$ such that
$\nu_i^H\neq \nu_j^H$ for $i\neq j$.
By Lemma \ref{k00}, $\nu_i^H=\Gamma_L^{\#H}$ for some Levi subgroup $L$, and by Lemma
\ref{ko3}, the set $\{\nu_1\ld \nu_k\}$ is in bijection with the set of non-conjugate Levi
subgroups.
Let $L_i$ correspond to $\nu_i$ for $i=1\ld k$. 
As $(\chi,\nu_i^H)=1$ (Lemma \ref{gg1}), we may denote by $\chi_i$ the constituent of 
 $\chi$ common with $\nu_i^H$. (So $\chi_i$ is a constituent of
 $\gamma(\chi).$) The mapping $i\ra \chi_i$ is surjective, but 
not injective in general.

To see when $\chi_i=\chi_j$, fix some $j$. Then  $\chi_j$ is a constituent of 
$\Gamma_{L_j}^{\#H}$, so there is a regular character $\rho_j$ of $L_j$
such that $(\chi, \rho_j^{\#H})=1$. Moreover, $\rho_j$ is unique. Indeed,
as $\overline{\chi}_{L_j} (1)=|L_j|_p$ and $\overline{\chi}_{L_j}$ is $Syl_p$-vanishing,
we have $(\overline{\chi}_{L_j}, \Gamma _{L_j})=1$, where  $ \Gamma _{L_j}$ is a non-degenerate Gelfand-Graev character of ${L_j}$. 
As $L_j$ is with connected center too \cite[8.1.4]{Ca}, $ \Gamma _{L_j}$ is unique, and hence $\rho_j$ is unique. 

Suppose that $\rho_j$ is not cuspidal.
Then $\rho_j$ is a constituent of $\lam_i^{\#L_j}$, where $L_i$  is some Levi 
contained in $L_j$ (up to conjugation) and $\lam_i$ is a cuspidal regular character 
of $L_i$ (Lemma \ref{ko1}). So $\chi_j$ is a constituent of $\nu_i^H$ too, and
hence  
we have $\chi_i=\chi_j$ (as $(\chi,\nu_i^H)=1$). Thus we need not count $\nu_j$ such that 
$\overline{\chi}_{L_j}$ contains no cuspidal regular character of $L_j$, and the claim
follows.
\end{proof}

\begin{lemma}\label{rr1}
Let $H$ be as in Proposition $\ref{d34}$ 
and
$\chi$ a $Syl_p$-vanishing character of $H$. 
Suppose that $\gamma'(\overline{\chi}_L)=0$ 
for every proper Levi subgroup $L$ of $H$. 
Then every \ir constituent of $\gamma'(\chi)$
is a cuspidal character.
\end{lemma}

\begin{proof}
Suppose the contrary. Let $\si$ be an \ir constituent 
of $\gamma'(\chi)$, which is not cuspidal. Then there exists a Levi subgroup $L$ and a cuspidal character $\mu\in \Irr L$ such that
$(\si,\mu^{\# H})>0.$ Therefore, $(\overline{\si}_L,
\mu)=(\si,\mu^{\# H})>0$. So $\mu$ is an \ir constituent 
of $\overline{\si}_L$, and hence of 
$\overline{\chi}_L$. By assumption, $\gamma'(\overline{\chi}_L)=0$. 
This means that $\mu$ is a 
constituent of a  Gelfand-Graev character of $L$. By Lemma \ref{ko1}, $\mu$ is a regular character of $L$
(as $\mu$ is cuspidal). 
Therefore, $\mu$ is a constituent of $\Gamma_L$, the non-degenerate 
Gelfand-Graev character of $L$.
By Lemma \ref{ko3}, $\Gamma_L^{\#H}$ is a  Gelfand-Graev character of $H$. As $\si$ is a
constituent of  $\mu^{\#H}$,
it is also a constituent of $\Gamma_L^{\#H}$, that is, $\si$ 
is a summand of $\gamma(\chi)$, which is a contradiction. 
\end{proof}

\begin{lemma}\label{ur2}
Let $H$ be a finite reductive group and let $\tau,\si$ be distinct \ir constituents
of $\gamma(\chi)$
for a $Syl_p$-vanishing character $\chi$ of degree $|H|_p$.
Then $\tau|_U $ and $ \si|_U$ have no common constituent of degree $1$,
in particular, $\tau|_U\neq \si|_U$.
\end{lemma}

\begin{proof}
Indeed, suppose the contrary. Then  $(\tau|_U, \nu)>0$ and $( \si|_U,\nu )>0 $
for some  linear character $\nu$ of $U$. As $\chi|_U=\rho_U^{reg}$, we have
 $( \chi|_U,\nu )=1 $, and the claim follows.
 \end{proof}

\begin{lemma} \label{g28}
Let $H=GL(n,q)$ or $SL(n,q)$. Then every \ir character of $H$ is a constituent of a 
 Gelfand-Graev character.
\end{lemma}

\begin{proof}
For $H=GL(n,q)$ the result is stated in the original paper by Gelfand and Graev \cite{GG},
and this implies the one for $G=H'=SL(n,q)$. Indeed, let $U$ be a 
\syl of $G$.  Let $\lam\in\Irr U$ with $\lam(1)=1$.
By Mackey's lemma, $(\lambda^H)|_{G}=\sum_g (\lambda^g)^{G}$, where $g$ are the
representatives of the double 
cosets of $G$ and $U$ in $H$.
Since $H=G\cdot T$, where $T$ is a maximal torus of the Borel subgroup of $H$ containing
$U$ \cite[Lemma 24.12]{MT}, we may choose $g$ in $T$. Then $U$ is $g$-stable,  $\lam^g$ is
a linear character of $U$ and $(\lam^g)^{G}$ is a  Gelfand-Graev
character of $G$. So the result follows from that  for $GL(n,q)$, as every \ir 
character of $G$ is a constituent of the restriction to $G $ of an \ir character of
$H$.
\end{proof}

\begin{corol} \label{g12}
Let $H=GL_n(q)$ or $SL_n(q)$, and let $\chi$
be a $Syl_p$-vanishing character of $H$ of degree   $|H|_p $.
Then $\chi=\gamma(\chi)$,
and hence $\chi$ is \mult free.
\end{corol}

\begin{proof}
This follows from Lemmas \ref{g28} and \ref{k03}.
\end{proof}

\section{$Syl_p$-decomposable \ir characters}\label{Syldec}

\subsection{General observations}  Let $\tau$ be an \ir character of a finite group  $G$.  
Let $U$ be a \syl of $G$. We say that $\tau$ is $Syl_p$-decomposable
if there are irreducible characters $\chi_1\ld\chi_k$ of $G$ (not necessarily distinct)
such that $\tau(u)=\chi_1(u)+\cdots 
+\chi_k(u)$ for every element $u\in U$. (Equivalently, $\tau-\chi_1-\cdots 
-\chi_k$ vanishes on $U$.) We express this by writing 
$\tau\equiv \chi_1+\cdots 
+\chi_k\pmod U$, and call this expression a $Syl_p$-{\it decomposition} of $\tau$. 
Note that some characters of $G$
can be $Syl_p$-{\it equivalent}, that is, they take the same value at every $u\in U$.

The problem of determining the
$Syl_p$-vanishing characters of degree $|U|$ discussed above for finite reductive 
groups in defining characteristic $p$  can be restated as that of  finding all $Syl_p$-decompositions of the Steinberg character of $G$.
We find out  that the knowledge of $Syl_p$-decompositions for characters $\tau$ of degree less that $|U|$ is very helpful for determining the $Syl_p$-vanishing characters, 
seems to be unavoidable for most groups of rank greater than $1$. 
Morever,
we feel that they  are of certain independent interest. Because of these we perform
an extensive computations of them, and collect the relevant results 
in the tables at the end of the paper.

We collect here some general facts on $Syl_p$-decompositions
for finite reductive groups $H$.

Recall that, if  $\phi$ is a class function on $H$ then
 $D(\phi)$ denotes the Curtis dual of $\phi$.  
If $D(\phi)=\pm \phi$ then we call $\phi$ Curtis self-dual. 
The following well known fact can be easily deduced from the definitions 
of the duality operation \cite[71.2]{CR2} and the Steinberg character 
\cite[71.18]{CR2}.

\begin{lemma}\label{a8}
Let 
 $\phi$ be a class function on $H$ constant on $U$.
Then $D(\phi)(g)=St(g)\phi(g)$ for all $g\in H$.
In particular, if $\phi(u)=0$ for all $u\in U$ then
$D(\phi)(u)=0$ for all $u\in U$.
\end{lemma}

\begin{lemma}\label{aa1}
Suppose $H$ is with connected center.  Let $\chi_1,\chi_2,\chi_3\in \Irr H$. Suppose that  
 $\chi_1\equiv\chi_2+\chi_3\pmod U$ is a $Syl_p$-decomposition of $\chi_1$.
\begin{itemize}
\item[(1)] At least one of the characters $\chi_i$, $1\leq i\leq 3$, is not regular. 
\item[(2)]  At least one of the characters $\chi_i$, $1\leq i\leq 3$, is not semisimple.
\item[(3)] If every $\chi_i$ is either regular or semisimple then at least one of them is 
regular and semisimple (and hence Curtis self-dual).  
\item[(4)] $D(\chi_1)-D(\chi_2)-D(\chi_3)$ is a  $Syl_p$-vanishing  generalized character of degree $0$.
Let $\chi_i'=\pm D(\chi_i)$ be a proper character, $i=1,2,3$. Then  reordering of 
the $\chi'_i$ yields a $Syl_p$-decomposition (which may coincide with that in the 
assumption).
\end{itemize}
\end{lemma}

\begin{proof}
Note that $\chi_1-\chi_2-\chi_3$ is a $Syl_p$-vanishing generalized character of degree 0.
By Lemma \ref{a8}, so is $D(\chi_1-\chi_2-\chi_3)$, which implies (4).
 Let $\Gamma $ be the non-degenerate Gelfand-Graev character of $H$.  

(1) If all $\chi_i$ are regular, then, taking the
 inner product of $\chi_1-\chi_2-\chi_3$ with $\Gamma$, we get $1-1-1=0 $, which is absurd. 

(2) Suppose all of them are semisimple. 
As $\pm D(\chi_i)$ is a regular character for $i=1,2,3$, 
the inner product with  $\Gamma$ yields  
again $\pm 1\pm 1\pm 1=0 $, a contradiction. 

(3) Suppose that none of the $\chi_i$'s is 
regular and semisimple simultaneously. If one of the characters is regular 
and the other two are not regular, the inner product with  $\Gamma$ yields  a
contradiction. So there are two regular characters among them, and they are not semisimple. 
Then $D(\chi_1)-D(\chi_2)-D(\chi_3)$ is $Syl_p$-vanishing, and two characters 
among $ D(\chi_1), - D(\chi_2), - D(\chi_3)$ are semisimple and not regular, and one is
regular. Again, the inner product with  $\Gamma$ yields  a contradiction.
\end{proof}

\begin{lemma}\label{cc1}
Suppose $H$ is with connected center. 
Let $\chi$ be a $Syl_p$-vanishing character of $H$ of degree $|H|_p$, and let
 $\si$ be a 
  regular character of $H$ that occurs as a 
 constituent of $\gamma(\chi)$. Let $\tau=\sum m_i\tau_i$ be 
a reducible character such that $\si-\tau$ is $Syl_p$-vanishing of degree $0$. 
\begin{itemize}
\item[(1)] There is a unique $i$ such that $\tau_i $ is regular;
\item[(2)] Suppose $H$ has no component of type 
$ B_n(2)$, $C_n(2)$, ${}^2B_2(2)$, $F_4(2),{}^2F_4(2)$, $G_2(2),
G_2(3),{}^2G_2(3)$. 
If  $ \si$ is cuspidal then $\gamma(\tau)=\tau_i$, that is, 
$\tau_j$ does not belong to any  Gelfand-Graev character for $j\neq i$; 
in addition, if $H=GL(n,q)$ then $\tau$ is \ir and hence $Syl_p$-equivalent to
$\si;$ 
\item[(3)] $\gamma(\tau)$ is 
\mult free, that is, $m_i=1$ if $\tau_i$ is a constituent of some  Gelfand-Graev character. 
\end{itemize}
\end{lemma}

\begin{proof}
(1) follows by taking the inner product of $\si-\tau$ with the non-degenerate Gelfand-Graev 
character of $H$. 

(2) By the way of contradiction, suppose that $\tau_j$ for $j\neq i$ is a 
constituent of a 
Gelfand-Graev character 
 $\Gamma$ of $H$. By Lemma \ref{k00}, there is a Levi subgroup $L$  of $H$ 
such that
 $\Gamma=\Gamma_L^{\#H}$. 
Note that $L\neq H$,
as  $(\tau,\Gamma)= (\tau_i,\Gamma)=1$. Since  $(\tau_j, \Gamma_L^{\#H})>0$, we have
$(\si,\Gamma_L^{\#H})>0$, which is false as $\si$ is cuspidal. 

To prove the additional statement, recall that every \ir character 
of $H=GL(n,q)$ is a constituent of some Gelfand-Graev character (Lemma \ref{g28}). 
This implies $\gamma(\tau)=\tau $. 

(3) Let $\tau_i$ is a constituent of a Gelfand-Graev character $\nu^H$.
Suppose that $m_i>1$. Then $(\tau,\nu^H)\geq m_i>1$, and hence $(\si,\nu^H)> 1$. 
Then  $(\chi,\nu^H)>1$, which contradicts Lemma \ref{gg1}.
\end{proof}

\begin{lemma} 
Let  $\chi,\chi_1\ld \chi_k\in \Irr H$, and $\chi(1)< |H|_p$. Suppose that 
$\chi(g)=\chi_1(g)+\cdots+  \chi_k(g)$ for every $p$-singular element $g\in H$. 
  Then all characters $\chi,\chi_1\ld \chi_k $ 
belong to the same $p$-block.
\end{lemma}

\begin{proof}
Let $B$ be the block $\chi$ belongs to. By reordering the $\chi_i$'s we can 
assume that $\chi_1\ld \chi_m\in B$, and $\chi_{m+1}\ld \chi_k\notin B$. By Feit
\cite[Ch. IV, Lemma 3.14]{Fe}, it follows that $\si:=\chi_{m+1}+\cdots+ \chi_k $
is a $p$-vanishing character, and hence $\si(1)$ is a multiple of $|H|_p$. 
So $\si(1)=0$, and hence $\si=0$, as claimed. 
\end{proof}


\subsection{The computational aspect} 
In Sections \ref{rk1}, \ref{rk2} and \ref{rk3} we compute the $Syl_p$-vanishing characters 
$\psi$ of degree $|G|_p$ for some 
Chevalley groups $G$, whose character tables are known. 
Our main reference here is the algebra package \textsc{chevie} \cite{chevie}, in particular,
 its \textsc{maple} part. 
Recall that \textsc{chevie} 
consists of a library of data and programs to deal with the generic character table of certain finite groups $G$ of Lie type of small rank.
If $G$ is a such group, defined over the field $F_q$,  the generic character table of $G$ depends on this parameter $q$. 
So, all the information stored in \textsc{chevie}  is expressed in the form of  functions of $q$: 
for instance, the degrees and the number of the irreducible characters as well as the number and the representatives of the conjugacy 
classes of $G$. 

We keep 
the notation of the \textsc{chevie}  package for the irreducible characters and the conjugacy classes. 
In particular, 
\textsc{chevie} partitions   $\Irr G$ in sets, which we call {\it \textsc{chevie} 
sets} and denote by  ${\bf X}_h=\{\chi_h(k)\mid k \in I_h\}$ in this paper. Here 
$h$ is a natural number
and $I_h$ is some parameter set,
called the character parameter group 
 (see the \textsc{chevie} manual  \cite{chevie}).
Furthermore, \textsc{chevie} 
partitions the conjugacy classes of $G$ in sets ${\bf C}_h=\{C_h(k)\mid
k \in I_h\},$
where, as before, $I_h$ is the corresponding parameter set. Note the classes belonging
to the same set ${\bf C}_h$ have the same order.

To avoid misunderstanding, we emphasize that \textsc{chevie} notation for parameters
$k\in I_h$ for different groups is often not homogeneous, but we are forced to follow
\textsc{chevie}. For instance, the set  $I_4$ of $U(3,q)$  is parametrised by $(u,v)$, and
the set $I_{10}$ of $SU(3,q)$, with $3\mid (q+1)$, is parametrised by $(n,m)$.

It is essential for us that the characters 
in the same \textsc{chevie} set ${\bf X}_h$ are $Syl_p$-equivalent 
(but the converse is not true).
This allows us, in performing computations with $Syl_p$-vanishing characters,  to ignore 
the parameter $k$ that determines an individual character 
$\chi_h(k)\in {\bf X}_h$. By this reason we  write $\chi_h$
to denote any character of ${\bf X}_h$ (to simplify the notation).

Our strategy involves three steps, that we describe as follows.
 
\begin{proc}\label{procedure}
In order to determine the set of all  $Syl_p$-vanishing characters $\psi$ of
a group $G$ with $\psi(1)=|G|_p$, we
apply the following procedure:
\begin{itemize}
\item[(1)]  Compute
some $Syl_p$-decompositions of the irreducible characters $\chi_h$ of $G$ 
(the results are given in the tables of Section \ref{tables}).
 For $Syl_p$-decompositions we choose to compact the notation by writing 
$\mathbf{v}=(j_1,\cdots ,j_t)$, especially in the tables.
The case where $t=1$ means that $\chi_h$ and $\chi_{j_1}$ are $Syl_p$-equivalent.
(This is reflected in the tables, see, for instance Table \ref{TU3}.3-A.)
Note that we do not guarantee that our 
tables contain all $Syl_p$-decomposable \ir characters, this is not necessary for our
purposes. To simplify the notation, we denote by $\Delta_G$ the set of the indices $h$
corresponding to
the irreducible characters $\chi_h$ of $G$ that have  no known $Syl_p$-decomposition.

\item[(2)] Determine all $Syl_p$-vanishing characters of degree $|G|_p$ whose \ir constituents 
$\chi_j$ are chosen with $j\in \Delta_G$.
Note that by  Lemma \ref{gg1}, every character $\psi$ 
has at least one
regular character constituent (exactly one if $G$ has
a unique non-degenerate Gelfand-Graev character).
We use \textsc{chevie} to 
identify the regular characters of $G$ in our list of $Syl_p$-decompositions.
This
allows us to assume that every regular  constituent of $\psi$ has no known  
 $Syl_p$-decomposition.

\item[(3)] If $\sum \chi_i$ is a subcharacter of a character $\psi$ obtained in step (2)
and this subcharacter is also a $Syl_p$-decomposition of an irreducible character $\tau$
of $G$ (obtained in step (1)), then the substitution  with $\tau$ in place of
$\sum \chi_i$ yields a new $Syl_p$-vanishing character with a lower number of
constituents. Repeating this, we arrive at $Syl_p$-vanishing characters that does not
contain such subcharacters. In this way, we are able to list all the $Syl_p$-vanishing
characters of $G$ of degree $|G|_p$.
\end{itemize}
\end{proc}

Afterwards, we use the character table of $G$ in order to detect and rule out 
those $Syl_p$-vanishing characters that are not $p$-vanishing.
In most cases 
we do this for groups with connected center.
However, in the last step we try to only use 
elements that belong
to the derived subgroup   $G'$ of $G$. In view of Proposition \ref{lu6},
this yields the result for groups
 with disconnected center. 

 Note that we have found  
 quite a lot of $Syl_p$-decomposable characters,
 including regular and semisimple, and this  hints that the use of them
 can organise $Syl_p$-vanishing characters in a more conceptual way.

Also, observe that every character $\chi_h\in {\bf X}_h$ occurs with the same 
\mult in a Gelfand-Graev 
character $\nu^G$ for a fixed $\nu\in\Irr U$. So, by Lemma \ref{b1}, 
we have:

\begin{lemma}\label{mu0} 
Suppose that $G$ is with connected center. If some character
 $\chi_h\in {\bf X}_h$ is regular then so are all 
 characters in $ {\bf X}_h$.
In addition, if $\chi$ is a $Syl_p$-vanishing character of $G$ of 
degree $|G|_p$ then for any fixed $h$ there is at most one 
regular character from ${\bf X}_h$ 
occurring 
in $\chi$.   
\end{lemma}

In the next sections we often use the \f simple observation. Let $H$ be a finite
reductive group and $G$ the subgroup
generated by unipotent elements. Then $H/G$ is abelian. Suppose that $G$ is quasi-simple. Let $\chi$
be a $p$-vanishing character of $H$ of degree $|H|_p$. Then $\chi$
contains a one-dimensional constituent \ii $\chi|_G$ contains
$1_G$. Therefore, $\chi\cdot \nu$ contains $1_H$ for some one-dimensional character $\nu$ of $H$.

\section{Chevalley groups of BN-pair rank $1$}\label{rk1}

Now we start dealing with Chevalley groups of rank $1$: $SL(2,q)$,  $SU(3,q)$, ${}^2B_2(q^2)=
Sz(q^2)$  with $q^2 = 2^{2m+1}$, and ${}^2G_2(q^2)$ with $q^2= 3^{2m+1}$.
 
In this section, as well as in Sections \ref{rk2} and \ref{rk3}, we shall use the \f 
notation for certain complex numbers:
$$\left. 
\begin{array}{llll}
\omega=\exp(\frac{2\pi i}{3}),\; &
\zeta_1=\exp(\frac{2\pi i}{q-1}),\;&
\xi_1=\exp(\frac{2\pi i}{q+1}),\; &
\xi_2=\exp(\frac{2\pi i}{q^2+1}),\; \\ 
\xi_3=\exp(\frac{2\pi i}{q^3+1}),\; & 
\varphi_8'=\exp(\frac{2\pi i}{q^2+\sqrt{2}q+1}), \; &
\varphi_{12}''=\exp(\frac{2\pi i}{q^2+\sqrt{3}q+1}). \\
\end{array}\right.$$

\subsection{Groups $SL(2,q)$ and $GL(2,q)$}

One observes that all irreducible characters of $GL(2,q)$ of the same degree are
$Syl_p$-equivalent.
Inspection of the character table of $G=SL(2,q)$ yields:

\begin{lemma}\label{s21} 
Let $G=SL(2,q)$ and  $\psi$ be a reducible character of degree $q$. Then $\psi$ is $Syl_p$-vanishing \ii one
of the \f holds:
\begin{itemize}
\item[(1)] $\psi-1_G$ is an \ir character;
\item[(2)] $\psi-1_G=\eta_1+\eta_2$, where
$\eta_1,\eta_2$ are distinct characters of degree $(q-1)/2$;
\item[(3)] $\psi=\tau+\eta$, where $\tau(1)=(q+1)/2$ and  $\eta(1)=(q-1)/2$  and $\tau,\eta$ belong to distinct Weil \reps of $G$.
\end{itemize}
In addition, $\psi$ is $p$-vanishing \ii $\psi$ is
$Syl_p$-vanishing and $\psi(z)=\psi(1)$ for all $z\in Z(G)$.
\end{lemma}

\begin{remar}
(i) The character  $\psi-1_G$ in (1) and $\eta_1,\eta_2$ are  cuspidal.
(ii) The cases (2), (3) occur only for $q$ odd. (iii) The cases (3) 
cannot occur for $G=PSL(2,q)$. 
(iv) In (3) the characters $\tau,\eta$ are in distinct blocks. 
\end{remar}

\begin{lemma}\label{s31}
Let $H=GL(2,q)$. A reducible character $\psi$ of $H$ of degree $q$ is $Syl_p$-vanishing 
\ii $\psi=\lam+\eta$, where $\lam(1)=1$, and $\eta$ is an \ir cuspidal character of degree
$\eta(1)=q-1$.
In this case, $\psi$ is $p$-vanishing \ii ${\rm Ker}\,(\psi)=Z(H)$.
\end{lemma}

In \textsc{chevie} notation, $\psi=\chi_1(k_1)\cdot (1_H+\chi_4(k_4))$, where 
$\chi_1(k_1)$ $(k_1\in I_1)$ is a linear character of $H$ 
and $k_4\in I_4$ with $k_4=(q-1)k_4'$ (see Tables \ref{TGL2}-A,B,C,D).

\subsection{The unitary groups $H=U(3,q)$ and $G=SU(3,q)$.}

We first consider the groups $H=U(3,q)$.
According to \textsc{chevie}, we  
partition  $\Irr H$ into $8$ sets ${\bf X}_h$, 
and the characters in ${\bf X}_h$ are  parametrized by a corresponding set 
$I_h$ of parameters (see Table \ref{TU3}.1-A).
Step (1) of Procedure \ref{procedure} yields  the 
$Syl_p$-decompositions
of Table \ref{TU3}.1-B and  $\Delta_H=\{1,2,6\}$.
On  step (2) we obtain a unique character 
$\psi_1= \chi_{1} + \chi_{2} + \chi_{2}' + \chi_{6}$, where $\chi_2,\chi'_{2}\in {\bf
X}_{2}$ are not necessarily
distinct. 
Finally, in step (3) we get  all  $Syl_p$-vanishing characters of $H$ of degree $|H|_p$ 
(see Table \ref{TU3}.1-C).
To determine the $p$-vanishing characters of $H$ of degree $|H|_p$,
we use the character table of $H$.

\begin{lemma}\label{LU3}
Let $H=U(3,q)$. Then $H$ admits a reducible $p$-vanishing character 
$\psi$ of degree $|H|_p$ if and only if $3\mid (q+1)$. In this case, there are exactly
$q+1$ characters $\psi$, all of type
$$\psi=\chi_{1}(u)\cdot (1_H+ \chi_{2}(a) + \chi_{2}(2a) + \chi_{6}(a,2a,3a)),$$
where $a=(q+1)/3$ and $\chi_1(u)$ is a linear character of $H$.
\end{lemma}

\begin{proof}
Let $\psi$ be a $Syl_p$-vanishing character of $H$ of degree 
$|H|_p=q^3$ and suppose that $\psi$ is also $p$-vanishing.
By Lemma \ref{bb1}, all the constituents of $\psi$ belong to the same block. 
 So it suffices to look at the classes $C_5(k,l)$, where 
$(k,l)\in \ZZ_{q+1}\times \ZZ_{q+1}$ and $k\neq l$.
 We have one of the following cases, modulo linear
characters of $H$:

\begin{itemize}
\item[(1)] Take $\psi_1= 1_H + \chi_{2}(u_2) + \chi_{2}(u_2') + 
\chi_{6}(u_6,v_6,w_6)$. As all the summands are in the same block,
 it follows that $q+1$ divides   $3u_2$, $3u_2'$ and $u_6+v_6+w_6$. 
We obtain  
$$1+\xi_1^{-(k-l)u_2}+\xi_1^{-(k-l)u_2'}=\xi_1^{(k-l)(u_6+v_6)}+
\xi_1^{-(k-l)v_6}+\xi_1^{-(k-l)u_6}.$$
This holds for every $k,l$ with $k\neq l$.
This implies
that $3$ divides $q+1$, and, by symmetry, 
$u_2=a$, $u_2'=2a$ and $(u_6,v_6,w_6)=(a,2a,3a)$, with $a=(q+1)/3$.
\item[(2)] Take $\psi_2= \chi_{2}(0) + \chi_{4}(u_4,v_4) + \chi_{6}(u_6,v_6,w_6)$, with $q+1$ dividing both  $2u_4+v_4$ and $u_6+v_6+w_6$. We obtain $$1+\xi_1^{-(k-l)u_4} +\xi_1^{2(k-l)u_4} = \xi_1^{(k-l)(u_6+v_6)}+\xi_1^{-(k-l)v_6} + \xi_1^{-(k-l)u_6}.$$
We have the following  possibilities:
\begin{itemize}
\item[(a)] Either $u_6=0$ or $v_6=0$: this implies $u_4=0$ and $v_4=0$, a contradiction with the conditions on the parameters.
\item [(b)] $u_6+v_6=0$: again we obtain $u_4=v_4=0$.
\end{itemize}
\item[(3)] Take $\psi_3= \chi_{2}(0) + \chi_{5}(u_5,v_5)$:  we obtain $\xi_1^{k(u_5+v_5)+lu_5}=1$, which implies $u_5=v_5=0$, a contradiction with the conditions on the parameters.
\end{itemize}
\end{proof}

Next, let $G=SU(3,q)$, where $3\nmid (q+1)$. This case is very similar to the case
$U(3,q)$. We list the the relevant informations in Tables \ref{TU3}.2-A,B,C. Thus, it can
be easily proved the following.

\begin{lemma}\label{LU3.n2}
Let $G=SU(3,q)$, where $3\nmid (q+1)$. Then $G$  admit no reducible $p$-vanishing character of degree $|G|_p$.
\end{lemma}

Finally, let $G=SU(3,q)$ with $3\mid (q+1)$. 
Applying Procedure \ref{procedure}  we obtain on step  (1) the
$Syl_p$-decompositions of Table
\ref{TU3}.3-B  and  $\Delta_G=\{1, 2, 7, 8, 9 \}$. 
Step (2) yields only one  character
 $\psi_1=\chi_1+2\chi_2+\chi_7+\chi_8+\chi_9$.
With step (3) we obtain all  $Syl_p$-vanishing characters of $G$ of degree
$|G|_p$ (see Table \ref{TU3}.3-C).

\begin{lemma}\label{LU3.2}
Let $G=SU(3,q)$, with $3 \mid (q+1)$. Then $G$ admits a unique reducible
$p$-vanishing character $\psi$ of degree $|G|_p$, namely
$\psi=1_G+ 2\chi_{2} + \chi_7+\chi_8+\chi_9$.
\end{lemma}

\begin{proof}
The cases $q=2,5,8$ can be easily dealt with \textsc{magma}. So, assume $q\geq 11$.
Applying Lemma \ref{bb1}, we  easily rule out all the options  where  the \ir constituents
of $\psi$  are not in the same block. So, we have to study only the following five cases:

\begin{itemize}
\item[(1)] $\psi_1= \chi_{1} + 2\chi_{2}  + \chi_{7} + \chi_{8} + \chi_{9}$: this
character is $p$-vanishing.
\item[(2)] $\psi_2= \chi_{1} +2 \chi_{2} +  \chi_{10}(n,m)$: on the
class $C_7(1)$ we obtain  
$\xi_1^{n-2m}+\xi_1^{m-2n}+\xi_1^{n+m}=3$. So,  $\xi_1^{n-2m}=\xi_1^{m-2n}=\xi_1^{n+m}=1$
and hence  
$(m,n)=(\frac{q+1}{3},\frac{2(q+1)}{3}),$ $(\frac{2(q+1)}{3},\frac{q+1}{3})$. But, for
this choice of the parameters,
$\chi_{10}=\chi_7+\chi_8+\chi_9$.
\item[(3)] $\psi_{21}= \chi_{2} + \chi_{4}(k) + \chi_{7} + \chi_{8} + \chi_{9}$: 
on the class $C_7(1)$ we obtain $\xi_1^{k}+\xi^{-2k}=2$ and so $\xi_1^k=1$, a
contradiction.
\item[(4)]  $\psi_{22}= \chi_{2} + \chi_{4}(k) + \chi_{10}(n,m)$:  
on the classes $C_7(a)$ ($a=1,2,3$) we obtain
$1+\xi_1^{ak}+\xi_1^{-2ak}=\xi_1^{a(n-2m)}+\xi_1^{a(m-2n)}+\xi_1^{a(n+m)}$. It follows
that $(m,n)=(\frac{q+1}{3},\frac{2(q+1)}{3}), (\frac{2(q+1)}{3},\frac{q+1}{3})$. So
$(q+1)\mid k$, a contradiction.
\item[(5)]  $\psi_{23}= \chi_{2} + \chi_{5}(n)$: on the class 
$C_7(1)$  we obtain $\xi_1^n=1$, i.e. $(q+1)\mid n$, a contradiction 
with the conditions on the parameter set $I_5$.
\end{itemize}
Thus, $\psi_1$ is the only  reducible $p$-vanishing character of $G$. 
\end{proof}

\subsection{Groups ${}^2B_2(q^2)$ and ${}^2G_2(q^2)$}


\begin{lemma}\label{L2B2}
Let $G={}^2B_2(q^2)$,  $q^2=2^{2m+1}$. Then $G$ admits exactly 
$2^{m-1}(2^m+1)$ reducible $p$-vanishing characters $\psi$ of degree $|G|_p$, all of type
$$\psi= 1_G + \chi_{2} + \chi_{3} + \chi_{6}(k).$$
\end{lemma}

\begin{proof} 
We apply Procedure \ref{procedure}.
In step (1) we obtain the
$Syl_p$-decompositions of Table \ref{T2B2}-B and $\Delta_G=\{1,2,3,6\}$.
Step (2) produces a unique character
$\psi_1= \chi_{1} + \chi_{2} + \chi_{3} + \chi_{6}$. 
Finally, with step (3) we are able to list all the
$Syl_p$-vanishing characters of $G$ of degree $|G|_p$ (see Table \ref{T2B2}-C).
As every element of $G$ is either semisimple or unipotent,
every $Syl_p$-vanishing character is $p$-vanishing. \end{proof}


Let $G={}^2G_2(q^2)$, with $q^2=3^{2m+1}$, $m\geq 0$.
In order to reduce the number of the characters 
$\chi_h$  we have to consider, it is easier to work first only with the classes 
$C_1, C_2, C_5$. 
We obtain the decompositions listed in Table \ref{T2G2}-B. 
These decompositions allow us to work only with the characters $\chi_h$,
 where $h\in \{ 1, 4, 5, 7, 9, 12, 13, 14 \}$. Working only with these 
constituents, we obtain a unique character of degree 
$|G|_p$ vanishing on the classes $C_2 $ and $ C_5$: 
$\theta=\chi_{1} + 2\chi_{5} + 2\chi_{7} + \chi_{14}$. 
Next we apply the decompositions of Table  \ref{T2G2}-B 
to the character $\theta$, to verify if the resulting 
character is $Syl_p$-vanishing. Finally, we are 
able to list all  $Syl_p$-vanishing characters of $G$ of 
degree $|G|_p$ (see Table \ref{T2G2}-C).

\begin{lemma}\label{L2G2}
Let $G={}^2G_2(q^2)$, with $q^2=3^{2m+1}$. Then $G$ admits exactly $3^m(3^m+1)/2 $ reducible $p$-vanishing characters $\psi$ of degree $|G|_p$, all of type
$$\psi=  1_G + \chi_{3} + \chi_{5} + \chi_{6} + \chi_{7} + \chi_{14}(k).$$
\end{lemma}

\begin{proof}
Direct inspection of the character table shows that  every $Syl_p$-vanishing character of
$G$ of degree $|G|_p$ is $p$-vanishing. 
\end{proof}

Observe that the group $G={}^2G_2(3)'$ has a unique reducible $3$-vanishing character of
degree $|G|_3$: $\phi=\chi_1+\chi_6$ (with the notation of \cite{Atl}). 
On the other hand, it has $16$ reducible $3$-vanishing characters of degree $3|G|_3$. 
\begin{enumerate}
\item Ten characters of shape $\phi_1=a\chi_7+b\chi_8+c\chi_9$, where $0\leq a\leq b\leq
c\leq 3$ and $a+b+c=3$.
\item Six characters of shape $\phi_2=\chi_1+\chi_6+d\chi_7+e\chi_8+f\chi_9$, where $0\leq
d\leq e\leq f\leq 2$ and $d+e+f=2$.
\end{enumerate}

This can be easily proved looking at the character table of $G$.

\section{$Syl_p$-vanishing characters and Levi subgroups}
 
 Let ${\mathbf H}$ be a reductive group with Frobenius endomorphism
$Fr$ and ${\mathbf G}$ the maximal semisimple subgroup of ${\mathbf
H}$. Set $H={\mathbf H}^{Fr}$ and $G={\mathbf G}^{Fr}$.

\begin{lemma}[{\cite[Proposition p. 162]{Lu}}] \label{lu5} 
Let $\tau$
be an \ir character of $H$. Then $\tau|_G$ is \mult free.
\end{lemma}

\begin{propo}\label{lu6}
Let $G,H$ be as above.
Let $\chi$ be a $p$-vanishing
character of $G$. Then $\chi$ is $H$-invariant and extends to a $Syl_p$-vanishing
character $\chi'$ of $H$ such that $\chi'|_G=\chi$.
\end{propo}

\begin{proof}
Note that $\chi$ is an integral linear combination of the
characters of projective indecomposable modules, and every of them 
 extends to $H$ 
(see \cite[Lemma 3.14]{Z}). It follows that $\chi$ extends to a
generalized $p$-vanishing  character $\phi$ of $H$ such that $\phi|_G=\chi$. This implies
that $\chi$ is $H$-invariant.
Next, we shall modify $\phi$ to obtain an
ordinary $Syl_p$-vanishing character.

For $\si\in\Irr G$ denote by $\tilde \si$ the sum of the distinct
$H$-conjugates of  $\si$. As $\chi$ is $H$-invariant, we can write
$\chi=\sum m(\tilde \si)\tilde \si$ for some integers $m(\tilde
\si)\geq 0$, with the sum over the minimal $H$-invariant
characters of $G$.

Let $\tau\in \Irr H$. By Clifford's theorem,
$\tau|_G=d_{\tau}\tilde \si$ for some integer $d_{\tau}>0 $. By
Lemma \ref{lu5}, $d_{\tau}=1$. Set $\Irr_{\tilde \si}H=\{\tau \in
\Irr H:\tau|_G=\tilde \si\}$. Note that if $\tau_1,\tau_2\in \Irr_{\tilde \si}H$
then $\tau_1|_G=\tau_2|_G=\tilde \si.$ Then we have

$$\phi=\sum_{\tilde \si} ~\sum_{\tau\in\Irr_{\tilde \si}H} m(\tau)\tau,$$ 
where the first sum runs over the minimal
$H$-invariant characters of $G$. Therefore, the \mult of $\tilde
\si$ in $\phi$ equals $m(\tilde\si)=\sum_{\tau\in\Irr_{\tilde \si}H} m( 
\tau) $.
Let
$a_{\tilde \si}$ (resp., $b_{\tilde \si}$) be the sum of all
positive (resp., negative) coefficients $m(\tau)$
occurring in this sum. Then $a_{\tilde \si}\geq b_{\tilde \si}$ as $m(\tilde \si)\geq0$.
Therefore, $\sum_{\tau\in\Irr_{\tilde \si}H} m( \tau)\tau|_G=
(a_{\tilde \si}-b_{\tilde \si})\tau'|_G $ for any
$\tau'\in\Irr_{\tilde \si}H$, and $(a_{\tilde \si}-b_{\tilde
\si})\cdot \tau'$ is an ordinary character of $H$. So we
replace every generalized character $\sum_{\tau\in\Irr_{\tilde \si}H}
m( \tau)\tau$ by $(a_{\tilde \si}-b_{\tilde \si})\cdot
\tau'$ to obtain a required modification $\chi'$ of $\phi$.
\end{proof}

\begin{remar}
 Proposition \ref{lu6} is not true for $Syl_p$-vanishing
characters. Indeed, the characters of $G=SL(2,q)$ in the item (3)
of Lemma \ref{s21} do not extend to $GL(2,q)$. 
\end{remar}

Note that $H$ is a group with split
BN-pair, see \cite[24.10]{MT} or \cite[p. 50]{Ca}. So the Steinberg character can be
defined for $H$ as in \cite[66.35 and 71.2]{CR2}. This is of
degree  $|H|_p$, and hence of defect 0.

\begin{lemma}\label{cc5}
Let $\chi$ be a $Syl_p$-vanishing character of $H$. Let $L$ be a Levi subgroup of some parabolic subgroup $P$ of
$H$. 
\begin{itemize}
\item[(1)] If $\chi\in\Irr H$ then $\chi(1)=|H|_p$, and $\chi|_G$ is the Steinberg character of $G$.
\item[(2)] Suppose that the truncation $\overline{\chi}_L$ is \irt Then
$\chi(1)=|H|_p$.
\item[(3)] If $\chi\in\Irr H$ is of defect $0$ then the truncation
$\overline{\chi}_L$ is \ir of defect $0$.
\end{itemize}
\end{lemma}

\begin{proof}
(1) is well known.
(Indeed,  $\chi(1)$ is divisible by
$|H|_p$, so $\chi$ is of defect 0, and hence $\chi\pmod p$ is
irreducible. The $p$-modular \ir characters of $H$ remain
irreducible under restriction to $G$ (cf. \cite[3.14]{Z}), and
hence have degree at most $|G|_p=|H|_p$ (see \cite[72.19]{CR2} or
\cite[Corollary of Theorem 46]{St}). This implies (1).)

(2) As $\overline{\chi}_L$ is irreducible, it is  of degree
$|L|_p$ by (1). As $|H :P|$ is coprime to $p$, the claim follows
from Lemma \ref{x2a}.

(3) As $\chi$ is of defect 0, the reduction of $\chi$ modulo $p$
is \ir and the corresponding module $M$ is projective.  By Smith's
theorem \cite[72.32]{CR2}, the Harish-Chandra restriction
$\overline{M}_L$ of $M$ to $L$ is \ir and, moreover, projective.
It is known that $\overline{\chi}_L$ is the character of
$\overline{M}_L$, cf. \cite[Lemma 3.3]{Z}. \itf
$\overline{\chi}_L$ is \ir of defect 0.
\end{proof}

\begin{lemma}\label{d0d}
 Let $\phi$ be a defect $0$ \ir character of $H$, and $L$ a Levi subgroup
 of $H$. Then $\phi=St$ \ii
$\overline{\phi}_L=St_L$.
\end{lemma}

\begin{proof}
The ``only if'' part is well known, see for instance \cite[p. 72]{DM}. 
For the ``if`` part suppose that $\overline{\phi}_L=St_L$. 
First we prove the lemma in the case where
$L=T\subset B$.

Recall that $\phi(1)=|H|_p$ \cite[72.20]{CR2}, equivalently,
 $\overline{\phi}_T(1)=1$ (Lemma \ref{x2a}).
Suppose that $(\phi,1_B)=1$. By Lemma \ref{cc5}, 
 $St_H|_G=St_G$  and $\phi|_G=St_G$. For every $h\in H$ we have $\phi(h)=St_H(h)\cdot \mu(h)$ for some complex number $\mu(h)$. \itf the function
 $\mu$ is multiplicative (view  $\phi$
 and $St_H$ as \reps rather than characters).
 As $H$ is a finite group, $\mu$ is a representation. Therefore, $\phi=St_H\cdot \mu$.
As $U\subset G$, we have $\mu(U)=1$.   (Usually $U$ is in the
kernel of every linear character of $H$ but there are a few
exceptions.)   Let $\phi=St\cdot \mu$. Then $1=(\phi|_B,
1_B)=(St|_B\cdot \mu|_B, 1_B)=(St|_B, \mu^*|_B\cdot 1_B )=(St|_B,
\mu^*|_B)$, where $\mu^*$ is the complex conjugate to  $\mu$. By
the above, $\mu^*|_B=1_B$. We show that $\mu^*=1_H$.

Let $H_0$ be the subgroup generated by the unipotent elements in
$H$. Then $H/H_0$ is a $p'$-group. As $H_0B$ contains $B$, this is
a parabolic subgroup of $H$, and  $H_0B=H$ as every parabolic
subgroup coincides with its normalizer \cite[65.19]{CR2}.  \itf
the mapping $\mu\ra \mu|_B$ yields an injective mapping from $\Irr
H/H_0$ to $\Irr B/B'$, where $B'$ is a derived subgroup of $B$.
Therefore, $\mu|_B=1_B$ implies $\mu=1_{H}$.

 Let $L\neq T$ and let $\eta=\overline{\phi}_L$. By
transitivity of truncation, $\overline{\phi}_T=\overline{\eta}_T$.
By the above, $\phi=St$ is equivalent to $\overline{\phi}_T=1_T$, and
$\overline{\eta}_T=1_T$ is equivalent to $\eta=St_L$. So the
result follows. 
\end{proof}

We keep the notation $H$, $G$ as in the beginning of the section, and assume that 
${\bf G}$ is simply connected.
  In what follows we normally 
 view $H$ and $G$ as  groups
with a split $BN$-pair defined by a Borel subgroup $B$ and a subgroup
$N$ satisfying certain conditions, see \cite[\S 65]{CR2}.  We set
$T=N\cap B$ and $W=N/T$ (which is called the Weyl group). We denote by $\ell$ the
BN-pair rank of $H$. (This is not the same as the rank
of ${\mathbf H}$.) Set $I_\ell=\{1,\ldots,\ell\}$. The standard
parabolic
subgroups of $H$ are those containing $B$.  They 
 form a lattice isomorphic to the
lattice of the subsets of $I_\ell$. We denote by $P_J$ the
standard parabolic subgroup corresponding  to a subset $J$ of
$I_\ell$ (and $B$ corresponds to the empty set). Let
$U_J:=O_p(P_J)$ be the unipotent radical of $P_J$. Then
$P_J=U_J\rtimes L_J$ for some subgroup $L_J$ called the Levi
subgroup of $P_J$. The Levi subgroups are conjugate in $P_J$, and
it is possible to choose $L_J$ to contain $T$. (This is always
assumed below.) If $K\subseteq J$, then $P_K\cap L_J$ is a
standard parabolic subgroup of $L_J$ containing the Borel subgroup
$B\cap L_J$ of $L_J$. Furthermore, $O_p(P_K\cap L_J)=U_K\cap L_J$,
$P_K\cap L_J=(U_K\cap L_J)\rtimes L_K$ and $U_K=U_J(U_K\cap L_J)$.
In turn, $L_J$ is a group with BN-pair whose Weyl group  $W_J$ is
a subgroup of $W$. The Weyl group $W$ has a unique (up to
equivalence) \irr $\rho$ of degree $\ell$ called the reflection
representation, and $\ep:=\det \circ \rho$ is called the sign
character of $W$.

Let $P_J=L_JU_J$, where $L_J$ is a  Levi subgroup of $P_J$. Denote
by  $G_J$  the normal subgroup of $L_J$ generated by the unipotent
elements of $L_J$. 
Then $L_J=G_JT$.  
Set $B_J=B\cap G_J$ and $T_J=T_0\cap G_J$. Then $B_J$ is a Borel
subgroup of $G_J$ and $B_J=(U\cap L_J)\rtimes T_J$. (Note that
$B_J$ differs from $B\cap L_J$, which is a Borel subgroup of
$L_J$.)

Recall that we call $G$  a Chevalley group if ${\mathbf G}$ is
simple and simply connected.

 In Section 5 we have obtained the following: 

\begin{propo}\label{r1g}
Let $G$ be  of
BN-pair rank $1$. Then $G$ admits a reducible $p$-vanishing character $\chi$ of degree $|G|_p$, 
unless  $G=SU(3,q)$ with 
$(q+1,3)=1$. Furthermore, $(\chi,1_G)=1$, and some constituent of $\chi$ is a cuspidal character of $G$.
\end{propo}

\begin{proof}
The result follows from Lemmas \ref{s21}, \ref{LU3.n2}, \ref{LU3.2}, \ref{L2B2} and \ref{L2G2}. 
\end{proof}

In order to deal with groups of higher rank, we establish some properties of $Syl_p$-vanishing characters.

\begin{lemma}\label{dj2}
Suppose that $G=H$.
Let $I_\ell=J\cup J'$ be the  union of two subsets $J,J'$. Then
$T=T_JT_{J'}$.
\end{lemma}
\begin{proof}
The result  follows from \cite[Theorem 2.4.7(a)]{GLS}.
\end{proof}

\begin{lemma}\label{dk3}
Let $\chi$ be an arbitrary character of $H$, and let
$\beta_J$ be the truncation of $\overline{\chi}_{L_J}$ to  $T$
(so $\beta_J$ is a character of $T)$.
\begin{itemize}
\item[(1)] $\beta_J=\overline{\chi}_{T}$.
\item[(2)] Set $\lam=\overline{\chi}_{L_J}|_{G_J}$. Then
$\overline{\lam}_{T_J}=\beta_J|_{T_J}$.
\end{itemize}
\end{lemma}

\begin{proof}
(1) follows from the transitivity of  truncation.

  (2) By the definition of  truncation, for all $t\in T$ and $g\in G_J$ we have
$$\beta_J(t)=\frac{1}{|U|}\sum_{u\in U} \chi(ut)~~
\textrm{ and } ~~
\lambda(g)=\frac{1}{|U_J|}\sum_{v\in U_J} \chi(vg).$$
So, for all $x \in T_J$,
$$
(\Tr{T_J}{\lambda})(x) = \frac{1}{|U\cap L_J|}\sum_{y\in U\cap L_J} \lambda(yx)= $$
$$ = \frac{1}{|U\cap L_J|\cdot |U_J|}\sum_{y\in U\cap L_J} \sum_{v \in U_J} \chi(vyx)\, =
\, \frac{1}{|U|}\sum_{u\in U} \chi(ux) \,=\, \beta_J(x),$$
using the fact that  $U=U_J(U\cap L_J)$.
\end{proof}

\begin{lemma}\label{dm4} 
Assume $G=H$.
Let $I_\ell=J\cup J'$ be the  union 
of two subsets $J,J'$.
Let $\chi$ be a $Syl_p$-vanishing character of $G$. 
Suppose that $\Tr{L_J}{\chi}$ and $\Tr{L_{J'}}{\chi}$ are \irt Then $\Tr{T}{\chi}=1_T$.
\end{lemma}

\begin{proof}
Let $\eta =(\Tr{L_J}{\chi})|_{G_J}$ and $\eta'=(\Tr{L_{J'}}{\chi})|_{G_{J'}}$. Then $\eta=St_{G_J}$ and $\eta'=St_{G_{J'}}$ 
(Lemma \ref{cc5}(1) and (2)). By Lemma \ref{d0d}, the truncation of $\eta$ to $T_{J}$ is $1_{T_{J}}$ and the truncation of $\eta'$ 
to $T_{J'}$ is $1_{T_{J'}}$.
By Lemma \ref{dk3}, $\Tr{T} {\chi}|_{T_J}=1_{T_{J}}$ and
$\Tr{T} {\chi}|_{T_{J'}}=1_{T_{J'}}$. Now Lemma \ref{dj2} yields the result.
\end{proof}

The \f result is our primary target; this allows us to use
classical results (quoted in Lemmas \ref{bb5} and \ref{bc5}) on \ir constituents of 
of $1_B^G$. 

\begin{propo}\label{pt1}
Let $G=H$ be 
and $\chi$ be a $p$-vanishing character of $G$ of degree $|G|_p$. 
Then $(\chi,1_B^G)=(1_B, \chi|_B)=1$ (equivalently, $\overline{\chi}_T=1_T$).
\end{propo}

\begin{proof}
The first equality is the Frobenius reciprocity. The second
equality is equivalent to saying that $\Tr{T}{\chi}=1_{T}$. We
use induction on the BN-pair rank of $G$. By Proposition \ref{r1g}, the proposition 
is true for groups of rank $1$. Let
$I_\ell=J\cup J'$ be the union of two subsets $J,J'$. Then
$\eta:=\Tr{L_J}{\chi}|_{G_J}$ (resp.,
$\eta':=\Tr{L_{J'}}{\chi}|_{G_{J'}}$) is a $p$-vanishing character of
$G_J$ (resp., $G_{J'}$) of degree $|G_J|_p$ (resp., $|G_{J'}|_p$), see Lemma \ref{x2a}.  
By induction, the proposition is true for $G_J$
and $G_{J'}$, so $\Tr{T_J}{\eta}=1_{T_J}$ and
$\Tr{T_{J'}}{\eta'}=1_{T_{J'}}$. By Lemma \ref{dk3},  $\Tr{T_J}{\eta}$ coincides
with $\Tr{T}{\chi}|_{T_J}$ (as $L$ contains $T$). Therefore,
$\Tr{T}{\chi}|_{T_J}=1_{T_J}$ and $\Tr{T}{\chi}|_{T_{J'}}=1_{T_{J'}}.$
Now the result follows from Lemma \ref{dj2}. 
\end{proof}

\begin{corol}\label{nn7}
Assume $G=H$. 
Let $\chi$ be a $Syl_p$-vanishing character of $G$ such that 
$\overline{\chi}_T=1_T$, and $L$ a Levi subgroup of $G$. 
\begin{itemize}
\item[(1)] If $\overline{\chi}_L$ is \ir of defect $0$, then $\overline{\chi}_L=St_L$.
\item[(2)] If $L$ is not solvable and $\overline{\chi}_L$ 
has a $1$-dimensional constituent $\nu$, then $\nu=1_L$.
\end{itemize}
\end{corol}

\begin{proof}
$(1)$ By transitivity of Harish-Chandra restriction,
the truncation of $\overline{\chi}_L$ to $T$ coincides with $1_T$. 
So the claim follows from Lemma \ref{d0d} applied to $L$ in place of $G$ and $T$
in place of $L$.

(2) Since the Borel subgroup of $L$ is $(U\cap L )T$, by the definition of truncation, 
$\overline{\nu}_T=(\nu|_{U\cap L},1_{U\cap L})\cdot \nu|_T$. Hence, $\nu|_T=1_T$. 
On the other hand, let $L_u$ be the subgroup of $L$ generated by the unipotent elements of $L$. 
Then, one easily observes that $L=L_u\cdot T$.  
Since $\nu$ is a linear character of $L$, the derived subgroup  of $L$ coincides with 
$L_u$. It follows that $\nu|_{L_u}=1_{L_u}$, so the result follows.
\end{proof}

It is known that the blocks of non-zero defect of a finite reductive
group $H$ are in bijection with the elements of $Z(H)$, see
\cite[\S 8.7]{Hub}; this is deduced from the fact that every such
block has defect $|U|$ and  $C_H(U)=Z(H)$. The \f result shows that if
a $p$-vanishing character of degree $|G|_p$ is in the principal $p$-block
of $G$, then the Harish-Chandra reduction  $\overline{\chi}_L$ of $\chi$ to every
Levi subgroup $L$ belongs to the principal $p$-block
of $L$.

\smallskip
Observe that Proposition \ref{pt1} is a special case (for $L=T$) of the  \f result.

\begin{corol}\label{ke1} 
Let $G=H$ 
and $L$ a Levi subgroup of $G$. Let $\chi$ be
a $p$-vanishing character of $G$ of degree $|G|_p$.  Then $Z(L)$ is
in the kernel of  $\overline{\chi}_L$.
\end{corol}

\begin{proof}
As $Z(L)$ is contained in a Borel subgroup $B_L$ of $L$, it
is contained in a maximal torus $T_L$ of $B_L$. However, $T_L$
coincides with  a maximal torus $T$ of a Borel subgroup $B$ of $G$. So $Z(L)\subset T$. By Proposition \ref{pt1}, $1_{Z(L)}$ is a constituent of
  $\overline{\chi}_L |_{Z(L)}$. Therefore, there is an \ir
  constituent $\tau$, say, of $\overline{\chi}_L $ such that $
\tau|_{Z(L)} =\tau(1)\cdot 1_{Z(L)}  $. In other words,
$Z(L)$ is in the kernel of $\tau$. 

Suppose that the lemma is false. Then  $Z(L)$ is not in the kernel
of some  \ir constituent $\tau'$, say, of $\overline{\chi}_L$.
Therefore,  $\tau'$ and $\tau$ belong to distinct $p$-blocks. 
As $Z(L)\subset T$, it follows that $(|Z(L)|,p)=1.$ This is
not the case by Lemma \ref{bb1}, as $\overline{\chi}_L $ is $p$-vanishing. 
\end{proof}

The \f lemma is one of  standard results on \reps of groups with
$BN$-pair, see  \cite[68.24, pp. 683 - 684]{CR2}. Recall that the
Weyl group $W$ of a group $H$ with BN-pair has a natural
representation as a linear group generated by reflections
\cite[64.29]{CR2}. The  determinant  mapping
$w\rightarrow \det w$ ($w\in W$) of this \rep yields a one-dimensional  \rep of $W$ called
the sign \rept In the statement below the character of the sign \rep  is
denoted by $\ep$. For a subgroup
$X$ of $W$ set $\ep_X=\ep|_X$.

\begin{lemma}\label{bb5}
Let $H$ be a finite reductive  group viewed as a 
group with $BN$-pair of rank $\ell$,
$B$ a Borel subgroup of $H$ and $W$ the Weyl group of the BN-pair.

\begin{itemize}
\item[(1)] There exists a bijection $\lambda\mapsto \chi_\lambda$
between $\Irr W$ and the set $\{\chi \in \Irr H: (\chi,1_B^H)>
0\}$ such that $\chi_{1_{W}}=1_H$ and $\chi_\ep=St$. In addition, 
$(\chi_\lam, 1_B^H)=\lam(1)$.

\item[(2)] For each subset $J$ of $I_\ell =\{1,\ldots, \ell\}$, let
$W_J$ and $P_J$ be the corresponding  subgroups of $W$ and $H$,
respectively. Then $$(\chi_\lambda, 1_{P_J}^{H})=(\lambda,
1_{W_J}^{W})=(\lam|_{W_J}, 1_{W_J}).$$ In particular,
$1_B^H=\sum_{\lam\in\Irr W}\lam(1)\cdot \chi_\lam$.

\item[(3)] {\rm \cite[70.24]{CR2}} $(\chi_\lam, St^{\#H}_{L_J})=(\lam,
\ep_{W_J}^{W})=(\lam|_{W_J}, \ep_{W_J})$.
\end{itemize}
\end{lemma}

\begin{lemma}[{\cite[(68.26)]{CR2}}]\label{bc5}
Let $W$ be in Lemma $\ref{bb5}$; assume that $W$ is not a direct product of proper subgroups.
 Then each irreducible character $\lambda$ of $\Irr W$
 is uniquely determined by the multiplicities
 $\{(\lambda, 1_{W_J}^W): J\subseteq I_\ell\}$, with the following exceptions:
\begin{itemize}
\item[(1)] the characters of degree $2$ if $W$ is the dihedral group of order $12$ or $16$;
\item[(2)] the two characters of degree $2^9$ if $W$ is of type $E_7$;
\item[(3)] the four characters of degree $2^{12}$ if $W$ is of type $E_8$.
\end{itemize}
\end{lemma}

\begin{corol}\label{bd6}
Keep notation of Lemma $\ref{bb5}$ and assumption of Lemma \emph{\ref{bc5}}.
\begin{itemize}
\item[(1)]  
 Each irreducible character
$\lambda$ of $\Irr W$
 is uniquely determined by the multiplicities
 $\{(\lambda, \ep_{W_J}^W): J\subseteq I_\ell\}$,
with the exceptions $(1),(2),(3)$ in Lemma $\ref{bc5}$.
\item[(2)] The character $\chi_\lam$ is uniquely determined by the
multiplicities $\{(\chi_\lambda, 1_{P_J}^{H}): J\subseteq
I_\ell\}$ as well as the multiplicities $\{(\chi_\lambda,
St_{L_J}^{\# H}): J\subseteq I_\ell\}$, unless $\lam$ is listed in
items $(1),(2),(3)$ of Lemma $\ref{bc5}$.
\end{itemize}
\end{corol}

\begin{proof}
(1) We have $(\lambda, \ep_{W_J}^W)=(\lambda,
(\ep|_{W_J})^W)=(\lambda, 1_{W_J}^W\cdot \ep)=
(\lambda\cdot \ep, 1_{W_J}^W )$. By  Lemma \ref{bc5},
$\lambda\cdot \ep$ is determined by the multiplicities
$\{(\lambda\cdot\ep , 1_{W_J}^W): J\subseteq I_\ell\}$, so
$\lambda$ is determined by $\{(\lambda, \ep_{W_J}^W): J\subseteq
I_\ell \} $.

(2)  follows from Lemmas \ref{bb5} and \ref{bc5}, and statement
(1) of this corollary.  
\end{proof}

The \f lemma is a next important step in our analysis, as this
settles a crucial special case, by showing that there is a Levi
 subgroup of $G$ such that  the Harish-Chandra restriction of
$\chi$ to $L$ is not irreducible.

\begin{propo}\label{tt8} 
Assume $G=H$.
Let $I_\ell=J\cup K$ be the  union of two subsets $J,K$. Let
$\chi$ be a 
$Syl_p$-vanishing character of $G$.
\begin{itemize}
\item[(1)] Suppose that $\overline{\chi}_{L_J}$ and $\overline{\chi}_{L_K}$ are irreducible. Then $\chi=St$.
\item[(2)] Suppose that $\overline{\chi}_T$ is irreducible and
$(\overline{\chi}_{L_J},1_{L_J})=(\overline{\chi}_{L_K},1_{L_K})=1$. 
Then $(\chi,1_G)=1$.
\end{itemize}
\end{propo}

\begin{proof}
(1)  As $\overline{\chi}_{L_J}$ and $\overline{\chi}_{L_{K}}$
are \ir and $p$-vanishing, they are of defect $0$. By Lemma \ref{cc5}(2),
$\overline{\chi}_{L_J}(1)=|L_J|_{p}$,  and hence   $\chi(1)=|G|_p$ by 
Lemma \ref{x2a}. Now,  $\overline{\chi}_{T}=1_{T}$ by Lemma \ref{dm4}, and  
$\overline{\chi}_{L_J}=St_{L_J}$ by Lemma \ref{d0d}. 
Similarly, $\overline{\chi}_{L_{K}}=St_{L_{K}}$. The equality $\overline{\chi}_{T}=1_{T}$
also means that
$(1_B^G, \chi)=1$. Let $\tau\in\Irr G$ occur both in $\chi$ and
$1_B^G$ (so $(\tau,1_B^G)=1$).  
By  Harish-Chandra reciprocity, $(
\overline{\chi}_{L_J}, St_{L_J})= (\chi, St_{L_J}^{\#G})=1$ and
$(\overline{\chi}_{L_K},St_{L_K})= (\chi, St_{L_K}^{\#G})=1$.
So $\tau$ occurs with \mult $1$ in $1_B^G$, $St_{L_J}^{\#G}$ and $St_{L_K}^{\#G}$. Let $\tau$ correspond to the linear character 
$\lam\in \Irr W$ according to Lemma \ref{bb5}. By Lemma \ref{bb5}(3), $(\tau,St_{L_J}^{\#G})=(\lam|_{W_J}, \ep_{W_J})$. 
Similarly, $(\tau,St_{L_K}^{\#G})=(\lam|_{W_K}, \ep_{W_K})$. This is only possible if $\lam=\ep$, and hence $\tau=St$.

(2) Since $\overline{\chi}_T$ is irreducible and $(\overline{\chi}_{L_J},1_{L_J})=1$, we
have $\overline{\chi}_T=1_T$. 
We proceed as  in item (1). 
Let $\tau\in\Irr G$ occur both in $\chi$ and
$1_B^G$ (so $\tau$ corresponds to the linear character $\lam\in \Irr W$ according to Lemma
\ref{bb5}). By Lemma \ref{bb5}(2),  $1=(\tau,1_{P_J}^G)=(\lam|_{W_J},1_{W_J})$. Similarly,
$(\lam|_{W_K},1_{W_K})=1$. This is only possible if $\lam=1_W$, and hence $(\chi,1_G)=1$.
\end{proof}

\section{Chevalley groups of BN-pair rank $2$}\label{rk2}

In this section, we deal with Chevalley groups of BN-pair rank $2$. 
These are $SL(3,q)$,  $Sp(4,q)$, $SU(4,q)$
$SU(5,q)$, $G_2(q)$, ${}^3D_4(q)$  and ${}^2F_4(q^2)$.

In view of  Proposition \ref{pt1} we state the following: 
 
\begin{lemma} 
Let $G$ be of BN-pair rank $2$ and let $\nu$ be 
a \mult $1$ \ir constituent of
$1_B^G$. Then, either $\nu=1_G$, or $\nu=St$ 
or of degree written below:

$G={}^2F_4(q^2)$:  $\frac{q^2(q^6+1)(q^{12}+1)}{(q^4+1)(q^2+1)}$ and
$\frac{q^{10}(q^6+1)(q^{12}+1)}{(q^4+1)(q^2+1)}$;
\med

$G={}^3D_4(q)$: $q^7(q^4-q^2+1)$ and $q(q^4-q^2+1)$;

$G=G_2(q)$: $2$ characters of degree $q(q^4+q^2+1)/3$;

$G=Sp(4,q)$: $2$ characters of degree $q(q^2+1)/2$;
$G=SU(4,q)$: $q(q^2-q+1)$ and $q^3(q^2-q+1)$;
$G=SU(5,q)$: $q^2(q^5+1)(q+1)$ and $q^4(q^5+1)(q+1)$.
\end{lemma}

\begin{proof}
These degrees are given in \cite[pp. 114 - 115]{CIK}.
\end{proof}

\subsection{Groups $SL(3,q)$ and $GL(3,q)$}

Note that $SL(3,q)$ is the only quasi-simple group of rank $2$ where we can avoid
substantial computations.

\begin{propo}\label{sg3}
Let $H=GL(3,q)$ and
 let $\chi$ be a reducible $Syl_p$-vanishing character of $H$ of degree $q^3$. Then
$\chi=\mu+\si^{\#H}$, where $\mu$ is a linear character of  $H$ and
$\si$ is a cuspidal \ir character of a proper Levi subgroup $L\neq T$
(and hence $\si(1)=q-1$). In addition, $\si^{\#H}$ is irreducible.

Let $\si$ be a cuspidal character of $L$. Then the  character $1_H+\si^{\#H}$ is $p$-vanishing \ii 
$Z(L)$ is the kernel of $\si$.
\end{propo}

\begin{proof}
Let $L$ be a Levi subgroup of a parabolic subgroup $P$ of $H$, where $P$ 
 is not a Borel subgroup. Then $L=XZ(L)$, where  
$X\cong  GL(2,q)$.

By Proposition \ref{tt8}, 
$\overline{\chi}_{L}=\mu|_L+\si$, where $\mu(1)=1$ and $\si$ is cuspidal. 
(Indeed, as $L/Z(L)\cong PGL(2,q)$, it follows that 
$L$ has no \ir character of degree $d$ with  $1<d\leq (q+1)/2$ for $q\neq 3$. In the
latter case $ (q-1)/2$ equals $1$.)
By Harish-Chandra reciprocity, $(\chi,\si^{\#H})=1$. In addition,
as $C_W(\si)=1$, the Harish-Chandra series of $\si$ consists of a single character, see
\cite[p. 678]{CR2}. 
(Note that in our case the ramification group defined in 
\cite[p. 678]{CR2} is trivial as $N_H(L)=L$.)
 Therefore, $\si^{\#H}$ is irreducible, and  
$\si^{\#H}(1)=q^3-1$. As $\chi(1)=q^3$, the first statement of the lemma  follows. 

Let $\si$ be any cuspidal character of $L$. 
It is easy to compute that $\si^{\# H}(u)=-1$ for every unipotent element $u\neq 1$. So 
$\mu+\si^{ \# H}$ is $Syl_p$-vanishing for every linear character $\mu$ of $H$.
Suppose that $\chi=1_H+\si^{\#H}$ is $p$-vanishing. Then, 
by Corollary \ref{ke1}, $\si$ is trivial on $Z(L)$. 
Conversely, let $g=su\in H$, where $s\neq 1$ is semisimple, 
$u$ is unipotent and $su=us$.
Then $g$ is conjugate to an element of $L$, and if $g\in L$ then  $s\in Z(L)$.
Therefore, $\si (su)=\si(u)=-1$. Let $C$ be the conjugacy class of $su$ in $H$.
One observes that $C\cap P$ forms a conjugacy class in  $P$. Let $\tau$ be the 
inflation of $\si$ to $P$. Then $\tau^H=\si^{\#H}$ and $\tau(su)=-1$. By the induced character 
formula, we have $\tau^H(su)=-1$, which proves the second statement. 
\end{proof}

\begin{corol}\label{s23} 
Let $G=SL(3,q)$,  and let $\chi$ be a
$p$-vanishing character of degree $q^3$. Then either $\chi=St$ or
$(\chi, 1_G)=1$. In addition, no \ir constituent of $\chi$ is cuspidal.
\end{corol}

Proposition \ref{sg3} can be compared with the results  
obtained with use of 
\textsc{chevie}. 
Below  are some details.

Let $H=GL(3,q)$. Following  \textsc{chevie}, we partition $\Irr H$ into 
$8$ sets ${\bf X}_h$, $1\leq h\leq 8$, see Table \ref{TGL3}.1-A. 
Applying procedure \ref{procedure} we obtain in step (1) the $Syl_p$-decompositions 
of Table \ref{TGL3}.1-B and $\Delta_H= \{  1, 2, 7, 8  \}$. In step (2) we get a unique
character
$\psi_1=\chi_{1}+\chi_{7}$.  Now, in step (3) we obtain all  $Syl_p$-vanishing characters
of $H$ of degree 
$|H|_p$ (see Table \ref{TGL3}.1-C). In particular, every such character has a one-dimensional \ir constituent.

\begin{lemma}\label{LGL3}
Let $H=GL(3,q)$. Then $H$ admits exactly $(q-1)(q-\delta )/2 $ reducible 
$p$-vanishing characters $\psi$ of degree $|H|_p$,  all of type
$$\psi=\chi_1(n_1)\cdot (1_H+\chi_7(q-1,n_7)),$$
where $\delta=\frac{1-(-1)^q}{2}$, $n_1\in \ZZ_{q-1}$ and 
$n_7=(q-1) n'_7$, with $ n'_7 \in \ZZ_{q+1}$ and $q+1$ does not divide $ n'_7$.
\end{lemma}

\begin{proof}
Modulo linear characters of $H$, take $\psi_1=1_H + \chi_{7}(m,n)$. 
Then, $(q-1)\mid (m+n)$ and  it suffices to look at the classes $C_5(a,b)$,
 where we have $\zeta_1^{m(b-a)}=1$. 
Thus, $m=q-1$ and $n=(q-1) n'$, where $ n' \in \ZZ_{q+1}$ and $q+1$ does not divide $ n'$.
\end{proof}

Now, let $G=SL(3,q)$. We have to distinguish two cases, depending on whether $3$ 
divides $q-1$ or not. However, the computations remain very similar to those for
$GL(3,q)$. 
 We provide the relevant  informations in Tables \ref{TGL3}.2-A,B,C if $3\mid (q-1)$,
and in 
Tables  \ref{TGL3}.3-A,B,C otherwise.
Thus, it can be easily proved the following

\begin{lemma}\label{LSL3}
Let $G=SL(3,q)$ and $\delta=\frac{1-(-1)^q}{2}$. Then, $G$ admits exactly $(q-\delta )/2 $ reducible $p$-vanishing characters $\psi$ of degree $|G|_p$,  all of type
$$\psi=1_G+\chi_i(a),$$
where $a=(q+1) a'$,  $i=10$ if $3\mid (q-1)$, and $i=7$,  otherwise.
\end{lemma}

\subsection{Groups $U(4,q)$ and $U(5,q)$}

First, apply Procedure \ref{procedure} to the group $H=U(4,q)$.
In step (1) we obtain  the $Syl_p$-decompositions of Table \ref{TU4}-B and 
$\Delta_H=\{ 2, 5, 7,  8, 12, 19, 23  \}$.
Note that $\chi_{19}$ and $\chi_{23}$ are
the only regular characters of this set. 
In step (2), let $\psi$ be a $Syl_p$-vanishing character of $H$ of degree $|H|_p$, whose
constituents lie in $\Delta_H$.
Since the degree of $\chi_{23}$ is greater than $|H|_p$, it follows  
that $\chi_{19}$ is a constituent of $\psi$. 
So the other constituents of $\psi$ are 
 $\{\chi_h \mid h =  2, 5, 7, 8, 12 \}$. 
Using  the classes $C_3(q+1)$ and $C_5(q+1)$, we observe  that  $\psi$ must have at least
$2$ constituents 
of type $\chi_7$ and a constituent of type $\chi_8$. Now, the character 
values at the class $C_4(q+1)$ imply that either $\chi_2$ is a constituent of $\psi$,
 or both $\chi_8$ and $\chi_{12}$ are other constituents of $\psi$. In the latter case
 we get a contradiction on the class $C_3(q+1)$. Therefore,  we get a unique
character  $\psi_1=\chi_2+\chi_7+\chi_7'+\chi_8
+\chi_{19}$,  where $\chi_7,\chi'_{7}\in {\bf X}_{7}$ are not necessarily
distinct. 
Finally, in step (3) we determine all 
$Syl_p$-vanishing characters of 
$H$ of degree $|H|_p$ (see Table \ref{TU4}-C).

\begin{lemma}\label{LU4}
Let $H=U(4,q)$. Then neither $H$ nor $G=H'$ has a reducible $p$-vanishing character of
degree $|H|_p$.
\end{lemma}

\begin{proof}
Suppose the contrary and assume $q\geq 4$ (the cases $q=2,3$ can be solved with
\textsc{magma}).
In view of Table \ref{TU4}-C, up to linear characters of $H$, 
we have to consider the \f four cases:
\begin{itemize}
\item[(1)]  $\psi_{1}=\chi_{2}(0) + \chi_{7}(k_{7,1},k_{7,2}) +
\chi_{7}(k_{7,1}',k_{7,2}') + \chi_{8}(k_{8,1},k_{8,2}) + 
\chi_{19}(k_{19,1},k_{19,2},$ $k_{19,3}, k_{19,4})$. 
On the class $C_{18}(1,2)$ we obtain $\xi_1^{k_{8,1}-k_{8,2}}=1$. This implies that $q+1$
divides $k_{8,1}-k_{8,2}$, a contradiction.
\item[(2)]  $\psi_{2}=\chi_{2}(0) + \chi_{7}(k_{7,1},k_{7,2}) +
\chi_{16}(k_{16,1},k_{16,2},k_{16,3}) + 
\chi_{19}(k_{19,1},k_{19,2}, k_{19,3}, k_{19,4})$.
On the classes $C_{16}(q+1,i_2,-i_2)$ we have 
$$2+Re(\xi_1^{(k_{7,1}-k_{7,2})i_2})+Re(\xi_1^{(k_{16,1}-k_{16,2})i_2})+Re(\xi_1^{(k_{16,1
}-k_{16,3})i_2})+$$
$$+Re(\xi_1^{(k_{16,2}-k_{16,3})i_2})=
Re(\xi_1^{(k_{19,1}-k_{19,2})i_2})+Re(\xi_1^{(k_{19,1}-k_{19,3})i_2})+$$
$$+Re(\xi_1^{(k_{19,1}-k_{19,4})i_2})+
+Re(\xi_1^{(k_{19,2}-k_{19,3})i_2})+Re(\xi_1^{(k_{19,2}-k_{19,4})i_2})+$$
$$+Re(\xi_1^{(k_{19,3}-k_{19,4})i_2}).$$
This holds for every $i_2\in \ZZ_{q+1}$ such that $q+1$ does not divide $2i_1$. If $q\geq
13$, this implies that $q+1$ divides $k_{19,i}-k_{19,j}$ for some $i\neq j$, a
contradiction.
For $4\leq q\leq 11$, it can be proved directly that this condition for $i_1=1,\ldots,
\lfloor(q-1)/2\rfloor$ is never satisfied.
\item[(3)] $\psi_{3}=\chi_{2}(0) + \chi_{7}(k_{7,1},k_{7,2}) +
\chi_{15}(k_{15,1},k_{15,2},k_{15,3})$. On the class $C_8(1,q-2)$ we have
$\xi_1^{k_{7,1}-k_{7,2}}=\xi_1^{ 2k_{15,1}-k_{15,2}-k_{15,3}}$  and on $C_{18}(1,2)$ we
have
$1=\xi_1^{2k_{15,1}-k_{15,2}-k_{15,3}}.$
So $\xi_1^{k_{7,1}-k_{7,2}}=1$, and hence $(q+1)\mid (k_{7,1}-k_{7,2})$, a contradiction.
\item[(4)]  $\psi_4=\chi_2(0)+\chi_6(k_1,k_2)$. Using the class $C_7(1,q-2)$ we have
$\xi_1^{k_2-k_1}=1$. It follows that $(q+1)\mid (k_2-k_1)$, a contradiction with the
conditions on the parameters.
\end{itemize} 
Since all the conjugacy classes we have considered belong to $H'$, it follows from
Proposition \ref{lu6} that also $H'$  has no $p$-vanishing character of degree $|H|_p$.
\end{proof}

Next, apply Procedure \ref{procedure} to $H=U(5,q)$.
In step (1) we obtain the $Syl_p$-decompositions of Table \ref{TU5}-B  and $\Delta_H=\{ 2,
5, 6, 7, 16, 20, 21, 37, 42, 43   \}$. 
Note that $\chi_{37}$, $\chi_{42}$ and
$\chi_{43}$ are the only regular characters in this set. 
In step (2), let $\psi$ be a $Syl_p$-vanishing character of $H$ of degree $|H|_p$, whose
constituents lie in $\Delta_H$. Observe that $\chi_{43}$ is ruled
out by degree reasons. 
\begin{itemize}
\item[(1)] Suppose that $\chi_{37}$ is a constituent of $\psi$. Then, looking at the
classes $C_3(q+1)$ and $C_7(q+1)$, 
the character $\psi$ must have at least one constituent of type $\chi_{21}$ and three
constituents of type $\chi_{20}$. With some little efforts, we get a unique
$Syl_p$-vanishing character
$\psi_1=\chi_2+\chi_5+\chi_{16}+\chi_{20}+\chi'_{20}+\chi''_{20}+\chi_{21}+\chi_{37}$,
where $\chi_{20},\chi'_{20},\chi''_{20}\in {\bf X}_{20}$ are not necessarily distinct. 
\item[(2)] Let $\chi_{42}$ be a constituent of $\psi$. Then, by degree reasons, 
the other constituents of $\psi$ must belong to the set $\{\chi_h \mid h = 5, 6, 7, 16, 21
\}$. Then, looking at the classes $C_5(q+1)$ and $C_7(q+1)$, the character $\psi$ must
have one
constituent of type $\chi_{7}$ and at least one constituent of type $\chi_{16}$. This
leads to a contradiction on the class $C_6(q+1)$.
\end{itemize}
With step (3)
we get  all  $Syl_p$-vanishing characters of $H$ of degree $|H|_p$ (see Table
\ref{TU5}-C).

\begin{lemma}\label{LeU5}
Let $H=U(5,q)$. Then, neither $H$ nor $G=H'$  admits any reducible $p$-vanishing character of degree $|H|_p$.
\end{lemma}

\begin{proof}
Suppose the contrary. The cases $q=2,3$ can be solved using \textsc{magma}. Assume $q> 3$.
First, take $\psi_{14}=\chi_2(k_{2,1})+\chi_8(k_{8,1},k_{8,2})$. Using the class
$C_9(1,q-3)$ we have $\xi_1^{k_{8,2}-k_{8,1}}=1$, a contradiction. In all the other cases,
it suffices to consider the classes $C_{34}(q+1,i_2,i_2)$. We obtain one of the two
following conditions:
\begin{itemize}
\item[(a)]  $1+\xi_1^{(k_{a,2}-k_{a,1})i_2} =\xi_1^{(k_{b,1}-k_{b,2})i_2} +\xi_1^{(k_{b,1}-k_{b,3})i_2}$, $i_2=1,2$;
\item[(b)]  $2+\xi_1^{(k_{a,2}-k_{a,1})i_2} =\xi_1^{(k_{b,1}-k_{b,2})i_2} +\xi_1^{(k_{b,1}-k_{b,3})i_2}+\xi_1^{(k_{b,1}-k_{b,4})i_2}$, $i_2=1,2,3$.
\end{itemize}
In  both the cases, this implies that $q+1$ divides $k_{b,1}-k_{b,j}$ for
some $j\neq 1$, a contradiction with the conditions on the parameters. (The case (b) for
$q=4$ needs to be done  directly.)

Since all the conjugacy classes used above 
belong to $H'=G$, if follows from
Proposition \ref{lu6} that also $G$  has no $p$-vanishing characters of degree $|G|_p$.
\end{proof}

\subsection{Groups $G=Sp(4,q)$}
Let $G=Sp(4,q)$ with $q$ even and apply Procedure \ref{procedure}.
In step (1) we get the $Syl_p$-decompositions of Table \ref{TSp4}.1-B and 
$\Delta_G=\{ 1,
2, 3, 4, 5, 9, 13, 18,
19 \}$. Note that the
regular characters belonging to this set are $\chi_{18}$ and $\chi_{19}$.
In step (2), let $\psi$ be a $Syl_p$-vanishing character of $G$ of degree
$|G|_p$, whose
constituents lie in $\Delta_G$.
\begin{itemize}
\item[(1)] Let $\chi_{18}$ be a constituent of $\psi$. Then, by degree reasons, the other
constituents of $\psi$ must belong to the set $\{\chi_h \mid h =  1, 5  \}$. In this case
we easily obtain a contradiction.
\item[(2)] Let $\chi_{19}$ be a constituent of $\psi$. Then, by degree reasons, the other
constituents of $\psi$ must belong to the set $\{\chi_h \mid h =  1, 2, 3, 4, 5, 9, 13
\}$. Looking at the class $C_5$, either $\chi_9$ or $\chi_{13}$ is a constituent of
$\psi$. In the former case, we obtain the $Syl_p$-vanishing characters
$\psi_1=\chi_1+\chi_9+\chi_{13}+\chi_{19}$ and $\psi_2=\chi_3+\chi_5+\chi_9+\chi_{19}$. In
the latter case, we obtain the characters $\psi_1$ and
$\psi_3=\chi_4+\chi_5+\chi_{13}+\chi_{19}$.
\end{itemize}
With step (3) we are able to list all the $Syl_p$-vanishing characters of
$G$ of degree $|G|_p$ (see Table \ref{TSp4}-C).

\begin{lemma}\label{LSp4.0}
Let $G=Sp(4,q)$,  $q$ even. Then $G$  has no reducible $p$-vanishing character 
of degree $|G|_p$.
\end{lemma}

\begin{proof} 
Suppose the contrary. Then we have to consider the \f cases. 

\begin{itemize}
 \item[(1)]  $\psi_1=1_G+\chi_9(k_9)+\chi_{13}(k_{13})+\chi_{19}(k_{19},l_{19})$. For
$i=1,2$, on the classes $C_{10}(i)$ we have $Re(\xi_1^{ik_{19}})+Re(\xi_1^{il_{19}})=
Re(\xi_1^{ik_{13}})+Re(\xi_1^{ik_9})$ and on the classes $C_{14}(i)$ we have
$Re(\xi_1^{(k_{19}+l_{19})i})+Re(\xi_1^{(k_{19}-l_{19})i})=Re(\xi_1^{2ik_{13}})+Re(\xi_1^{
ik_9})$. 
 Set $A= Re(\xi_1^{k_{19}})$, $B=Re(\xi_1^{l_{19}})$,  $C= Re(\xi_1^{k_{13}})$ and $D=
Re(\xi_1^{k_{9}})$:
$$\left\{\begin{array}{l} 2AB= 2C^2-1 +D \\ (2AB)^2-2(A^2+B^2)+1=4C^4-4C^2+ D^2\\
(A,B)=(C,D),(D,C)\\ 
 \end{array}\right. . $$
If $C=\frac{1}{2}$, then $2AB=D-\frac{1}{2}$. This implies that either $A=C=\frac{1}{2}$
and $B=D=D-1/2$  or  $B=C=\frac{1}{2}$ and $A=D=D-\frac{1}{2}$. So, in both the cases we
have an absurd.
Hence, we may assume $C\neq \frac{1}{2}$ and $D=\frac{2C^2-1}{2C-1}$. We get the equation
$$8(C-1)^2\Big(C^3+\frac{1}{2}C^2-\frac{1}{2}C-\frac{1}{8}\Big)=0.$$
Now, since $C\neq 1$: 
\begin{equation}\label{q7}
0=C^3+\frac{1}{2}C^2-\frac{1}{2}C-\frac{1}{8}=\Big(C-\cos(\frac{2\pi}{7}
)\Big)\Big(C-\cos(\frac{4\pi}{7})\Big)
\Big(C-\cos(\frac{6\pi}{7})\Big). 
\end{equation}
Thus, we have a solution if and only if $7$ divides $q+1$. However, this does not happen
since $q$ is even.

\item[(2)]  $\psi_2=1_G+\chi_9(k_9)+\chi_{14}(k_{14})$. On the class $C_{14}(1)$ we have
$Re(\xi_1^{k_9})=1$, a  contradiction.
\item[(3)]   $\psi_{3}= \chi_{1} + \chi_{10}(k_{10}) + \chi_{13}(k_{13})$. On the class
$C_{10}(1)$ we have $Re(\xi_1^{k_{13}})=1$, a contradiction.
\item[(4)]  Take $\psi_{4}= \chi_{3} + \chi_{5} + \chi_{9}(k_9) +
\chi_{19}(k_{19},l_{19})$. 
 On the classes $C_{14}(i)$ ($i=1,2$) we have 
$1+Re(\xi_1^{2ik_{9}})=Re(\xi_1^{i(k_{19}+l_{19})})+Re(\xi_1^{i(k_{19}-l_{19})})$, which
implies that $q+1$ divides either $k_{19}+l_{19}$ or $k_{19}-l_{19}$. In both cases, we
get a contradiction.

\item[(5)]   $\psi_{5}= \chi_{3} + \chi_{5} + \chi_{10}(k)$. On the class $C_{14}(1)$ we
have $Re(\xi_1^k)=1$, a contradiction.

\item[(6)]   $\psi_{6}= \chi_{4} + \chi_{5} + \chi_{13}(k_{13}) +
\chi_{19}(k_{19},l_{19})$. On the classes $C_{14}(i)$ ($i=1,2$) we have
$1+Re(\xi_1^{2ik_{13}})=Re(\xi_1^{i(k_{19}+l_{19})})+Re(\xi_1^{i(k_{19}-l_{19})})$, which
implies that $q+1$ divides either $k_{19}+l_{19}$ or $k_{19}-l_{19}$. In both cases, we
get a contradiction.

\item[(7)]  $\psi_{7}= \chi_{4} + \chi_{5} + \chi_{14}(k)$. On the class $C_{10}(1)$ we
have $Re(\xi_1^k)=1$, a contradiction.
\end{itemize} 
\end{proof}

Let $G=Sp(4,q)$, $q$ odd. The character table of this group is not implemented in
\textsc{chevie},
but it was described in papers by Srinivasan \cite{Sri}, Przygocki \cite{Prz} and Bleher  \cite{Ble}.
We keep the notation of \cite{Sri} for the conjugacy classes. We partition $\Irr G$ into
$49$ sets ${\bf X}_h$ as done in Table \ref{TSp4}.2-A. (Note
that this is not taken from \textsc{chevie} but again all characters in every set are
$Syl_p$-equivalent.) 

Applying Procedure \ref{procedure} we obtain in step (1) the $Syl_p$-decompositions of
Table \ref{TSp4}.2-B and $\Delta_G=\{ 6, 16, 17, 37, 38, 40, 42, 44 \}$. 
When we consider step (2), we note that the only regular characters in $\Delta_G$ are
$\chi_{16}$ and $\chi_{17}$. By
Lemma \ref{gg1}, at least one of them is a constituent of a $Syl_p$-vanishing character
$\psi$ of degree
$|G|_p$ whose
constituents lie in $\Delta_G$. Inspecting \cite{Sh},
we
observe that $\chi_{16}$ and $\chi_{17}$  are $CSp(4,q)$-conjugate. By Proposition
\ref{lu6}, both of them are constituents of $\psi$. (\itf they are constituents of two
distinct non-degenerate Gelfand-Graev characters.)
Working only with these constituents, we obtain only  two $Syl_p$-vanishing characters of degree $|G|_p$:
 $$\psi_1=\chi_6+\chi_{16}+\chi_{17}+\chi_{37}+\chi_{38}+\chi_{40}+\chi_{40}'+
 \chi_{44}$$ and $$\psi_2= \chi_{16}+ \chi_{17}+ \chi_{37}+ \chi_{38}+ \chi_{40}+
\chi_{40}'+ \chi_{40}''+ \chi_{42}.$$
Now step (3) produces all  $Syl_p$-vanishing characters of $G$ of degree $|G|_p$ (see
Table \ref{TSp4}.2-C).

\begin{lemma}\label{LSp4.1}
Let $G=Sp(4,q)$,  $q$ odd. Then $G$ admits reducible $p$-vanishing characters $\psi$ of degree $|G|_p$ if and only if $7\mid (q+1)$. In this case, there are exactly $3$ characters $\psi$, all of type
$$\psi=1_G+\chi_4(a,b)+\chi_6(c)+\chi_{10}(d),$$
where
$$(a,b,c,d)  \in   \{ (T, 3T, T, 3T ),  (T,2T,2T,T), (2T,3T,3T,2T)\}
$$
and $T=(q+1)/7$.
\end{lemma}

\begin{proof}
Most of the $Syl_p$-vanishing characters of $G$ can be easily ruled out. We present here some examples for the most interesting cases.

\begin{itemize}
\item[(1)] Let $\psi_1=1_G+\chi_4(k_4,l_4)+\chi_6(k_6)+\chi_{10}(k_{10})$. On the class $D_{31}$ we have $(-1)^{k_4}+(-1)^{l_4}=(-1)^{k_6}+(-1)^{k_{10}}$
On the class $D_{21}$ we have
$(q-1)(-1)^{k_6} +q=(q-1)((-1)^{k_4}+(-1)^{l_4} )+(-1)^{k_{10}}$.
Thus, $k_{10}$ is even. 
On the class $B_7(i)$ we have 
$Re(\xi_1^{(k_4+l_4)i})+Re(\xi_1^{(k_4-l_4)i})=Re(\xi_1^{2ik_6})+Re(\xi_1^{ik_{10}})$
and on the class $C_{21}(i)$:
$Re(\xi_1^{k_4i})+Re(\xi_1^{l_4i})=Re(\xi_1^{ik_6})+Re(\xi_1^{ik_{10}})$.

We set $A= Re(\xi_1^{k_4})$, $B=Re(\xi_1^{l_4})$,  $C= Re(\xi_1^{k_{6}})$ and $D= Re(\xi_1^{k_{10}})$, obtaining
$$\left\{\begin{array}{l} 2AB= 2C^2-1 +D \\ (2AB)^2-2(A^2+B^2)+1=4C^4-4C^2+ D^2\\
(A,B)=(C,D),(D,C)\\ 
 \end{array}\right. . $$
If $C=\frac{1}{2}$, then $2AB=D-\frac{1}{2}$. This implies that either $A=C=\frac{1}{2}$ and $B=D=D-1/2$  or  $B=C=\frac{1}{2}$ and $A=D=D-\frac{1}{2}$. So, in both the cases we have an absurd.
Hence, we may assume $C\neq \frac{1}{2}$ and $D=\frac{2C^2-1}{2C-1}$. We get the equation
$$8(C-1)^2\Big(C^3+\frac{1}{2}C^2-\frac{1}{2}C-\frac{1}{8}\Big)=0,$$
and so Equation \eqref{q7} which implies $C=\cos(2t\pi/7)$, $t=1,2,3$, and
$D=\cos(6t\pi/7)$. Thus,
\begin{eqnarray*}
(A,B,C,D)& \in   &\big\{
\Big(\cos\frac{2\pi}{7},\cos\frac{6\pi}{7},\cos\frac{2\pi}{7},\cos\frac{6\pi}{7}
\Big),\\ {} &  &
\Big(\cos\frac{6\pi}{7},\cos\frac{2\pi}{7},\cos\frac{2\pi}{7},\cos\frac{6\pi}{7}
\Big),\\ {} &  &
\Big(\cos\frac{4\pi}{7},\cos\frac{2\pi}{7},\cos\frac{4\pi}{7},\cos\frac{2\pi}{7}
\Big),\\
{} &  &
\Big(\cos\frac{2\pi}{7},\cos\frac{4\pi}{7},\cos\frac{4\pi}{7},\cos\frac{2\pi}{7}
\Big),\\ 
{} &  &
\Big(\cos\frac{6\pi}{7},\cos\frac{4\pi}{7},\cos\frac{6\pi}{7},\cos\frac{4\pi}{7}
\Big),\\ {} &  &
\Big(\cos\frac{4\pi}{7},\cos\frac{6\pi}{7},\cos\frac{6\pi}{7},\cos\frac{4\pi}{7}
\Big)\big\}.\\
\end{eqnarray*}
This means, denoting $T=\frac{q+1}{7}$, that
\begin{eqnarray*}
(k_4, l_4, k_6,k_{10} )&\in&   \{(T, 3T, T, 3T), (3T, T, T, 3T), (2T, T,2T,T)\\ {}
&  & (T,2T,2T,T), (3T,2T,3T,2T), (2T,3T,3T,2T)\}.
\end{eqnarray*}
It is clear now that $Re(\xi_1^{(k_4+l_4)i})+Re(\xi_1^{(k_4-l_4)i})=Re(\xi_1^{2ik_6})+Re(\xi_1^{ik_{10}})$ and $Re(\xi_1^{i k_4})+
Re(\xi_1^{i l_4})= Re(\xi_1^{ik_6})+Re(\xi_1^{ik_{10}})$ for all $i$. \itf  $\psi$ is a $p$-vanishing character if and only if $7$ divides
$q+1$. In this case, since $\chi_4(a,b)=\chi_4(b,a)$, $G$ admits exactly three $p$-vanishing characters of degree $|G|_p$ as described in the statement.

\item[(2)] Take 
$\psi_{16}=1_G +\chi_6(k_6) +\chi_{10}(k_{10})+\chi_{16}(k_{16})+\chi_{17}(k_{17})$.
Looking at classes $C_{21}(1)$ and $C_{22}(1)$ we 
obtain $Re(\xi_1^{k_{16}} )=Re(\xi_1^{k_{17}})$. On the classes $B_7(i)$ ($i=1,2)$ we get
$Re(\xi_1^{2ik_{6}})+Re(\xi_1^{ik_{10}})=2(-1)^iRe(\xi_1^{ik_{16}})$. Hence $Re(\xi_1^{k_{10}})=-1$, a contradiction.

\item[(3)]   $\psi_{29}=1_G+\chi_6(k_6)+\chi_{16}(k_{16})+\chi_{23}+\chi_{24}$. On the classes $C_{21}(1)$ and $C_{22}(1)$ we obtain $Re(\xi_1^{k_{16}})=1$, a contradiction.
\end{itemize}
\end{proof}

\begin{remar}
It can be shown that the $p$-vanishing characters $\psi_1$ of $Sp(4,q)$, described above,
can be extended to $p$-vanishing characters of $CSp(4,q)$. In Tables \ref{TSp4}.3-A,B,C we
report information about irreducible characters, $Syl_p$-decompositions and
$Syl_p$-vanishing characters of $CSp(4,q)$, $q$ odd. 
\end{remar}

\subsection{Groups ${}^3D_4(q)$, $G_2(q)$ and ${}^2F_4(q^2)$}

Let $G={}^3D_4(q)$. The character table of $G$ depends on the parity of $q$.  Here we
describe in detail only the case $q$ odd; the case $q$ even is very similar (see Tables
\ref{T3D4}-1.A,B,C).

So, let $q$ be odd and apply Procedure \ref{procedure}.
In step (1) we obtain the $Syl_p$-decompositions of Table \ref{T3D4}.2-B  and
$\Delta_G=\{
 1, 2, 3, 4, 5, 6, 7, 9, 15, 16, 21, 24, 25, 30,$ $33  \}$.
 Note that, for these $h$, the only regular characters $\chi_h$ are with $h=30,33$.
 In step (2), let $\psi$ be a $Syl_p$-vanishing character of $G$ of degree $|G|_p$, whose
constituents lie in $\Delta_G$.
First, some computation  shows that $(\chi_{30},\psi)=0$.

Next, assume that $\chi_{33}$ is a constituent of $\psi$.
Looking at the values of $\chi_h$ at the class $C_7$, we observe that 
$\chi_{21}$ or $\chi_{24}$ is a constituent of  $\psi$. 
In the former case, we get the character $\psi_1= \chi_{1} + \chi_{21} + \chi_{24} +
\chi_{25} + \chi_{33}$. 
In the latter case, we obtain  the characters $\psi_{2}= \chi_{5} + \chi_{6} + \chi_{7} +
\chi_{21} + \chi_{33}$
and $\psi_{3}= \chi_{2} + \chi_{5} + \chi_{6} + \chi_{24} + \chi_{25} + \chi_{33}$.
With step (3) we determine all  $Syl_p$-vanishing characters of 
$G$ of degree $|G|_p$ (see Table \ref{T3D4}.2-C).

\begin{lemma}\label{L3D4}
The group $G={}^3D_4(q)$ has no reducible $p$-vanishing character of degree $|G|_p$.
\end{lemma}

\begin{proof} Suppose the contrary. Assume that $q$ is odd (the case of $q$ even is
similar).

\begin{itemize}

\item[(1)] Take $\psi_{1}= 1_G + \chi_{21}(k_{21}) + 
\chi_{24}(k_{24}) + \chi_{25}(k_{25}) + 
\chi_{33}(k_{33},l_{33})$. On the classes $C_{22}(a)$ ($a=1,2,3$) we have
$$1+Re(\xi_1^{ak_{21}})+Re(\xi_1^{2ak_{21}})
=Re(\xi_1^{al_{33}}) +Re(\xi_1^{a(k_{33}+l_{33} )})+Re(\xi_1^{a(k_{33}+2l_{33}) }).$$ 
This implies that $q+1$ divides at least one of the numbers $l_{33}$, $k_{33}+l_{33}$,
$k_{33}+2l_{33}$, which violates the description of the parameter set $I_h$.

\item[(2)] Take $\psi_{2}= \chi_{5} + \chi_{6} + \chi_{7} + \chi_{21}(k_{21}) +
\chi_{33}(k_{33},l_{33})$. In this case, we have the same contradiction as in the previous
case.

\item[(3)] Take $\psi_{3}= \chi_{2} + \chi_{5} + \chi_{6} + \chi_{24}(k_{24}) +
\chi_{25}(k_{25}) + \chi_{33}(k_{33},l_{33})$. 
On the class $C_{22}(1)$  we obtain
$Re(\xi_1^{l_{33}})+Re(\xi_1^{k_{33}+l_{33}})+
Re(\xi_1^{k_{33}+2l_{33}})=3$. This leads to a contradiction.

\item[(4)] Take $\psi_{4}= 1_G + \chi_{21}(k_{21}) + \chi_{27}(k_{27}) +
\chi_{33}(k_{33},l_{33})$. On the class $C_{22}(a)$ we get
$$Re(\xi_1^{ak_{21}})+Re(\xi_1^{2ak_{21}})+Re(\xi_1^{ak_{27}})
=Re(\xi_1^{al_{33}}) +Re(\xi_1^{a(k_{33}+l_{33} )})+Re(\xi_1^{a(k_{33}+2l_{33}) }).$$ 
Using  the classes $C_{25}(1)$ and $C_{26}(1)$ we get
$Re(\varphi_6^{k_{27}})=Re(\varphi_6^{k_{33}})=1$.
Furthermore, setting $k_{27}=(q^2-q+1) k'_{27}$ and   
$k_{33}=(q^2-q+1) k'_{33}$, on the classes $C_{28}(a)$ we have
$$Re(\xi_1^{ak_{21}})+Re(\xi_1^{a k'_{27}})+Re(\xi_1^{2a k'_{27}})= Re(\xi_1^{a
k'_{33}})+Re(\xi_1^{a(k'_{33}+l_{33})})+Re(\xi_1^{a(2 k'_{33}+l_{33} )}).$$
With some work, this yields a contradiction.

\item[(5)] Take  $\psi_{5}= \chi_{1} + \chi_{21}(k_{21}) + \chi_{28}(k_{28})$. On the
classes $C_{22}(a)$ ($a=1,2$) we obtain
$1+Re(\xi_1^{ak_{28}})=Re(\xi_1^{ak_{21}}) +Re(\xi_1^{2ak_{21}})$, 
which implies $(q+1)/2 \mid k_{21}$, a contradiction.

\item[(6)] Take $\psi_{6}= 1_G + \chi_{22}(k_{22}) + \chi_{24}(k_{24}) +
\chi_{25}(k_{25})$. On the class $C_{22}(1)$ we obtain the equality
$Re(\xi_1^{k_{22}})=1$, which is a contradiction.

\item[(7)] Take $\psi_{7}=1_G + \chi_{22}(k_{22}) + 
\chi_{27}(k_{27})$. On the class $C_{22}(1)$ we obtain
$Re(\xi_1^{k_{22}})=Re(\xi_1^{k_{27}})$ and on the class $C_{26}(1)$ we get
$Re(\varphi_6^{k_{27}})=1$. On the class $C_{28}(1)$ we have $Re(\xi_3^{q(q-1)k_{27}})=1$.
This implies that $(q^3+1)/2$ divides $k_{27}$, which is a contradiction.

\item[(8)] Take $\psi_{8}= \chi_{2} + \chi_{5} + \chi_{6} + 
\chi_{27}(k_{27}) + \chi_{33}(k_{33})$.
On the classes $C_{22}(a)$ ($a=1,2,3$) we get
$$2+Re(\xi_1^{ak_{27}})=
Re(\xi_1^{al_{33}})+Re(\xi_1^{a(k_{33}+l_{33})})+
Re(\xi_1^{a(k_{33}+2l_{33})}).$$
This implies that $q+1$ divides at least one of the numbers $l_{33}$, $k_{33}+l_{33}$,
$k_{33}+2l_{33}$, which is 
a contradiction.

\item[(9)] Take  $\psi_{9}= \chi_{2} + \chi_{5} + \chi_{6} + \chi_{28}(k_{28})$. 
On the class $C_{22}(1)$ we have $Re(\xi_1^{k_{28}})=1$, a contradiction.

\item[(10)] Take $\psi_{10}= \chi_{5} + \chi_{6} + \chi_{7} + \chi_{22}(k_{22})$. On the
class $C_{22}(1)$ we obtain $Re(\xi_1^{k_{22}})=1$, a contradiction.
\end{itemize}
\end{proof}

Now, let $G=G_2(q)$. The character table of $G$ provided by \textsc{chevie} 
distinguishes five cases, depending on the parity of $q$ and on the congruence class of
$q$ modulo $3$. All the cases are very similar, so we  provide   details only for
the case with $q$ odd,  $q \equiv -1\pmod 3$, which requires more work.

So, let $q$ be odd with $3 \mid (q+1)$ and apply Procedure \ref{procedure}.
With step (1) we obtain the $Syl_p$-decompositions of Table \ref{TG2}.5-B and $\Delta_G=\{
1, 2, 3, 5, 6, 7, 9, 11, 15, 22, 27$, $30  \}$.
Note that $\chi_{27}$ and $\chi_{30}$ are the only regular characters of this set. 

In step (2), let $\psi$ be a $Syl_p$-vanishing character of $G$ of degree $|G|_p$, whose
constituents lie in $\Delta_G$. First, assume that $\chi_{27}$ is a constituent of
$\psi$. Using  the class $C_7$, we conclude that either $\chi_{15}$ or $\chi_{22}$  occurs
as a constituent of $\psi$. 
In the former  case, we get $\psi_1= \chi_{5} + \chi_{6} + \chi_{6} + \chi_{7} + \chi_{7}
+ \chi_{9} + \chi_{15} + \chi_{27}$. In the latter case, we get 
 $\psi_2=\chi_{1} + \chi_{6} + \chi_{7} + \chi_{9} + \chi_{
15} + \chi_{22} + \chi_{27}$.
Next, assume that $\chi_{30}$ is a constituent of $\psi$. Then, with some work we obtain a
contradiction.

 Finally, in step (3) we  determine all  $Syl_p$-vanishing characters of 
$G$ of degree $|G|_p$ (see Table \ref{TG2}.5-C).

\begin{lemma}\label{LG2}
The group $G=G_2(q)$ has no reducible $p$-vanishing character of degree $|G|_p$.
\end{lemma}

\begin{proof}
Suppose the contrary. As before, we consider here only the case with $q$ odd and $3 \mid
(q+1)$. The proof for the other cases is very similar.
\begin{itemize}
\item[(1)] Take $\psi_{1}, \psi_2, \psi_3,\psi_4,\psi_5,\psi_{13}$. All these characters
do not vanish on the class $C_{14}$.

\item[(2)] Take $\psi_{6}= 1_G + \chi_{15} + \chi_{16} + \chi_{22}(k_{22}) +
\chi_{27}(k_{27},l_{27})$.
On the classes $C_{23}(i)$ ($i=1,2,3$), we have
$$1+Re(\xi_1^{ik_{22}})+Re(\xi_1^{2ik_{22}})=Re(\xi_1^{i(k_{27}+l_{27})})+Re(\xi_1^{i(k_{
27}-2l_{27})}) +Re(\xi_1^{i(2k_{27}-l_{27})}).$$
This implies that $q+1$ divides at least one of the integers $k_{27}+l_{27}$,
$k_{27}-2l_{27}$, $2k_{27}-l_{27}$, 
a contradiction.

\item[(3)] Take $\psi_{7}= 1_G+ \chi_{15} + \chi_{16} + \chi_{23}(k_{23})$. On the class
$C_{23}(1)$ we have $Re(\xi_1^{k_{23}})=1$, a contradiction.

\item[(4)] Take $\psi_{8}= 1_G + \chi_{22}(k_{22}) + \chi_{24}(k_{24}) +
\chi_{27}(k_{27},l_{27})$.
On the classes $C_{23}(i)$ we have
$$Re(\xi_1^{ik_{22}})+Re(\xi_1^{2ik_{22}})+Re(\xi_1^{3ik_{24}})=Re(\xi_1^{i(k_{27}+l_{27})
})+Re(\xi_1^{i(k_{27}-2l_{27})}) +Re(\xi_1^{i(2k_{27}-l_{27})})$$
and on the classes $C_{25}(i)$ we have
$$Re(\xi_1^{ik_{22}})+ Re(\xi_1^{ik_{24}})+
Re(\xi_1^{2ik_{24}})=
Re(\xi_1^{ik_{27}})+ Re(\xi_1^{il_{27}}) +Re(\xi_1^{i(k_{27}-l_{27})}).$$
These conditions lead to a contradiction.

\item[(5)] Take $\psi_{9}= 1_G + \chi_{22}(k_{22}) + \chi_{25}(k_{25})$. On the classes
$C_{23}(i)$ ($i=1,2$) we have
$1+Re(\xi_1^{3ik_{25}})=Re(\xi_1^{ik_{22}})+Re(\xi_1^{2ik_{22}})$. This implies that
$(q+1)/2$ divides $k_{22}$, which is  a contradiction.

\item[(6)]  Take $\psi_{10}= 1_G + \chi_{23}(k_{23}) + \chi_{24}(k_{24})$. On the classes
$C_{25}(i)$ ($i=1,2$) we obtain
$1+Re(\xi_1^{ik_{23}})=Re(\xi_1^{ik_{24}})+Re(\xi_1^{2ik_{24}})$. This implies that
$(q+1)/2$ divides $k_{24}$, a contradiction.

\item[(7)] Take $\psi_{11}= \chi_{4} + \chi_{6} + \chi_{7} + \chi_{22}(k_{22}) +
\chi_{27}(k_{27},l_{27})$.
On the classes $C_{23}(i)$ ($i=1,2,3$) we obtain
$$1+Re(\xi_1^{ik_{22}})+Re(\xi_1^{2ik_{22}})
=Re(\xi_1^{i(k_{27}+l_{27})})+Re(\xi_1^{i(k_{27}-2l_{27} )})+Re(\xi_1^{i(2k_{27}+l_{27})
}).$$ 
This implies that $q+1$ divides at least one of the numbers $k_{27}+l_{27}$,
$k_{27}-2l_{27}$, $2k_{27}-l_{27}$, which 
 contradicts the description of the parameter sets $I_h$.

\item[(8)]  Take $\psi_{12}= \chi_{4} + \chi_{6} + \chi_{7} + \chi_{23}(k_{23})$. On the
class $C_{23}(1)$ we obtain $Re(\xi_1^{k_{23}})=1$, a contradiction.

\item[(9)] Take $\psi_{14}= \chi_{5} + \chi_{6} + \chi_{7} + \chi_{15} + \chi_{16} +
\chi_{27}(k_{27},l_{27})$.
On the classes $C_{23}(1)$  we obtain
$Re(\xi_1^{k_{27}+l_{27}})+Re(\xi_1^{k_{27}-2l_{27} }) +
Re(\xi_1^{2k_{27}+l_{27} })=3$, which implies that $(q+1)$ divides $(k_{27}+l_{27})$, a
contradiction.

\item[(10)] Take $\psi_{15}= \chi_{5} + \chi_{6} + \chi_{7} + \chi_{24}(k_{24}) +
\chi_{27}(k_{27},l_{27})$. On the classes $C_{23}(i)$ ($i=1,2,3$) we obtain
$$Re(\xi_1^{i(k_{27}+l_{27})})+Re(\xi_1^{i(k_{27}-2l_{27} )})+Re(\xi_1^{i(2k_{27}-l_{27})
})=2+Re(\xi_1^{3ik_{24}}),$$
which implies that $q+1$ divides at least one of the numbers $k_{27}+l_{27}$,
$k_{27}-2l_{27}$, $2k_{27}-l_{27}$, which is  
a contradiction.

\item[(11)] Take  $\psi_{16}= \chi_{5} + \chi_{6} + \chi_{7} + \chi_{25}(k_{25})$. On the
class $C_{23}(1)$ we have $Re(\xi_1^{3k_{25}})=1$, a contradiction.
\end{itemize}
\end{proof}

Observe that the group $G=G_2(2)'$ has no reducible $Syl_2$-vanishing characters of degree
$|G|_2$. 
On the other hand, $G$ has six reducible $Syl_2$-vanishing characters of degree $2|G|_2$.
In the notation of \cite{Atl}:
\begin{itemize}
\item[(1)] $\phi_1=2\chi_{13}$;
\item[(2)] $\phi_2=2\chi_{14}$;
\item[(3)] $\phi_3=\chi_{13}+\chi_{14}$;
\item[(4)] $\phi_4=\chi_1+\chi_3+\chi_4+\chi_9+\chi_{11}$;
\item[(5)] $\phi_5=\chi_1+\chi_3+\chi_5+\chi_8+\chi_{12}$;
\item[(6)] $\phi_6=\chi_1+\chi_3+\chi_6+\chi_8+\chi_9$.
\end{itemize}
The three characters $\phi_1,\phi_2,\phi_3$ are $2$-vanishing, whereas the other three
characters are not. This can be easily proved
using  the character table of $G$.

\medskip

Finally, let $G={}^2F_4(q^2)$, with $q^2=2^{2m+1}$ and $m\geq 0$. The groups
${}^2F_4(q^2)$ require a much larger amount of computations. Note that the
official distribution of \textsc{chevie} contains only a partial character table of these
groups. In this case, our computations are based on the character table kindly provided by
F. Himstedt (constructed for \cite{HH}). This table is available in \textsc{chevie} format
and so we may apply the same argument as in the previous cases.

Applying Procedure \ref{procedure}, in step (1) we obtain the $Syl_p$-decompositions of
Table \ref{T2F4}-B and $\Delta_G=\{1, 2, 3, 4, 5, 6, 7, 8, 9, 10, 11, 12, 13, 14,
15, 16, 17,
18, 19, 20,   29, 30,$ $38, 40, 41, 42, 43, 44,  46, 47, 48, 51\}$.
Considering step (2),  let $\psi$ be a $Syl_p$-vanishing character of degree $|G|_p$ whose
constituents lie in $\Delta_G$. If $(\chi_{51},\psi)>0$, we obtain two characters $\psi$:
\begin{eqnarray*}
\psi_1 & =&  \chi_{2} + \chi_{4} + \chi_{10} + \chi_{11} + \chi_{13} + 
\chi_{17} + \chi_{20} + \chi_{46} + \chi_{47} + \chi_{47} + \chi_{48} + \chi_{51};\\ 
\psi_2 & =&  \chi_{3} + \chi_{4} + \chi_{10} + \chi_{12} + \chi_{14} + \chi_{17} +
\chi_{19} + \chi_{46} + \chi_{47} + \chi_{48} + \chi_{48} + \chi_{51}.
\end{eqnarray*}
On the other side, if  $\chi_{38}$, $\chi_{40}$ or $\chi_{41}$ is a constituent of $\psi$,
we get a
contradiction.
With step (3)  we  determine all  $Syl_p$-vanishing characters of $G$ of degree $|G|_p$
(see Table \ref{T2F4}-C).

\begin{lemma}\label{L2F4}
The group $G={}^2F_4(q^2)$ has no reducible $p$-vanishing character of degree $|G|_p$.
\end{lemma}

\begin{proof}
Suppose the contrary. According with Table \ref{T2F4}-C, there are four 
cases. 
\begin{itemize}

\item[(1)] Take  $\psi_1= \chi_{2} + \chi_{4} + \chi_{10} + \chi_{11} + \chi_{13} +
\chi_{17} + \chi_{20} + \chi_{46}(k_{46}) + \chi_{47}(k_{47})$ + $\chi_{47}(k_{47}') +
\chi_{48}(k_{48
}) + \chi_{51}(k_{51},l_{51})$. 
Comparing the values at the classes $C_{42}(a)$ and $C_{43}(a)$ ($a=1,\ldots,4$)
we obtain
$$Re((\varphi_8')^{(aqk_{47}\sqrt{2})/2})+
Re((\varphi_8')^{(ak_{47}(q\sqrt{2}+2)/2)})+
Re((\varphi_8')^{(aqk'_{47}\sqrt{2})/2})+$$
$$Re((\varphi_8')^{(ak'_{47}(q\sqrt{2}+2)/2)})
=2+
Re((\varphi_8')^{(aqk_{48}\sqrt{2})/2})+
Re((\varphi_8')^{(ak_{48}(q\sqrt{2}+2)/2)}).
$$
This implies that either $Re((\varphi_8')^{(qk_{47}\sqrt{2})/2})=1$ or
$Re((\varphi_8')^{(k_{47}(q\sqrt{2}+2)/2)})=1$. In both the cases, we obtain that
$q^2+q\sqrt{2}+1$ divides $k_{47}$, which is a contradiction.

\item[(2)] Take $\psi_2$. Here, we get a contradiction similar to  that obtained for
$\psi_1$.

\item[(3)] Take  $\psi_3=\chi_{2} + \chi_{4} + \chi_{10} + \chi_{11} + \chi_{13} +
\chi_{17} + \chi_{20} + \chi_{47}(k_{47}) + \chi_{49}(k_{49},l_{49})$.
Comparing the values at the classes  $C_{42}(1)$ and $C_{43}(1)$
we obtain
$$Re((\varphi_8')^{(qk_{47}\sqrt{2})/2})+
Re((\varphi_8')^{(k_{47}(q\sqrt{2}+2)/2)})=2.
$$
This implies $Re((\varphi_8')^{(qk_{47}\sqrt{2})/2})=1$, and 
hence $q^2+q\sqrt{2}+1$ divides $k_{47}$, which is a contradiction.
\item[(4)] Take $\psi_4$. Here, we get a contradiction similar to  that obtained for
$\psi_3$.
\end{itemize}
\end{proof}

Observe that the group $G={}^2F_4(2)'$ has no reducible $Syl_2$-vanishing characters of
degree $|G|_2$. 
On the other hand, it has four reducible $Syl_2$-vanishing characters of degree $2|G|_2$.
In the notation of \cite{Atl}:
\begin{itemize}
\item[(1)] $\phi_1=2\chi_{21}$;
\item[(2)] $\phi_2=2\chi_{22}$;
\item[(3)] $\phi_3=\chi_{21}+\chi_{22}$;
\item[(4)] $\phi_4=\chi_2+\chi_3+\chi_4+\chi_5+2\chi_6+\chi_{10}+\chi_{11}+ 
\chi_{16}+\chi_{17}+\chi_{20}$.
\end{itemize}
The three characters $\phi_1,\phi_2,\phi_3$ are $2$-vanishing, whereas the character
$\phi_4$ is not. This can be easily proved by 
using  the character table of $G$.

\section{The trivial constituent of a $p$-vanishing character}

Our next target   is to prove that $1_G$ occurs as a constituent of
every  reducible $p$-vanishing character of degree $|G|_p$, if $G$ is quasi-simple and admits such characters. 
This turns out to be a rather non-trivial task. 
We observe that for some groups there is a stronger version of this result, 
see \cite{Z}, Proposition 6.2 and Remark following it: 
 
\begin{propo} \label{n12}
Let $G=SL(n,q)$, $n>2$, or $G=E_i(q)$, $i=6,7,8$.
Let $\chi\neq St$ be a $Syl_p$-vanishing character of $G$ of
degree $|G|_p$. Then $(\chi,1_G)=1$. 
\end{propo}

Note that the above statement is not valid for symplectic and special unitary groups, as we
already seen.

The groups of rank $2$ have been analyzed in the previous section. We  summarize the
results about these groups in the following.

\begin{propo}\label{l2}
Let $G$ be a Chevalley group of BN-pair rank $\ell=2$. Let $\chi\neq St$ be a
$p$-vanishing character of $G$ of degree $|G|_p$. Then  
$(\chi,1_G)=1$ and either $G =SL(3,q)$ or
 $G=Sp(4,q)$ with $7\mid (q+1)$.
\end{propo}

\begin{proof}
See Lemmas \ref{LSL3}, \ref{LU4}, \ref{LeU5}, \ref{LSp4.0}, \ref{LSp4.1}, \ref{L3D4},
\ref{LG2} and \ref{L2F4}.
\end{proof}

In order to deal with groups with BN-pair rank $\ell>2$, we begin with some technical results.

\begin{propo}\label{st0}
Let $\chi$ be a character of $G$. Set  $\lambda=(\Tr{L_J }{\chi})|_{G_J}$ and $\lambda'=(\Tr{L_{K} }{\chi})|_{G_{K}}$ with $J\cup K=I_\ell$.

\begin{itemize}
\item[(1)] Suppose that $\lambda=St_{G_J}$ and $(\Tr{T_K}{\lambda'}, 1_{T_K})_{T_K}>0$. Then 
$\Tr{L_J}{\chi}=St_{L_J}$. 

\item[(2)]  Suppose that $\Tr{T}{\chi}=1_{T}$ and that $(\lambda, 1_{G_J})_{G_J}>0$. Then $(\Tr{L_J}{\chi}, 1_{L_J})_{L_J}=1$.

\item[(3)] Suppose that $(\lambda, 1_{G_J})_{G_J}>0$ and $(\lambda',
1_{G_K})_{G_K}>0$. In addition, suppose that $\Tr{T}{\chi}$ is an
\ir character (this holds if $\chi$ is a $Syl_p$-vanishing
character of degree $|G|_p)$. Then, $1_{L_J}$ (resp. $1_{L_K})$ is
a constituent of
$\Tr{L_J}{\chi}$ (resp. $\Tr{L_K}{\chi})$ and $\overline{\chi}_T=1_T$. 
\end{itemize}
\end{propo}

\begin{proof} 
(1) Note that $1_{T_J}$ coincides with  the truncation of $St_{G_J}$ to $T_J$. By Lemma \ref{dk3}(1),
 $\Tr{T}{\chi}$
 is a linear character (as so is $\beta_J$), and then $\Tr{T_K}{\lambda'}$ is a linear character. So
$(\Tr{T}{\chi})|_{T_K}=1_{T_K}$, again by Lemma \ref{dk3}. Therefore,  $\Tr{T}{\chi}=1_{T}$
by Lemma \ref{dj2}. 

Since $\lambda=St_{G_J}$, the character $\Tr{L_J}{\chi}$ is irreducible of degree $|L_J|_p=|G_J|_p$,
and by Lemma \ref{d0d} this character is $St_{L_J}$.

(2) By Frobenius reciprocity,
$(\lambda,1_{G_J})_{G_J}=(\Tr{L_J}{\chi},1_{G_J}^{L_J})_{L_J}>0$. 
Since $L_J/G_J$ is abelian, this implies that $\Tr{L_J}{\chi}$ has an irreducible constituent 
$\mu$, say, such that $\mu(g)=1$ 
for all $g\in G_J$. We  show that actually $\mu=1_{L_J}$. Indeed, the character 
$\Tr{T}{\chi}=1_{T}$ contains the character $\Tr{T}{\mu}$, and, obviously, 
$\Tr{T}{\mu}=\mu|_{T}$. Therefore,
$\Tr{T}{\mu}=\mu|_{T}=1_{T}$.

As $L_J=G_JT$, it follows that  $\mu=1_{L_J}$. Furthermore,
since $\Tr{T}{\chi}=1_{T}$, we have $(\Tr{L_J}{\chi},
1_{L_J})_{L_J}=1$.

(3) First  note that if $\chi$ is a $Syl_p$-vanishing character of
 degree $|G|_p$, then $(\Tr{T}{\chi})(1)=(\chi|_U,1_U)=1$. Now,
 since $1_{G_J}$ is a constituent of $\lambda$, it follows that
 $1_{U\cap L_J}$ is a constituent of $\lambda|_{U\cap L_J}$. This
 implies that $1_{T_J}$ is a constituent of $\Tr{T_J}{\lambda}$.
 By Lemma \ref{dk3}, $\Tr{T_J}{\lambda}$
coincides with  the linear character $(\Tr{T}{\chi})|_{T_J}$. So
$\Tr{T_J}{\lambda}=1_{T_J}$ and, similarly,  $\Tr{T_K}{\lambda}=1_{T_K}$. By Lemma
\ref{dj2}, $\Tr{T}{\chi}=1_{T}$. Now the result follows from  the previous item. 
\end{proof}

\begin{remar} It follows from $(1)$  that if  $\lambda=St_{G_J}$
and $\lambda'=St_{G_K}$ then $\Tr{L_J}{\chi}=St_{L_J}$.
\end{remar}

\begin{corol}\label{u98-1}
Let $G$ be a group of BN-pair rank $\ell>2$, and let $J, K$ be distinct subsets of $ I_\ell$ such that $J\cup K=I_\ell$ and $J\cap
K\neq \emptyset$. 
Let $\chi$ be a $Syl_p$-vanishing character of $G$ and set $\lambda=(\Tr{L_J }{\chi})|_{G_J}$ and 
$\lambda'=(\Tr{L_{K} }{\chi})|_{G_{K}}$.  If $\lambda=St_{G_J}$ then $(\lambda',1_{G_K})=0$.
\end{corol}

\begin{proof}  
Assume $(\lambda',1_{G_K})>0$. 
By Proposition \ref{st0}(1), $\overline{\chi}_{L_J}=St_{L_J}$. Hence,
$\overline{\chi}_{L_{J\cap K}}=St_{L_{J\cap K}}$ and $\overline{\chi}_T=1_T$. Now,  by
Proposition \ref{st0}(2) we obtain $(\overline{\chi}_{L_K},1_{L_K})=1$ and so
$\overline{\chi}_{L_{J\cap K}}$ contains a linear character, a contradiction. 
\end{proof}

 We may prove now the following.

\begin{theo} \label{nh2}
Let $G$ be a Chevalley group and $\chi$ be a $p$-vanishing character of $G$ of degree $|G|_p$. Then either $\chi=St$ or $(\chi,1_G)=1$.
\end{theo}

\begin{proof}
We prove this theorem by induction on the BN-pair rank $\ell$ of  $G$. We already proved
that this theorem is true for $\ell=1$ (Proposition  \ref{r1g}) and for $\ell=2$
(Proposition \ref{l2}). Assume $\ell>2$. Take two subsets $J,K$ of $I_\ell$ such that
$I_\ell=J\cup K$, $J\cap K\neq \emptyset$ and $|K|=2$.
Set $\lambda=(\overline{\chi}_{L_J})|_{G_J}$ and
$\lambda'=(\overline{\chi}_{L_K})|_{G_K}$. Hence, by the induction hypothesis, we have
three possibilities:
\begin{itemize}
\item[(a)] $\lambda=St_{G_J}$ and $\lambda'=St_{G_K}$;
\item[(b)] either $\lambda=St_{G_J}$ and $(\lambda',1_{G_K})=1$ or
$(\lambda,1_{G_J})=1$ and $\lambda'=St_{G_K}$;
\item[(c)] $(\lambda,1_{G_J})=1$ and $(\lambda',1_{G_K})=1$.
\end{itemize}
If (a) holds then $\chi=St$ by Propositions \ref{tt8} and  \ref{st0}(1). Case (b) is
excluded by Corollary \ref{u98-1}. 
If (c) holds, then 
$(\overline{\chi}_{L_J},1_{L_J})=(\overline{\chi}_{L_K},1_{L_K})=1$ by Proposition
\ref{st0}(3). Hence, Proposition \ref{tt8} implies $(\chi,1_G)=1$. 
\end{proof}

\section{Some Chevalley groups of higher rank}\label{rk3}

The results of the previous sections are insufficient to start induction. 
To set up the basis of induction, we  additionally need to perform 
computational work for $H\in\{GL(n,q), n=4,5,6, \,CSp(6,q)\}$ in spirit of that done for
groups of rank $2$. 
Our computations  heavily use material available in the
\textsc{chevie} package.
Note that \textsc{chevie}  does not provide  information for groups $H=Spin(n,q)$, $n=7,9,11$ and 
$H=Spin^+(n,q)$, $n=8,10$, so for these groups $p$-vanishing characters of degree $|H|_p$ remain unknown.

\subsection{Groups $H=GL(4,q)$, $GL(5,q)$ and $GL(6,q)$}

We apply Procedure \ref{procedure} for these three groups.
If $H=GL(4,q)$, we obtain in step (1) the $Syl_p$-decompositions of Table \ref{TGL4}-B and
$\Delta_H=\{3, 4, 5, 14, 18, 22, 23  \}$.
Note that $\chi_{22}$ and $\chi_{23}$ are the only regular characters in this set. 
In step (2), let $\psi$ be a $Syl_p$-vanishing character of $H$ of degree $|H|_p$, whose
constituents lie in $\Delta_H$.

\begin{itemize}
\item[(1)] Suppose that $\chi_{22}$ is a constituent of $\psi$. Then, by degree reasons,
the other constituents of $\psi$ must belong to the set $\{\chi_h \mid h =4, 5, 14 \}$.
Looking at the class $C_5(q-1)$ we easily obtain a contradiction.
\item[(2)] Suppose that $\chi_{23}$ is a constituent of $\psi$. Then the other
constituents of $\psi$ must belong to the set $\{\chi_h \mid h=3, 4, 5, 14, 18 \}$.
Looking at the class $C_3(q-1)$, the character $\chi_{18}$ must be a constituent of
$\psi$. Now, the character values on the class $C_2(q-1)$ imply that $\chi_{14}$ is
another constituent of $\psi$. So, we obtain a unique $Syl_p$-vanishing character of
degree $|H|_p$: $\psi_1=\chi_5+ \chi_{14}+\chi_{18} +\chi_{23}$.
\end{itemize}
In step (3) we get all
$Syl_p$-vanishing characters of $H$ of degree $|H|_p$ (see Table \ref{TGL4}-C).

\begin{lemma}\label{LGL4}
Let $H=GL(4,q)$. Then $H$ admits reducible $p$-vanishing characters $\psi$ of degree
$|H|_p$ if and only if $3\mid (q+1)$. 
In this case $H$ admits exactly $(q-1)(q^2-\delta) /4$  characters $\psi$ with
$\delta=\frac{1-(-1)^q}{2}$, where
$$\psi=\chi_{5}(k)\cdot (1_H+\chi_{14}(a)+ \chi_{18}(0,a) +\chi_{23}((q^2-1)b)),$$
 $\chi_5(k)$ is a linear character of $H$, $a=(q+1)/3$ and $b\in \ZZ_{q^2+1}\setminus
\{0\}$.
\end{lemma}

\begin{proof}
Up to linear characters of $H$ we have the \f possibilities for $\psi$:
\begin{itemize}
\item[(1)]  $\psi_1=1_H+\chi_{14}(k_{14,1})+ \chi_{18}(k_{18,1},k_{18,2})
+\chi_{23}(k_{23,1})$, with $(q-1)$ dividing $2k_{14,1}$, $(2k_{18,1}+k_{18,2})$ and
$k_{23,1}$ (so $k_{23,1}=(q-1) k'_{23,1}$). 
On the classes $C_7(i_1,i_2)$ we obtain $\zeta_{1}^{(i_2-i_1)k_{18,1}}=1$, i.e.
$k_{18,1}=0$ and $k_{18,2}=(q-1) k'_{18,2}$.
Now, on the classes $C_{12}(i_1,i_2)$ we have $\zeta_1^{(i_1+i_2)k_{14,1}}=1$, so
$k_{14,1}=(q-1) k'_{14,1}$.   
On the classes $C_{14}(i_1)$ we have
$$1+Re(\xi_1^{2i_1 k'_{14,1}}) = Re(\xi^{i_1 k'_{18,2}})+Re(\xi_1^{i_1 k'_{23,1}}),$$
and on the classes $C_{18}(i_1,i_2)$ we have 
$$Re(\xi_1^{i_2 k'_{14,1}})=Re(\xi_1^{i_2 k'_{18,2}}).$$
This holds if and only if $3$ divides $q+1$, $\xi_1^{ k'_{14,1}},\xi_1^{ k'_{18,2}}\in
\{\omega,\omega^2\}$ and $(q+1)\mid  k'_{23,1}$.

\item[(2)]  $\psi_2=1_H+\chi_{13}(k_{13,1})+\chi_{18}(k_{18,1},k_{18,2})$. On the class
$C_{14}(q-1)$ we obtain $Re(\xi_1^{k_{18,2}})=1$, i.e. $(q+1)\mid k_{18,2}$, a
contradiction.
\end{itemize}
\end{proof}

If $H=GL(5,q)$, with step (1) we obtain  the $Syl_p$-decompositions of Table \ref{TGL5}.B
and $\Delta_H=\{5, 6, 7, 23, 24, 30, 36, 41, 42, 43 \}$.
Note that $\chi_{41}$, $\chi_{42}$
and $\chi_{43}$ are the only regular characters of this set.
In step (2), let $\psi$ be a $Syl_p$-vanishing character of $H$ of degree $|H|_p$, whose
constituents lie in $\Delta_H$.
\begin{itemize}
\item[(1)] Let $\chi_{41}$ be a constituent of $\psi$.  By degree reasons, the other
constituents of $\psi$ must belong to the set $\{\chi_h \mid h = 5, 6, 7, 24, 30, 36  
\}$. Looking at the class $C_6(q-1)$, we observe that $\psi$ must have at least $q-1$
constituents of type $\chi_{24}$, in contradiction with Lemma \ref{ur2}.
\item[(2)] Let $\chi_{42}$ be a constituent of $\psi$.  Then  the other constituents of
$\psi$ must belong to the set $\{\chi_h \mid h =  5, 6, 7, 24, 30    \}$. Using  the
class $C_5(q-1)$, we conclude that  $\chi_{24}$  must  be a constituent of $\psi$. Now,
using the class $C_4(q-1)$, with these constituents we  obtain a unique $Syl_p$-vanishing
character of degree $|H|_p$: $\psi_1=\chi_7+\chi_{24}+\chi_{30}+\chi_{42}$.
\item[(3)] Let $\chi_{43}$ be a constituent of $\psi$. Then the other constituents of
$\psi$ must belong to the set $\{\chi_h \mid h =  5, 6, 7, 24, 30, 36  \}$. Looking at the
classes $C_4(q-1)$ and $C_7(q-1)$, we observe that both $\chi_{24}$ and $\chi_{30}$ must
be constituents of $\psi$. However, using  $C_7(q-1)$, it follows that  $\psi$ must admit
another constituent of type $\chi_{24}$, which  contradicts Lemma \ref{ur2}.
\end{itemize}
In step (3)
we get all  $Syl_p$-vanishing characters of $H$ of degree $|H|_p$ (see Table
\ref{TGL5}-C).

\begin{lemma}\label{LGL5}
Let $H=GL(5,q)$. Then $H$ admits reducible $p$-vanishing characters $\psi$ of degree
$|H|_p$ if and only if $3\mid (q+1)$. In this case, there are at most $q^2(q-1)/2  $
characters
$$\psi=\chi_7(k)\cdot (1_H+\chi_{24}(0,a)+\chi_{30}(a, 0)+\chi_{42}((q^2-1)b,0)),$$
where $\chi_7(k)$ is a linear character of $H$, $a=(q+1)/3$ and $b\in \ZZ_{q^2+1}\setminus
\{0\}$.
\end{lemma}

\begin{proof}
Modulo linear characters of $H$ we have  the \f possibilities:
\begin{itemize}

\item[(1)]  $\psi_1=1_H+\chi_{24}(k_{24,1},k_{24,2})+\chi_{30}(k_{30,1},k_{30,2})
+\chi_{42}(k_{42,1},k_{42,2})$, where $q-1$ must divide $3k_{24,1}+k_{24,2}$,
$2k_{30,1}+k_{30,2}$ and $k_{41,1}+k_{42,2}$. On the classes $C_{36}(i_1,i_2)$ we have
$\zeta_1^{(i_2-3i_1)k_{24,1}}=1$. So $k_{24,1}=0$ and $k_{24,2}=(q-1) k'_{24,2}$. Now, on
the classes $C_{15}(i_1,i_2)$ we have $\zeta_1^{(i_2-i_1)k_{30,1}}=1$, i.e.
$k_{30,1}=(q-1) k'_{30,1}$ and $k_{30,2}=0$. Next, on  $C_9(i_1,i_2)$ we have
$\zeta_1^{(i_2-i_1)k_{42,2}}=1 $, i.e. $k_{42,2}=0$ and $k_{42,1}=(q-1) k'_{42,1}$. On the
classes $C_{23}(i_1,i_2)$ we have $Re(\xi_1^{i_2 k'_{24, 2}})=Re(\xi_1^{i_2 k'_{30,1}})$.
Finally, consider the classes $C_{30}(i_1, i_2)$ where we have 
$$1+Re(\xi_1^{2i_1 k'_{30,1}})=Re(\xi_1^{i_1 k'_{30,1}})+Re(\xi_1^{i_1 k'_{42,1}}).$$
This means that $(q+1)\mid k'_{42,1}$ and $\xi_1^{ k'_{30,1}}=\omega^{\pm 1}$. Thus, $3$
must divide $q+1$.

\item[(2)]  $\psi_2=1_H+\chi_{24}(k_{24,1},k_{24,2}) +\chi_{29}(k_{29,1},k_{29,2})$. On
the class $C_{30}(q-1,q-1)$ we obtain $Re(\xi_1^{k_{24,2}})=1$, which implies that $q+1$
divides $k_{24,2}$, a contradiction.
\end{itemize}
\end{proof}

If $H=GL(6,q)$, with step (1) we obtain the $Syl_p$-decompositions of Table \ref{TGL6}-B
and $\Delta_H=\{ 6, 7, 9, 10, 11, 39, 40, 48, 49, 50, 68, 70, 78, 79, 99,
107, 110,
111, 112
\}$.
No\-te that $\chi_{107}$, $\chi_{110}$, $\chi_{111}$ and $\chi_{112}$ are the only regular
characters in this set.
In step (2), let $\psi$ be a $Syl_p$-vanishing character of $H$ of degree $|H|_p$, whose
constituents lie in $\Delta_H$.
\begin{itemize}
\item[(1)] Suppose that $\chi_{107}$ is a constituent of $\psi$. By degree reasons, the
other constituents of $\psi$ must belong to the set $\{\chi_h \mid h =7, 9, 10, 11, 40, 50
 \}$. Then, using  the class $C_3(q-1)$ we get a contradiction.
\item[(2)] Suppose that $\chi_{110}$ is a constituent of $\psi$. Then, by degree reasons,
the other constituents of $\psi$ must belong to the set $\{\chi_h \mid h =6, 7, 9, 10, 11,
39, 40, 48, 49$, $50, 68, 70$, $79, 99  \}$.
Then, using  the class $C_3(q-1)$, we conclude that some constituent of $\psi$ must be of
type $\chi_{99}$. With some computations, we observe that the only possibility remained is
the $Syl_p$-character
$\psi_1=\chi_{11}+\chi_{40}+\chi_{50}+\chi_{68}+\chi_{99}+\chi_{110}$.

\item[(3)] Suppose that $\chi_{111}$ is a constituent of $\psi$. By degree reasons, the
other constituents of $\psi$ must belong to the set $\{\chi_h \mid h =6, 7, 9, 10, 11, 39,
40, 48$, $49, 50, 68, 70, 79 \}$. Using  the class $C_{11}(q-1)$, we see that either
$\chi_{40}$ or $\chi_{50}$ must be a constituent of $\psi$. In both the cases, we obtain a
contradiction.

\item[(4)] Suppose that $\chi_{112}$ is a constituent of $\psi$. Then, by degree reasons,
the other constituents of $\psi$ must belong to the set $\{\chi_h \mid h =6, 7, 9, 10, 11,
39, 40, 48, 49$, $50, 68, 70, 79, 99  \}$. Using  the class $C_{5}(q-1)$, we observe
that
$\chi_{39}$ must be a constituent of $\psi$. With some computations we again obtain a
contradiction.

\end{itemize}
In step (3)
we obtain the list of all  $Syl_p$-vanishing characters of $H$ of degree $|H|_p$ (see
Table \ref{TGL6}-C).

\begin{lemma}\label{LGL6}
Let $H=GL(6,q)$. Then, both $H$ and $G=H'$ have no reducible $p$-vanishing character of degree $|H|_p$. 
\end{lemma}

\begin{proof}                   
Suppose the contrary. Modulo linear characters of $H$, we have to inspect  the following
cases.
\begin{itemize} 
\item[(1)]
 $\psi_1~=~ 1_H ~~+~~ \chi_{40}(k_{40,1})~~ + ~~ \chi_{50}(k_{50,1},k_{50,2})~~ + ~~
\chi_{68}(k_{68,1},k_{68,2})~~ + ~~\chi_{99}(k_{99,1},k_{99,2})~~ + ~~~~~~~~~           
   \newline+\chi_{110}(k_{110,1},k_{110,2})$. On the classes $C_{40}((q-1)i_1' )$ ($i_1' 
=1,2$) we have 
$1+Re(\xi_1^{2i_1' k_{68,2}})+ Re(\xi_1^{i_1' (k_{110,1}+k_{110,2})}) +Re(\xi_1^{i_1'
(k_{110,1}-k_{110,2})})=Re(\xi_1^{i_1' k_{40,1}})+Re(\xi_1^{3i_1'
k_{40,1}})+Re(\xi_1^{i_1' k_{50,2}})+\newline+Re(\xi_1^{i_1' k_{99,2}})$.
Next, on the classes $C_{39}((q-1)i_1' )$ ($i_1' =1,2$) we get $1+
Re(\xi_1^{2i_1' k_{68,2}})=Re(\xi_1^{i_1' k_{1,40}})+Re(\xi_1^{i_1' k_{50,2}})$.
This implies that $q+1$ divides either $k_{40,1}$ or $k_{50,2}$. In both the cases, we
have a contradiction.

\item[(2)]  $\psi_2= 1_H + \chi_{40}(k_{40,1}) + \chi_{50}(k_{50,1},k_{50,2}) +
\chi_{67}(k_{67,1},k_{67,2}) + \chi_{110}(k_{110,1},k_{110,2})$. On the classes
$C_{40}((q-1)i_1' )$  ($i_1' =1,2$) we have
$$Re(\xi_1 ^{i_1'  k_{40,1}})+Re(\xi_1^{3i_1'  k_{40,1}})+Re(\xi_1 ^{i_1'  k_{50,2}})=1+
Re(\xi_1^{i_1' ( k_{110,1}+k_{110,2}) })+Re(\xi_1^{i_1' (k_{110,1}-k_{110,2})}).$$ 
Furthermore, on the classes $C_{39}((q-1)i_1' )$ ($i_1' =1,2$) we get $1+Re(\xi_1^{2i_1' 
k_{67,1}})=Re(\xi_1^{i_1'  k_{40,1}})+Re(\xi_1^{i_1'  k_{50,2}})$. This implies that $q+1$
divides either $k_{40,1}$ or $k_{50,2}$. In both the cases, we have a contradiction.

\item[(3)] $\psi_3= 1_H + \chi_{38}(k_{38,1}) + \chi_{50}(k_{50,1},k_{50,2}) +
\chi_{68}(k_{68,1},k_{68,2}) + \chi_{99}(k_{99,1},k_{99,2})$. On the class
$C_{49}(q-1,q-1)$ we have $Re(\xi_1^{k_{50,2}})=Re(\xi_1^{k_{68,2}})$. On the classes
$C_{40}((q-1)i_2' )$ ($i_2'  =1,2$) we have $1+Re(\xi_1^{i_2' k_{99,2}})=Re(\xi_1^{i_2'
k_{50,2}} )+Re(\xi_1 ^{2i_2' k_{68,2}})$, which implies $Re(\xi_1 ^{2 k_{68,2}})=1$ and
$Re(\xi_1^{k_{99,2}})=Re(\xi_1^{k_{50,2}} )$. 
On the classes $C_{39}((q-1)i_1' )$ ($i_2'  =1,2$) we have 
$(q^2+1)(1+Re(\xi_1^{2i_2' k_{68,2}}) =(q^2+1)Re(\xi_1^{i_1' k_{50,2}} )+q^2Re(\xi_1
^{i_1' k_{38,1}})+Re(\xi_1 ^{i_1' k_{99,1}})$. 

So $(q^2+2)Re(\xi_1^{i_1' k_{50,2}})+q^2Re(\xi_1^{i_1' k_{38,1}})=2(q^2+1)$
which implies $Re(\xi_1 ^{k_{38,1}})=1$, i.e. $(q+1)\mid k_{38,1}$, a contradiction.

\item[(4)]  $\psi_4=1_H + \chi_{40}(k_{40,1}) + \chi_{50}(k_{50,1},k_{50,2}) +
\chi_{68}(k_{68,1},k_{68,2}) + \chi_{98}(k_{98,1},k_{98,2})$. 
On the classes $C_{99}(q-1,q-1)$ we have $Re(\xi_1^{k_{40,1}})= Re(\xi_1^{ k_{68,2}})$,
and on the classes $C_{49}(q-1,(q-1)i_2' )$ ($i_2' =1,2$) we have
$q-1+Re(\xi_1^{i_2'  k_{50,2}})=(q-1)Re(\xi_1^{i_2'  k_{40,1}})+Re(\xi_1^{i_2' 
k_{68,2}})$. Hence, $q-1+Re(\xi_1^{i_2'  k_{50,2}})=qRe(\xi_1 ^{i_2'  k_{40,1}})$, which
implies $Re(\xi_1^{k_{40,1}})=1$, i.e. $(q+1)\mid k_{40,1}$, a contradiction.

\item[(5)]  $\psi_5= 1_H + \chi_{38}(k_{38,1}) + \chi_{50}(k_{50,1},k_{50,2}) +
\chi_{67}(k_{67,1},k_{67,2})$. On the class $C_{40}(q-1)$ we have
$Re(\xi_1^{k_{50,2}})=1$, i.e. $(q+1)\mid k_{50,2}$. However, this contradicts the
conditions on the parameters.
\end{itemize}

Since all the conjugacy classes that we have considered belong to $G=H'$, if follows from
Proposition \ref{lu6} that  $G$ also has no $p$-vanishing character of degree $|G|_p$.
\end{proof}

\subsection{Groups $H=CSp(6,q)$ with $q$ odd}

Applying Procedure \ref{procedure}, in step (1) we obtain the $Syl_p$-decompositions of
Table \ref{TSp6}-B and $\Delta_H=\{
 1, 4, 5, 7, 10, 11, 18, 36,$ $45, 50, 51, 65, 71, 97, 98   \}$. 
For step (2), we note that the only regular characters belonging to this set are
$\chi_{97}$ and $\chi_{98}$. Let $\psi$ be a $Syl_p$-vanishing character whose
constituents lie in $\Delta_H$.
\begin{itemize}
 \item[(i)] Let $\chi_{97}$ be a constituent of $\psi$. Then, by degree reasons, the other
constituents of $\psi$ must belong to the set $\{\chi_h \mid h =1, 4, 5, 18, 50, 71  \}$. 
Using  the class $C_{10}(q-1)$, $\psi$ must have at least one constituent of type
$\chi_{18}$. This  easily leads to a contradiction.
\item[(ii)] Let $\chi_{98}$ be a constituent of $\psi$. Then, using  the class $C_8(q-1)$,
we observe that another constituent of $\psi$ must be either of type $\chi_{45}$, or
$\chi_{50}$, or $\chi_{51}$.
We get, with some efforts, the following $Syl_p$-vanishing characters:
\begin{eqnarray*}
\psi_1 & =  &\chi_{1} + \chi_{5} + \chi_{11} + \chi_{18} + \chi_{45} + \chi_{45}' +
\chi_{45}'' + \chi_{50} + \chi_{51} + \chi_{65} + \chi_{71} + \chi_{98},\\
\psi_2 & =& \chi_{1} + \chi_{11} + \chi_{11}' + \chi_{18} + \chi_{45} + \chi_{45}' +
\chi_{45}'' + \chi_{45}''' + \chi_{50} + \chi_{65} + \chi_{71} + \chi_{98},\\
\psi_3 & = & \chi_{1} + \chi_{5} + \chi_{5}' + \chi_{18} + \chi_{45} + \chi_{45}' +
\chi_{50} + \chi_{51} + \chi_{51}' + \chi_{65} + \chi_{71} + \chi_{98},\\ 
\psi_4 &= & \chi_{5} + \chi_{10} + \chi_{11} + \chi_{11}' + \chi_{18} + \chi_{45} +
\chi_{45}' + \chi_{45}'' + \chi_{50} + \chi_{71} + \chi_{98},\\  
\psi_5 & = &  \chi_{5} + \chi_{5}' + \chi_{10} + \chi_{11} + \chi_{18} + \chi_{45} +
\chi_{45}' + \chi_{50} + \chi_{51} + \chi_{71} + \chi_{98},\\ 
\psi_6 & =&  \chi_{5} + \chi_{5}' + \chi_{5}'' + \chi_{10} + \chi_{18} + \chi_{45} +
\chi_{50} + \chi_{51} + \chi_{51}' + \chi_{71} + \chi_{98}.
\end{eqnarray*}
\end{itemize}
Now, in step (3)
we obtain the list of all  $Syl_p$-vanishing characters of $H$ of degree $|H|_p$ (see
Table \ref{TSp6}-C).

\begin{lemma}\label{LSp6}
Let $H=CSp(6,q)$, $q$ odd. Then both $H$ and $G=H'$ have no reducible $p$-vanishing
character of degree $|H|_p$. 
\end{lemma}

\begin{proof} 
Suppose the contrary. We consider the characters $\psi_j$ described in Table \ref{TSp6}-C.
By Lemma \ref{LSp4.1}, a reducible $p$-vanishing character $\psi$ of $G$ of degree $|G|_p$
must satisfy 
$(\overline{\psi}_{P_J})|_{G_J}=1_G+\chi_4(a,b)+\chi_6(c)+\chi_{10}(d)$, where $G_J\cong
Sp(4,q)$. So it suffices 
to consider only the following cases:
\begin{itemize}
\item[(1)] Take $\psi_j$ with $j \in \{55,56\}$. We obtain a contradiction with the value
at the class $C_{26}(1,0)$.

\item[(2)] Take $\psi_j$ with $j\in \{58,59,61,62,63,64,65,66,67,70,71,74,75\}$. We obtain
a contradiction with the values at  the 
class $C_{25}(1,0)$.

\item[(3)] Take $\psi_{81}$. Using  the class $C_{70}(1,q-1)$ we have
$Re(\xi_1^{k_{85,1}})= Re(\xi_1^{k_{65,1}})$  and using  the class $C_{72}(i_1,1,q-1)$ we
have $Re(\xi_1^{k_{85,2}})=Re(\xi_1^{k_{41,1}})$.
On the classes $C_{74}(i_1,i_2,q-1)$ we have
$Re(\xi_1^{(k_{87,1}+k_{87,2})i_1})+Re(\xi_1^{(k_{87,1}-k_{87,2})i_1})
=Re(\xi_1^{2k_{65,1}i_1})+Re(\xi_1^{k_{41,1}i_1})$ for all $i_1=1,\ldots,\frac{q-1}{2}$.
Using  the classes $C_{68}(i_1, i_2$, $q-1)$ we have $Re(\xi_1^{i_1k_{87,1}})+
Re(\xi_1^{i_1k_{87,2}}) =Re(\xi_1^{i_1k_{65,1}})+Re(\xi_1^{i_1k_{41,1}})$ for all
$i_1=1,\ldots,\frac{q-1}{2}$.
Using  the classes $C_{38}(i_1,q-1)$ and  the previous information, we obtain 
$$Re(\xi_1^{(k_{98,1}+k_{98,2}+k_{98,3})i_1})+Re(\xi_1^{(k_{98,1}-k_{98,2}+k_{98,3})i_1})
+$$ 
$$+Re(\xi_1^{(k_{98,1}-k_{98,2}-k_{98,3})i_1})+Re(\xi_1^{(k_{98,1}+k_{98,2}-k_{98,3})i_1
})   = $$
$$=Re(\xi_1^{(2k_{85,1}+k_{85,2})i_1})+Re(\xi_1^{(2k_{85,1}-k_{85,2})i_1})+Re(\xi_1^{k_
{51,1}i_1})+Re(\xi^{k_{85,2}i_1})$$ for all $i_1=1,\ldots,\frac{q-1}{2}$.
Using  the classes $C_{37}(i_1,q-1)$ and  the previous information, we obtain
 $1+
Re(\xi_1^{(2k_{85,1}+k_{85,2})i_1})+Re(\xi_1^{(2k_{85,1}-k_{85,2})i_1})=Re(\xi_1^{2k_{65,1
}i_1 })+Re(\xi_1^{k_{41,1}i_1})+Re(\xi_1^{3k_{51,1}i_1})$ for all
$i_1=1,\ldots,\frac{q-1}{2}$.  This implies that either $\xi_1^{2k_{65,1}}=1$, or
$\xi_1^{k_{41,1}}=1$, or $\xi_1^{3k_{51,1}}=1$. 
 In the first two cases, we get a contradiction with the conditions on the parameters
$k_{65,1}$ and $k_{41,1}$. 
 So $\xi_1^{3k_{51,1}}=1$, i.e. $\xi_1^{k_{51,1}}=\omega^{\pm 1}$. Furthermore,
$Re(\xi_1^{(2k_{85,1}+k_{85,2})i_1})+Re(\xi_1^{(2k_{85,1}-k_{85,2})i_1})=Re(\xi_1^{2k_{65,
1}i_1 })+Re(\xi^{k_{41,1}i_1})$.

First observe that 
 $$\left\{\begin{array}{l} Re(\xi_1^{k_{85,1}})=Re(\xi_1^{k_{65,1}}) \\
Re(\xi_1^{k_{85,2}})=Re(\xi_1^{k_{41,1}}) \\
Re(\xi_1^{2k_{85,1}+k_{85,2}})+Re(\xi_1^{2k_{85,1}-k_{85,2}})=Re(\xi_1^{2k_{65,1}}
)+Re(\xi_1^{k_{41,1}}) \\
Re(\xi_1^{4k_{85,1} +2k_{85,2} })+Re(\xi_1^{4k_{85,1} -2k_{85,2}
})=Re(\xi_1^{4k_{65,1}})+Re(\xi_1^{2k_{41,1}}) \\
\end{array}\right. .$$
Setting $A=Re(\xi_1^{k_{85,1}})$ and $B=Re(\xi_1^{k_{85,2}})$, we obtain

$$\quad \left\{\begin{array}{l} 2(2A^2-1)B=(2A^2-1)+B\\
2(8A^4-8A^2+1)(2B^2-1)=(8A^4-8A^2+1)+(2B^2-1)\\
\end{array}\right.  . $$
Solving this system and remembering that $A\neq \pm 1$, we obtain that either
$$(a)\quad A=\frac{1\pm \sqrt{5}}{4}, \qquad B=-A,$$
or
$$(b)\quad A=\frac{-1\pm \sqrt{5}}{4}, \qquad B=A.$$

Now, consider the following equations: $$\left\{\begin{array}{l}
Re(\xi_1^{k_{87,1} })+Re(\xi_1^{k_{87,2}})=Re(\xi_1^{k_{65,1}})+Re(\xi_1^{k_{41,1}} ) \\
Re(\xi_1^{k_{87,1}+k_{87,2}})+Re(\xi_1^{k_{87,1}-k_{87,2}})=Re(\xi_1^{2k_{65,1}}
)+Re(\xi_1^{k_{41,1}}) \\
Re(\xi_1^{2k_{87,1}})+Re(\xi_1^{2k_{87,2}})=Re(\xi_1^{2k_{65,1}})+Re(\xi_1^{2k_{41,1}} )
\\
\end{array}\right. .$$
Setting $X=Re(\xi_1^{k_{87,1}})$ and $Y=Re(\xi_1^{k_{87,2}})$, we obtain in case (a)
$$\left\{\begin{array}{l} 2XY=-1/2\\ X+Y =
0\\ X^2+Y^2=2A^2
\end{array}\right.  . $$
So $Y=-X$, $X^2=\frac{1}{4}$ and $\frac{1}{2}=2(\frac{1\pm
\sqrt{5}}{4})^2$, an absurd.
In case (b) we get
$$\left\{\begin{array}{l} 2XY=-1/2\\ X+Y =
2A\\ X^2+Y^2=2A^2
\end{array}\right.  . $$
This implies $A^2=-1/4$, again an absurd.

\item[(4)] Take $\psi_{82}$. Using the classes $C_{70}(1,q-1)$ and $C_{72}(i_1,q+1,q-1)$,
we have
 $Re(\xi_1^{k_{85,1}})=Re(\xi_1^{k_{65,1}})$ and
$Re(\xi_1^{k_{85,2}})=Re(\xi_1^{k_{41,1}})$. Now, considering also the class
$C_{32}(q,q-1)$, we obtain
$Re(\xi_1^{k_{85,1}})+Re(\xi_1^{k_{85,2}})+Re(\xi_1^{k_{51,1}})=1+Re(\xi_1^{k_{65,1}}
)+Re(\xi_1^{k_{41,1}})$, whence $Re(\xi_1^{k_{51,1}})=1$, a contradiction.

\item[(5)] Take $\psi_{83}$. Using  the class $C_{70}(1,q-1)$ we have
$Re(\xi_1^{k_{86,1}})=Re(\xi_1^{k_{65,1}})$, and looking at the class $C_{72}(i_1,1,q-1)$
we have $Re(\xi_1^{k_{86,2}})=Re(\xi_1^{k_{41,1}})$. Using  the classes
$C_{74}(i_1,i_2,q-1)$ we have
$Re(\xi_1^{(k_{87,1}+k_{87,2})i_1})+Re(\xi_1^{(k_{87,1}-k_{87,2})i_1})=Re(\xi_1^{2k_{65,1}
i_1})+Re(\xi_1^{k_{41,1}i_1})$ for all $i_1=1,2,3$.
Using  the classes $C_{37}(i_1,q-1)$ and the previous information, we obtain
$q+qRe(\xi_1^{(2k_{86,1}-k_{86,2})i_1}) +qRe(\xi_1^{(2k_{86,1}+k_{86,2})i_1})
+(q-1)Re(\xi_1^{k_{41,1}i_1})=qRe(\xi_1^{2k_{65,1}i_1})+qRe(\xi_1^{3k_{51,1}i_1})+
(2q-1)Re(\xi_1^{k_{51,1}i_1})$ for all $i_1=1,\ldots,\frac{q-1}{2}$. This implies that
either
$\xi_1^{2k_{65,1}}=1$, or $\xi_1^{k_{51,1}}=1$, or $\xi_1^{3k_{51,1}}=1$. In the first two
cases we get a contradiction with the conditions on parameters. So $\xi_1^{3k_{51,1}}=1$,
i.e. $\xi_1^{k_{51,1}}=\omega^{\pm 1}$ and
$Re(\xi_1^{k_{51,1}})=-\frac{1}{2}$. Furthermore,
$Re(\xi_1^{k_{41,1}})=Re(\xi_1^{k_{51,1}})$ and
\begin{equation}\label{37}
 Re(\xi_1^{2k_{86,1}+k_{86,2}})+Re(\xi_1^{2k_{86,1}-k_{86,2}})=Re(\xi_1^{2k_{65,1}}
)+Re(\xi_1^{k_{41,1}}).
\end{equation}

Using  the classes $C_{32}(i_1,q-1)$ we have
$Re(\xi_1^{i_1k_{87,1}})+Re(\xi_1^{i_1k_{87,2}})=
Re(\xi_1^{i_1k_{86,1}})+Re(\xi_1^{i_1k_{65,1}})=2Re(\xi_1^{i_1k_{65,1}})$, for all
$i_1=1,\ldots,\frac{q-1}{2}$. This implies
$Re(\xi_1^{k_{87,1}})=Re(\xi^{k_{87,2}})=Re(\xi_1^{k_{65,1}})$.
 Using  the classes $C_{68}(i_1,j,q-1)$ we have
$Re(\xi_1^{k_{87,2}})=Re(\xi_1^{k_{41,1}})$.
Therefore,
$Re(\xi_1^{k_{86,1}})=Re(\xi_1^{k_{86,2}})=Re(\xi_1^{k_{65,1}})+Re(\xi_1^{k_{41,1}}
)=Re(\xi_1^{k_{51,1}})=-\frac{1}{2}$.
Now, equation \eqref{37} becomes $$
Re(\xi^{2k_{86,1}+k_{86,2}})+Re(\xi^{2k_{86,1}-k_{86,2}})=-1,$$ which has no solution with
our values of the parameters.

\item[(6)] Take $\psi_{84},\psi_{85}$. Using the classes $C_{74}(1,i_2,0)$, we have
$k_{41,1}=q+1$, a contradiction.

\item[(7)] Take $\psi_{86},\psi_{87}$. Using the classes $C_{72}(i_1,1,0)$, we have
$k_{41,1}=q+1$, a contradiction.

\item[(8)] Take $\psi_{92},\psi_{93},\psi_{100}$. Using the class $C_{70}(1,0)$, we have
$k_{65,1}=q+1$, a contradiction.

\item[(9)] Take $\psi_{96}, \psi_{97}$. Using the classes $C_{68}(i_1,q,0)$, we have
$k_{65,1}=q+1$, a contradiction.

\item[(10)] Take $\psi_{78}$. Using the classes $C_{70}(q,0)$ and
$C_{72}(i_1,0,0)$  we obtain
$Re(\xi_1^{k_{41,1}})=Re(\xi_1^{k_{50,1}})=Re(\xi_1^{k_{65,1}})$. Using  the
classes $C_{68}(i_1,q,0)$   we obtain
$Re(\xi_1^{k_{87,1}})=Re(\xi_1^{k_{87,2}})$, a contradiction.

\item[(11)] Take $\psi_{79}$. Proceed as for $\psi_{78}$.

\item[(12)] Take $\psi_{89}$. Using  the classes $C_{37}(i_1,0)$ and $C_{74}(i_1,i_2,0)$
($i_1=0,1$), 
we obtain $1+Re(\xi_1^{k_{52,1}})=Re(\xi_1^{k_{41,1}})+Re(\xi_1^{2k_{65,1}})$. So either
$q+1$ divides 
$k_{41,1}$ or $(q+1)/2$ divides $k_{65,1}$. In both the cases, we have a contradiction.
\end{itemize}

\medskip
\noindent Since all the conjugacy classes that we have considered belong to $H'$, it
follows from
Proposition \ref{lu6} that also $G=H'$ has no
$p$-vanishing character of degree $|G|_p$.
\end{proof}

\section{$p$-vanishing characters in high rank groups}

We can now prove that a Chevalley group of rank $\ell\geq 6$ has no $p$-vanishing character of degree $|G|_p$.
 
\begin{theo}\label{fr4}
Let $G$ be a Chevalley group, and 
$\chi$ be a $p$-vanishing character of $G$ of degree $|G|_p$.
Then $\chi=St$ in each of the following cases:
\begin{itemize}
\item[(1)] $SL(n,q)$ with $n\geq 6$;

\item[(2)] $SU(n,q)$ with $n\geq 4$;

\item[(3)] $Sp(2m,q)$ with either $m\geq 3$ if $7\mid (q+1)$ or  $m\geq 2$ otherwise;

\item[(4)] $Spin(2m+1,q)$, $q$ odd, with $m\geq 6$;

\item[(5)] $Spin^+(2m,q)$  with $m\geq 6$;

\item[(6)] $Spin^-(2m,q)$ with $m\geq 3$;

\item[(7)] $E_6(q)$, $E_7(q)$, $E_8(q)$;

\item[(8)] ${}^2E_6(q)$;

\item[(9)] $F_4(q)$;

\item[(10)] $Spin(2m+1,q)$, $q$ odd, with either  $(q+1,7)=1$ and $m=3,4,5$ or $(q+1,3)=1$ and $m=4,5$.

\item[(11)] $Spin^+(2m,q)$  with $ m=4, 5$ and $(q+1,3)=1$.
\end{itemize}
\end{theo}

\begin{proof}
We proceed by induction on the BN-pair rank $\ell$ of the group.  Let $\chi$ be a $p$-vanishing character of $G$ of degree $|G|_p$. Take two subsets $J,K$ of $I_\ell$ with $I_\ell=J\cup K$ and $J\cap K \neq \emptyset$, 
such that $|K|=2$. 
By Theorem \ref{nh2}, either $\chi=St$ or $(\chi,1_G)=1$. In the latter case, $(\overline{\chi}_L,1_L)=1$ for every Levi subgroup $L$ of $G$.
So, we obtain the statement if we are able to find $J$ such that $G_J$ has no reducible $p$-vanishing characters of degree $|G_J|_p$.

(1) Let $G=SL(n,q)$ with $n\geq 7$. Then we can take $J,K$ as described above, such
that $G_J\cong SL(n-1,q)$ and $G_K\cong SL(3,q)$.

(2) Let $G=SU(n,q)$ with $n\geq 6$. In this case, we can take $J,K$ such that
$G_J\cong SU(n-2,q)$ and $G_K\cong SL(3,q)$.

(3) Let $G=Sp(2m,q)$, with either $m\geq 4$ if $7$ divides $q+1$ or $m\geq 3$,
otherwise. In this case, we can take $J,K$ such that $G_J=Sp(2m-2, q)$ and $G_K\cong
SL(3,q)$.

(4) Let $G=Spin(2m+1,q)$ with $m\geq 6$. In this case, we can take $J,K$ such that
$G_J=SL(m,q)$ and $G_K\cong Sp(4,q)$. So the result follows from item (1).

(5) Let $G=Spin^+(2m,q)$  with $m\geq 6$. In this case, we can take $J, K$ such that
$G_J=SL(m,q)$ and $G_K\cong SL(3,q)$. So the result follows from item (1).

(6) If $G=Spin^-(6,q)$, then $G\cong SU(4,q)$ and so $\chi=St$ by Lemma \ref{LU4}.
So, assume that  $G=Spin^-(2m,q)$ with $m> 3$. Thus, we can take $J,K$ such that $G_J\cong
Spin^-(2m-2,q)$ and $G_K\cong SL(3,q)$.

(7) Let $G=E_n(q)$ with $n=6,7,8$. In this case we can take $J,K$ such that
$G_J\cong SL(n,q)$ and $G_K\cong SL(3,q)$. So, the result follows from item (1).

(8) Let $G={}^2 E_6(q)$. In this case we can take $J,K$ such that $G_J\cong SU(6,q)$
and $G_K\cong SL(3,q)$.

(9) Let $G=F_4(q)$. In this case we can take $J,K$ such that $G_J\cong Sp(6,q)$ and
$G_K\cong \Omega(5,q)$.

(10) Let $G=Spin(2m+1,q)$ and $q$ odd. Assume first that $(q+1,7)=1$. In this case,
for $m=3,4,5$ we can take $G_J\cong Spin(2m-1,q)$ and $G_K\cong SL(3,q)$. Now, suppose
that $(q+1,3)=1$. Hence, for $m=4,5$ we can take  $G_J\cong SL(m,q)$ and $G_K\cong
Sp(4,q)$.

(11) Let $G=Spin^+(2m,q)$  with $m=4,5$ and $(q+1,3)=1$. In this case we can take
$J,K$ such that $G_J\cong SL(m,q)$ and $G_K\cong SL(3,q)$. So the result follows from
Lemmas \ref{LGL4} and \ref{LGL5}.
\end{proof}

\newpage
\section{The tables}\label{tables}

The tables below contain information on the $Syl_p$-vanishing characters and, when they exist, 
on the reducible $p$-vanishing characters of degree $St(1)$. (This equals the order of a \syl $U$ of the group in question.)
 The tables  are organized as follows.
For every quasi-simple Chevalley group $G$ and/or its version $H$ with connected center we assign 
a subsection in which we provide 
some properties of the groups $G,H$ and their characters, necessary for understanding
our results about these groups tabulated in the same subsection. In some cases we 
have distinct subsections for the same group $H$, depending on whether   $q$ is odd or even, or
congruent to $1$ or $-1$ modulo $3$, etc. 

First we recall that \textsc{chevie} partitions the conjugacy classes  and 
\ir characters of $H$ in subsets which we denote here by ${\mathbf C}_h$ and ${\mathbf X}_h$ ($h=1,2,\ldots$). 
Note that the size of every class in ${\mathbf C}_h$ is the same. Similarly,
the degrees of the characters in ${\mathbf X}_h$ are the same.
(Sometimes the degrees of characters from distinct ${\mathbf X}_h$ may coincide.) We keep the notation 
of \textsc{chevie} for irreducible characters and conjugacy classes. The symbol $\chi_h(k)$ 
identifies the irreducible character of the set ${\bf X}_h$ associated to the particular element $k\in I_h$ 
according to \textsc{chevie}. We denote by $I_h$ the set of parameters describing the individual elements of ${\bf X}_h$.
If it is not necessary to specify the parameter $k$, we simply write $\chi_h$ to 
indicate any element of the set ${\bf X}_h$. Similarly, an individual conjugacy class
$C_h(k)\in {\bf C}_h$ is the class which corresponds to the parameter $k\in I_h$.

The parameters $k\in I_h$ labeling  individual characters $\chi_h(k)\in {\bf X}_h$ and
classes $C_h(k)\in {\bf C}_h$ are represented as 
elements of a direct product of cyclic groups, with some conditions (note that not
necessarily to distinct parameters 
correspond distinct characters). In the tables we indicate the group $I_h$ belongs to in
the columns headed $I_h$
and the exact description of $I_h$ is obtained by the exclusion of the elements written in
the columns headed `exceptions'.

In each subsection we usually give two or  three tables   headed as Table A,
Table B, Table C and Table D (in some cases Table D is omitted).  

Tables  A  provides some information on the irreducible characters of $H$ occurred in Tables B,C,D.
Specifically, Tables A  list the degrees of the members of the sets ${\bf X}_h$, 
the range of $h$ and we also describe the structure of the character parameter group $I_h$.
In some degenerate case the set ${\bf X}_h$ consists of a single element, so no parameter
needs to be assigned to the element of  the corresponding set $I_h$. In this case the
respective position in the table is marked by $-$, see for instance Table \ref{TU3}.2.A.
In order to avoid misunderstanding with \textsc{chevie} notation $\chi_h(k)$, in Tables A
we use $\deg \chi_h$ to indicate the degree of a character $\chi_h$.

In Tables B, we report the $Syl_p$-decompositions that we have found and  used in   our
computations, 
see Sections \ref{rk1}, \ref{rk2} and \ref{rk3}. 

In Tables C we list all the $Syl_p$-vanishing characters of degree $|H|_p$. 
The column headed by ${\bf v}$ lists, in more friendly notation, the $Syl_p$-vanishing
characters $\psi=\sum\chi_h$, where $\chi_h\in{\bf X}_h$ and the sum ranges over the
indices from ${\bf v}$. Say, if ${\bf v}=(1,4)$
then $\psi=\chi_1+\chi_4$. The column headed by $j$ is simply the ordering numbers for
${\bf v}$; we could write ${\bf v}_j$ to be more precise.

 Finally, in Tables D, when they exist, we provide the values of the $p$-vanishing
characters $\psi$ of  
degree $|H|_p$ satisfying the condition $(\psi,1_H)=1$. 
Note that no information on the character values is lost because all the other $p$-vanishing characters 
can be obtained from these by multiplying them with linear characters of $H$.  In addition, we explicitly 
write down only the values of $\psi$ at those conjugacy classes where $\psi$ differs from  the 
Steinberg character. Unless otherwise stated, the conjugacy classes in Table D are
parametrised by the same 
sets $I_h$ that are used to parametrise the characters in Table A.

To facilitate understanding the tables, as an illustration, 
 in Section \ref{TGL2} for groups $GL(2,q)$ and $SL(2,q)$ we write 
additional  comments about the table content.

We use the following notation for some primitive complex roots of the unity:
$\omega=\exp(\frac{2\pi i}{3})$, 
$\xi_1=\exp(\frac{2\pi i}{q+1})$, $\xi_2=\exp(\frac{2\pi i}{q^2+1})$, 
$\varphi_8'=\exp(\frac{2\pi i}{q^2+\sqrt{2}q+1})$ and $\varphi_{12}''=\exp(\frac{2\pi i}{q^2+\sqrt{3}q+1})$.

\subsection{Groups $H=GL(2,q)$}\label{TGL2}

\begin{tab}[\ref{TGL2}-A]\tabA{$GL(2,q)$}

\smallskip
\begin{center}
\begin{footnotesize}
$\left.
\begin{array}{l|l||l|l}
h & \deg \chi_h & h & \deg \chi_h\\\hline 
1 &  1 & 2 &  q \\
3 &  q+1 & 4 & q-1\\
\end{array}
\right.
\quad 
\left.
~~~~~~~~~~\begin{array}{l|l|l}
h & I_h & {\rm exceptions} \\\hline
1,2  & k\in \ZZ_{q-1} \\ 
3 & (k,l)\in \ZZ_{q-1}^2 & (q-1)\mid (k-l) \\ 
4 & k\in \ZZ_{q^2-1} & (q+1)\mid k\\
\end{array}\right.$
\end{footnotesize}
\end{center}

\smallskip
These two tables provide the following information. The irreducible characters of the
group $GL(2,q)$ are partitioned into $4$ sets ${\bf X}_h$. The characters of the set ${\bf
X}_1$ have degree $\deg \chi_1=1$ and are parametrised by the elements of the set
$I_1=\ZZ_{q-1}$. This means that ${\bf X}_1=\{\chi_1(0), \ldots,\chi_1(q-2)\}$. The
characters of the set ${\bf X}_2$ have degree $q$ and are  parametrised by the elements of
$I_2=\ZZ_{q-1}$. So, ${\bf X}_2=\{\chi_2(u) \mid u \in \ZZ_{q-1}\}$. The characters of the
set ${\bf X}_3$ have degree $q+1$ and  are parametrised by the elements of 
$I_3=\ZZ_{q-1}\times \ZZ_{q-1}$. 
One more time, we put in evidence that different choices of the parameters $(k,l)\in I_3$
may identify the same irreducible character. Furthermore, the exceptions listed in the
third column mean that the character $\chi_3(k,l)$ is not a member of ${\bf X}_3$ whenever
$(q-1)\mid (k-l)$. It turns out that the set ${\bf X}_3$ contains $(q-2)(q-1)/2$ distinct
characters. Finally, the characters of the set ${\bf X}_4$ have degree $q-1$ and are
parametrised by the elements of  $I_
4=\ZZ_{q^2-1}$. In this case, $\chi_4(k)$ does not belong to ${\bf X}_4$, whenever $ (q+1)
\mid k$. (Observe that $(q-1)\mid(k-l)$ is equivalent to $k=l$, but we prefer to keep
\textsc{chevie}  notation.)
\end{tab}

\begin{tab}[\ref{TGL2}-B] \tabB{$GL(2,q)$}

\begin{center}
\begin{footnotesize}
$\left.\begin{array}{c|c||c|c}
h & \mathbf{v} & h & \mathbf{v} \\ \hline 
2 & (1,4) & 3 & (1,2)\\
\end{array}\right.$
\end{footnotesize}
\end{center}

\smallskip
This table   give the following information on the $Syl_p$-decompositions that we were able to find. The restriction to $U$ of each character $\chi_2$ belonging to the set ${\bf X}_2$ can be written as the sum of two summands: the first one is the restriction of any character of the set ${\bf X}_1$, the second one is the restriction of any character of the set ${\bf X}_4$. Briefly, using the notation of Section \ref{Syldec}, $\chi_2\equiv \chi_1+\chi_4 \pmod U$. Similarly, $\chi_3\equiv \chi_1+\chi_2\pmod U$.
\end{tab}

\begin{tab}[\ref{TGL2}-C] \tabC{$GL(2,q)$}

\begin{center}
\begin{footnotesize}
$\left. \begin{array}{c|c||c|c}
j & \mathbf{v} & j & \mathbf{v} \\ \hline 
1 & (1,4) & 2 & (2)\\
\end{array}\right.$
\end{footnotesize}
\end{center}

\smallskip
This table  lists all  $Syl_p$-characters $\psi_j$  of $H=GL(2,q)$ of degree $|H|_p$. For $H$ they are of two types: $\psi_1=\chi_1+\chi_4$ 
(i.e. the sum of two irreducible characters: the first one is any character belonging to the set ${\bf X}_1$,
the second one is any character belonging to the set ${\bf X}_4$); $\psi_2=\chi_2$ (i.e. any character of the set ${\bf X}_2$). 
\end{tab}

\begin{tab}[\ref{TGL2}-D] \tabD{$GL(2,q)$}

\smallskip
In view of Lemma \ref{LG2} we have \tabDD{$\psi=1_H+\chi_4(k)$, where $k\in I_4$
and $k=(q-1)k'$} 

\begin{center} 
\begin{footnotesize}
$\left.
\begin{array}{l||l|l}
C_h & St & \psi\\\hline
C_4(i) & -1 & 1-2Re(\xi_1^{ik'})\\
\end{array}
 \right.$
\end{footnotesize}
\end{center}

\smallskip
In the above table, we give  the values of the reducible $p$-vanishing character $\psi$ and those of the Steinberg character $St$ of $H$ only at the conjugacy classes where these values differ. 
The notation for the conjugacy classes, which is similar to the notation for the irreducible characters, is taken according to \textsc{chevie}.
\end{tab}

\subsection{Groups $U(3,q)$ and $SU(3,q)$}\label{TU3}

\begin{tab}[\ref{TU3}.1-A] \tabA{$U(3,q)$}

\smallskip
\begin{center}
\begin{footnotesize}
$\left.
\begin{array}{l|l||l|l||l|l}
h & \deg \chi_h & h & \deg \chi_h& h & \deg \chi_h \\\hline 
1 &  1 & 2 &  q (q-1) &
3 &  q^3 \\
4 & q^2-q+1 &
5 &  q (q^2-q+1) &
 6 &  (q-1) (q^2-q+1)\\
7 &  (q+1) (q^2-q+1) & 8 &  (q+1)^2 (q-1)\\
\end{array}
\right.$

\med
$\left. 
\begin{array}{l|l|l}
h & I_h & {\rm exceptions} \\ \hline
1,2,3  & u\in \ZZ_{q+1}\\
4,5 & (u,v)\in \ZZ_{q+1}\times \ZZ_{q+1} & (q+1)\mid (u-v)\\
6 & (u,v,w)\in \ZZ_{q+1}\times \ZZ_{q+1}\times \ZZ_{q+1} &  (q+1)\mid (u-v), (u-w),
(v-w)\\
7 & (v,u)\in \ZZ_{q+1}\times \ZZ_{q^2-1}  &  (q-1)\mid u\\
8 & u \in \ZZ_{q^3+1} & (q^2-q+1)\mid u\\
\end{array} \right.$
\end{footnotesize}
\end{center}
\end{tab}

\begin{tab}[\ref{TU3}.1-B] \tabB{$U(3,q)$}

\begin{center}
\begin{footnotesize}
$\left. \begin{array}{c|c||c|c||c|c||c|c||c|c||c|c}
h & \mathbf{v} & h & \mathbf{v} & h & \mathbf{v} & h & \mathbf{v} & h & \mathbf{v} & h & \mathbf{v} \\ \hline 
3 &  (2,5) & 4 & (1,2) & 5 & (4,6) & 7 & (1,3) & 7 & (4,5) & 8 & (2,2,2,6)\\
\end{array} \right.$
\end{footnotesize}
\end{center}

\smallskip
The above table describes 6 types of $Syl_p$-decompositions of \ir characters of $U(3,q)$.
For instance, the last column tells us that we have
$\chi_8\equiv \chi_2+\chi'_2+\chi''_2+\chi_6\pmod U$. This means that the restriction to $U$ of any character of the set ${\bf X}_8$ is the sum of the restrictions of four characters: any $3$ characters $\chi_2,\chi'_2,\chi''_2$ of the set ${\bf X}_2$ (distinct or not) and an arbitrary  
character $\chi_6$ from  ${\bf X}_6$.
\end{tab}

\begin{tab}[\ref{TU3}.1-C] \tabC{$U(3,q)$}

\begin{center}
\begin{footnotesize}
$\left. \begin{array}{c|c||c|c||c|c||c|c}
j & \mathbf{v} & j & \mathbf{v} & j & \mathbf{v} & j & \mathbf{v} \\\hline 
1 & (1,2,2,6) & 2 & (2,4,6) & 3 & (2,5) & 4  & (3)\\
\end{array}\right.$
\end{footnotesize}
\end{center}

\smallskip
Note that  the above table gives  the $Syl_p$-vanishing character 
 $\psi_1=\chi_1+\chi_2+\chi'_2+\chi_6$. 
This means that $\psi_1$ is the sum of $4$ characters:
$\chi_1$ is any character the set ${\bf X}_1$, $ \chi_2,\chi'_2$ are any  characters (distinct or not) of 
 ${\bf X}_2$, and $\chi_6$ is any character of  ${\bf X}_6$.
\end{tab}

\begin{tab}[\ref{TU3}.1-D] \tabD{$U(3,q)$ when $3$ divides $q+1$}

\smallskip
In view of Lemma \ref{LU3}, we have
\tabDD{$\psi= 1_H+ \chi_{2}(a) + \chi_{2}(2a) + \chi_{6}(a,2a,3a)$, where $a=(q+1)/3$}

\begin{center} 
\begin{footnotesize}
 $\left.
\begin{array}{l||l|l}
C_h & St & \psi\\\hline 
C_6(k,l,m) & -1 & 1+4Re(\omega^{k+l+m})-2Re(\omega^{k-l})-2Re(\omega^{k-m})-2Re(\omega^{l-m}) \\
C_8(k) & -1 & 1-2Re(\omega^k)\\
\end{array}
 \right.$
\end{footnotesize}
\end{center}
\end{tab}

\begin{tab}[\ref{TU3}-2.A] \tabA{$SU(3,q)$ with the condition that $3\nmid (q+1)$}

\smallskip
\begin{center}
\begin{footnotesize}
$\left.
\begin{array}{l|l||l|l||l|l||l|l}
h & \deg \chi_h & h & \deg \chi_h& h & \deg \chi_h &  h & \deg \chi_h \\\hline 
1 &  1 &
2 &  q(q-1)&
3 &  q^3 & 
4 &  q^2-q+1\\
5 &  q(q^2-q+1)&
6 &  (q-1)(q^2-q+1) &
7 &  (q+1)(q^2-q+1)&
8 & (q+1)^2(q-1)\\
\end{array}
\right.$

\med
$\left. \begin{array}{l|l|l}
h & I_h & {\rm exceptions} \\ \hline
1,2,3 & - & \\
4,5 & n\in \ZZ_{q+1} & (q+1)\mid n \\
6 & (n,m)\in \ZZ_{q+1}\times \ZZ_{q+1} &  (q+1)\mid (n-m)\\
7 & n \in \ZZ_{q^2-1}  &  (q-1)\mid n\\
8 & n \in \ZZ_{q^2-q+1} & (q^2-q+1)\mid n\\
\end{array}\right.$
\end{footnotesize}
\end{center}

\smallskip

Note that each set ${\bf X}_1$, ${\bf X}_2$ and ${\bf X}_3$  consists  of  a single 
character.
\end{tab}

\begin{tab}[\ref{TU3}.2-B] \tabB{$SU(3,q)$  with the condition that $3\nmid (q+1)$}

\begin{center}\begin{footnotesize}
$\left. \begin{array}{c|c||c|c||c|c||c|c||c|c||c|c}
h & \mathbf{v} & h & \mathbf{v} & h & \mathbf{v} & h & \mathbf{v} & h & \mathbf{v} & h & \mathbf{v} \\ \hline 
3 &  (2, 5) & 
4 &  (1, 2) & 
5 &  (4, 6) & 
7 &  (1, 3) & 
7 &  (4, 5) &
8 &  (2, 2, 2, 6) \\
\end{array}\right. $
\end{footnotesize}\end{center}
\end{tab}

\begin{tab}[\ref{TU3}.2-C] \tabC{$SU(3,q)$ with  the condition that $3\nmid (q+1)$}

\begin{center}\begin{footnotesize}
$\left. \begin{array}{c|c||c|c||c|c||c|c}
j & \mathbf{v} & j & \mathbf{v} & j & \mathbf{v} & j & \mathbf{v} \\\hline 
1 &  (1, 2, 2, 6) & 
2 &  (2, 4, 6) & 
3 &  (2, 5) & 
4 &  (3) \\ 
\end{array}\right. $
\end{footnotesize}\end{center}
\end{tab}

\begin{tab}[\ref{TU3}.3-A]  \tabA{$SU(3,q)$ with the condition that $3\mid (q+1)$}

\smallskip
\begin{center}\begin{footnotesize}
$\left.
\begin{array}{l|l||l|l||l|l}
h & \deg \chi_h & h & \deg \chi_h& h & \deg \chi_h \\\hline 
1  & 1 & 
2  & q (q-1)&
3  & q^3  \\
4  & q^2-q+1 &
5  & q (q^2-q+1) & 
6  & (q+1) (q^2-q+1)\\
7,8,9  & \frac{1}{3} (q-1) (q^2-q+1) & 
10 &  (q-1) (q^2-q+1)& 
11,\ldots,16  & \frac{1}{3} (q+1)^2 (q-1) \\
17  & (q+1)^2 (q-1)\\
\end{array}
\right.$

\med

$\left. \begin{array}{l|l|l}
h & I_h & {\rm exceptions} \\ \hline
1,2,3,7,8,9,11,\ldots,16 & - & \\
4,5 & n\in \ZZ_{q+1} & (q+1)\mid n \\
6 & n \in \ZZ_{q^2-1} & (q-1)\mid n \\
10 &  (n,m)\in \ZZ_{q+1}\times \ZZ_{q+1} & (q+1)\mid n,m, (n-m)\\
17 & n\in \ZZ_{q^2-q+1}  & \frac{q^2-q+1}{3}\mid n\\
\end{array}\right. $ 
\end{footnotesize}\end{center}
\end{tab}

\begin{tab}[\ref{TU3}.3-B] \tabB{$SU(3,q)$ with the condition that $3\mid (q+1)$}

\begin{center}\begin{footnotesize}
$\left. \begin{array}{c|c||c|c||c|c||c|c||c|c}
h & \mathbf{v} & h & \mathbf{v} & h & \mathbf{v} & h & \mathbf{v} & h & \mathbf{v}  \\ \hline 
11 &  (14) & 
12 &  (16) & 
13 &  (15) & 
3 &  (2, 5) & 
4 &  (1, 2) \\ 
5 &  (4, 10) & 
6 &  (1, 3) & 
6 &  (4, 5) & 
14 &  (2, 7) & 
15 &  (2, 9) \\ 
16 &  (2, 8) & 
10 &  (7, 8, 9) & 
17 &  (2, 2, 2, 7, 8, 9)  \\
\end{array}\right.$
\end{footnotesize}\end{center}

\smallskip
Note that the characters of the set ${\bf X}_{11}$ are $Syl_p$-equivalent to those of  ${\bf X}_{14}$. Similarly for the pairs $({\bf X}_{12},{\bf X}_{16})$ and $({\bf X}_{13},{\bf X}_{15})$.
\end{tab}

\newpage

\begin{tab}[\ref{TU3}.3-C] \tabC{$SU(3,q)$ with the condition that $3\mid (q+1)$}

\begin{center}\begin{footnotesize}
$\left. \begin{array}{c|c||c|c||c|c||c|c}
j & \mathbf{v} & j & \mathbf{v} & j & \mathbf{v} & j & \mathbf{v} \\\hline 
1 &  (1, 2, 2, 7, 8, 9) & 
2 &  (1, 2, 2, 10) & 
3 &  (1, 2, 7, 8, 13) & 
4 &  (1, 2, 7, 8, 15) \\ 
5 &  (1, 2, 7, 9, 12) & 
6 &  (1, 2, 7, 9, 16) & 
7 &  (1, 2, 8, 9, 11) & 
8 &  (1, 2, 8, 9, 14) \\ 
9 &  (1, 7, 12, 13) & 
10 &  (1, 7, 12, 15) & 
11 &  (1, 7, 13, 16) & 
12 &  (1, 7, 15, 16) \\ 
13 &  (1, 8, 11, 13) & 
14 &  (1, 8, 11, 15) & 
15 &  (1, 8, 13, 14) & 
16 &  (1, 8, 14, 15) \\ 
17 &  (1, 9, 11, 12) & 
18 &  (1, 9, 11, 16) & 
19 &  (1, 9, 12, 14) & 
20 &  (1, 9, 14, 16) \\ 
21 &  (2, 4, 7, 8, 9) & 
22 &  (2, 4, 10) & 
23 &  (2, 5) & 
24 &  (3) \\
25 &  (4, 7, 8, 13) & 
26 &  (4, 7, 8, 15) & 
27 &  (4, 7, 9, 12) & 
28 &  (4, 7, 9, 16) \\
29 &  (4, 8, 9, 11) & 
30 &  (4, 8, 9, 14) \\
\end{array}\right.$
\end{footnotesize}\end{center}
\end{tab}

\begin{tab}[\ref{TU3}.3-D] \tabD{$SU(3,q)$ with the condition that $3\mid (q+1)$}

\smallskip
In view of Lemma \ref{LU3.2}, we have \tabDD{$\psi= 1_G+ 2\chi_{2}+
\chi_7+\chi_8+\chi_9$}

\begin{center} \begin{footnotesize}
 $\left.
\begin{array}{l||l|l}
C_h & St & \psi\\\hline 
C_8(a,b) & -1 & 5-6Re(\omega^{a-b})\\
\end{array}
 \right.$
\end{footnotesize}\end{center}

\smallskip
Here $(a,b)\in I_8=\ZZ_{q+1}\times \ZZ_{q+1}$ with the exceptions: $(q+1)\mid a,b,(a-b)$.
\end{tab}

\subsection{Groups ${}^2B_2(q^2)$}\label{T2B2}

\begin{tab}[\ref{T2B2}-A]  \tabA{${}^2B_2(q^2)$}

\smallskip
\begin{center}\begin{footnotesize}
 $\left.
\begin{array}{l|l||l|l}
h & \deg \chi_h & h & \deg \chi_h\\\hline 
1  &  1  & 2  &  \frac{1}{2} \sqrt{2} q(q-1)(q+1)\\
3  &  \frac{1}{2} \sqrt{2}q(q-1) (q+1) & 4  &  q^4\\
5  &  q^4+1 & 6  &  (q-1) (q^2-\sqrt{2} q+1) (q+1)\\
7  &  (q-1) (q^2+\sqrt{2} q+1) (q+1) & & \\
\end{array}
\right.$

\med 
$\left.\begin{array}{l|l|l}
h & I_h & {\rm exceptions} \\ \hline
1,2,3,4 & - & \\
5 & s\in \ZZ_{q^2-1} & (q^2-1)\mid s\\
6 & k \in\ZZ_{q^2+\sqrt{2}q+1} & (q^2+\sqrt{2}q+1)\mid k\\
7 & u \in \ZZ_{q^2-\sqrt{2}q+1} &  (q^2-\sqrt{2}q+1) \mid u\\
\end{array}\right. $
\end{footnotesize}\end{center}
\end{tab}

\begin{tab}[\ref{T2B2}-B] \tabB{${}^2B_2(q^2)$}

\begin{center}\begin{footnotesize}
$\left. \begin{array}{c|c||c|c||c|c}
h & \mathbf{v} & h & \mathbf{v} & h & \mathbf{v} \\ \hline 
5 & (1,4) & 4 & (1,2,3,6) & 7 & (2,2,3,3,6)\\
\end{array}\right.$
\end{footnotesize}\end{center}
\end{tab}

\begin{tab}[\ref{T2B2}-C] \tabC{${}^2B_2(q^2)$}

\begin{center}\begin{footnotesize}
$\left.\begin{array}{c|c||c|c}
j & \mathbf{v} & j & \mathbf{v}  \\ \hline 
1 & (1,2,3,6) & 2 & (4)
\end{array}\right.$
\end{footnotesize}\end{center}
\end{tab}

\begin{tab}[\ref{T2B2}-D] \tabD{${}^2B_2(q^2)$}

\smallskip
In view of Lemma \ref{L2B2} we have \tabDD{$\psi= 1_G + \chi_{2} + \chi_{3} +
\chi_{6}(k)$}

\smallskip
\begin{center}
\begin{footnotesize}
 $\left.
\begin{array}{l||l|l}
C_h & St & \psi\\\hline
C_6(b) & -1 & 3-2Re((\varphi_8')^{bk})-2Re((\varphi_8')^{q^2bk})
\end{array}
 \right.$
\end{footnotesize}
\end{center}
\end{tab}

\subsection{Groups ${}^2G_2(q^2)$}\label{T2G2}

\begin{tab}[\ref{T2G2}-A] \tabA{${}^2G_2(q^2)$}

\smallskip
\begin{center}\begin{footnotesize}
 $\left.\begin{array}{l|l||l|l}
h & \deg \chi_h & h & \deg \chi_h\\\hline 
1 &  1 &
2 &  \frac{1}{6}  \sqrt{3}  q  (q^2-1)  (q^2+\sqrt{3}  q+1)\\
3 &  \frac{1}{6}  \sqrt{3}  q  (q^2-1)  (q^2-\sqrt{3}  q+1)&
4 &  \frac{1}{6}  \sqrt{3}  q  (q^2-1)  (q^2+\sqrt{3}  q+1)\\
5 &  \frac{1}{6}  \sqrt{3}  q  (q^2-1)  (q^2-\sqrt{3}  q+1)&
6 &  \frac{1}{3}  \sqrt{3} q (q-1)  (q^2+1)    (q+1)\\
7 &  \frac{1}{3}  \sqrt{3} q (q-1)  (q^2+1)    (q+1)&
8 &  q^6\\
9 &  q^4-q^2+1&
10&   q^2  (q^4-q^2+1)\\
11&   (q^2+1)  (q^4-q^2+1)&
12&   (q^2-1)  (q^2+\sqrt{3}  q+1)  (q^2+1)\\
13&   (q^2-1)  (q^4-q^2+1)&
14&   (q^2-1)  (q^2-\sqrt{3}  q+1)  (q^2+1)\\
\end{array}
\right.$
\med

$\left. \begin{array}{l|l|l}
h & I_h & {\rm exceptions} \\ \hline
1,\ldots,10 & - & \\
11 &  k\in \ZZ_{q^2-1} & \frac{q^2-1}{2}\mid k\\
12 &  k \in \ZZ_{q^2-\sqrt{3}q+1} & (q^2-\sqrt{3}q+1)\mid k\\
13 &  (k,l)\in \ZZ_{\frac{q^2+1}{2}}\times \ZZ_{2} &  \frac{q^2+1}{4} \mid k\\
14 &  k \in \ZZ_{q^2+\sqrt{3}q+1} & (q^2-\sqrt{3} q+1)\mid k\\
\end{array}\right.$
\end{footnotesize}\end{center}
\end{tab}

\begin{tab}[\ref{T2G2}-B] Some decompositions of the restrictions of the irreducible
characters on the classes $C_1, C_2, C_5$ for ${}^2G_2(q^2)$. 

\begin{center}\begin{footnotesize}
$\left. \begin{array}{c|c||c|c||c|c||c|c}
h & \mathbf{v} & h & \mathbf{v} & h & \mathbf{v} & h & \mathbf{v}  \\ \hline 
2 & (4) & 3 & ( 5 ) & 6 & ( 7 ) &10 & ( 9, 13 ) \\
 11 & (1, 8) & 11 & (9, 10 ) & 8 & ( 1, 3, 5, 6, 7, 14 ) 
\end{array}\right.$
\end{footnotesize}\end{center}
\end{tab}

\begin{tab}[\ref{T2G2}-C] \tabC{${}^2G_2(q^2)$}

\begin{center}\begin{footnotesize}
$\left.\begin{array}{c|c||c|c}
j & \mathbf{v} & j & \mathbf{v} \\ \hline 
1 & (1,3,5,6,7,14) & 2 & (8)
\end{array}\right.$
\end{footnotesize}\end{center}
\end{tab}

\begin{tab}[\ref{T2G2}-D] \tabD{${}^2G_2(q^2)$}

\smallskip
In view of Lemma \ref{L2G2} we have \tabDD{$\psi=  1_G + \chi_{3} + \chi_{5} + \chi_{6} +
\chi_{7} + \chi_{14}(k)$}

\begin{center}\begin{footnotesize}
 $\left.
\begin{array}{l||l|l}
C_h & St & \psi\\\hline 
C_{14}(i) & -1 & 5-2Re((\varphi_{12}'')^{ik}) -2Re((\varphi_{12}'')^{(\sqrt{3}q+1)ik})-
2Re((\varphi_{12}'')^{(\sqrt{3}q+2)ik})\\
\end{array}
 \right.$
\end{footnotesize}\end{center}
\end{tab}

\subsection{Groups $GL(3,q)$ and $SL(3,q)$}\label{TGL3}

\begin{tab}[\ref{TGL3}.1-A] \tabA{$GL(3,q)$}

\smallskip
\begin{center}\begin{footnotesize}
 $\left.
\begin{array}{l|l||l|l}
h & \deg \chi_h & h & \deg \chi_h\\\hline 
1 &  1 & 2 &  q  (q+1)\\
3 &  q^3 & 4 &  q^2+q+1\\
5 &  q  (q^2+q+1) & 6 &  (q+1)  (q^2+q+1) \\
7 &  (q-1)  (q^2+q+1) & 8 &  (q-1)^2  (q+1)\\
\end{array}
\right.$
\med 

$\left. \begin{array}{l|l|l}
h & I_h & {\rm exceptions} \\ \hline
1,2,3 & n \in \ZZ_{q-1} & \\
4,5  & (n,m)\in \ZZ_{q-1}\times \ZZ_{q-1} & (q-1)\mid (n-m)\\
6 & (n,m,l) \in \ZZ_{q-1}\times \ZZ_{q-1}\times \ZZ_{q-1} & (q-1)\mid (n-m), (n-l),(m-l)\\
7 & (m,n)\in \ZZ_{q-1}\times \ZZ_{q^2-1} & (q+1)\mid n\\
8 &  n\in \ZZ_{q^3-1} & (q^2+q+1)\mid n\\
\end{array}\right.$
\end{footnotesize}\end{center}
\end{tab}

\begin{tab}[\ref{TGL3}.1-B] \tabB{$GL(3,q)$}

\begin{center}\begin{footnotesize}
$\left. \begin{array}{c|c||c|c||c|c||c|c||c|c}
h & \mathbf{v} & h & \mathbf{v} & h & \mathbf{v} & h & \mathbf{v} & h & \mathbf{v} \\ \hline 
3 &  (1, 7) & 
4 &  (1, 2) & 
5 &  (2, 3) &
5 &  (4, 7) & 
6 &  (4, 5) \\ 
\end{array}\right.$
\end{footnotesize}\end{center}
\end{tab}

\begin{tab}[\ref{TGL3}.1-C] \tabC{$GL(3,q)$}

\begin{center}\begin{footnotesize}
$\left. \begin{array}{c|c||c|c}
j & \mathbf{v} & j & \mathbf{v}  \\ \hline 
1 &  (1, 7) & 
2 &  (3) \\
\end{array}\right.$
\end{footnotesize}\end{center}
\end{tab}

\begin{tab}[\ref{TGL3}.1-D] \tabD{$GL(3,q)$}

\smallskip
In view of Lemma \ref{LGL3} we have \tabDD{$\psi=1_H+\chi_7(q-1,n)$, with $n=(q-1)n'$}

\begin{center}\begin{footnotesize}
$\left.
\begin{array}{l||l|l}
C_h & St & \psi\\\hline
C_{7}(a,b) & -1 & 1-2Re(\xi_1^{b n'})\\
\end{array}
 \right.$
\end{footnotesize}\end{center}
\end{tab}

\begin{tab}[\ref{TGL3}.2-A] \tabA{$SL(3,q)$  with the condition that $3\mid (q-1)$}

\smallskip
\begin{center}\begin{footnotesize}
 $\left.
\begin{array}{l|l||l|l}
h & \deg \chi_h & h & \deg \chi_h\\\hline 
1 &  1 &
2 &  q(q+1)\\
3 & q^3 & 
4 &  q^2+q+1\\
5 &  q(q^2+q+1) &
6,7,8 &  \frac{1}{3}(q+1)(q^2+q+1)\\
9 &  (q+1)(q^2+q+1) & 
10 &  (q-1)(q^2+q+1) \\
11,\ldots,16 &  \frac{1}{3}(q-1)^2(q+1)& 
17 &  (q-1)^2(q+1)\\
\end{array}
\right.$

\med
$\left. \begin{array}{l|l|l}
h & I_h & {\rm exceptions} \\ \hline
1,2,3,6,7,8,11,\ldots,16  & - & \\
4,5 & n\in \ZZ_{q-1} & (q-1)\mid n\\
9 & (n,m)\in \ZZ_{q-1}^2 & (q-1)\mid n,m, (n-m) \\
10 & n\in \ZZ_{q^2-1} & (q+1)\mid n\\
17 & n\in \ZZ_{q^2+q+1} & \frac{q^2+q+1}{3}\mid n\\
\end{array}\right.$
\end{footnotesize}\end{center}
\end{tab}

\begin{tab}[\ref{TGL3}.2-B] \tabB{$SL(3,q)$  with the condition that $3\mid (q-1)$}

\begin{center}\begin{footnotesize}
$\left. \begin{array}{c|c||c|c||c|c||c|c||c|c||c|c}
h & \mathbf{v} & h & \mathbf{v} & h & \mathbf{v} & h & \mathbf{v} & h & \mathbf{v}  & h &
\mathbf{v} \\ \hline 
11 &  (14) & 
12 &  (16) & 
13 &  (15) & 
3 &  (1, 10) & 
4 &  (1, 2) &
5 &  (2, 3) \\ 
5 &  (4, 10) & 
6 &  (2, 14) & 
7 &  (2, 16) & 
8 &  (2, 15) & 
9 &  (4, 5) & 
17 &  (14, 15, 16) \\
\end{array}\right.$
\end{footnotesize}\end{center}
\end{tab}

\begin{tab}[\ref{TGL3}.2-C] \tabC{$SL(3,q)$ with the condition that $3\mid (q-1)$}

\begin{center}\begin{footnotesize}
$\left. \begin{array}{c|c||c|c}
j & \mathbf{v} & j & \mathbf{v}  \\ \hline 
1 &  (1, 10) & 
2 &  (3) \\
\end{array}\right.$
\end{footnotesize}\end{center}
\end{tab}

\begin{tab}[\ref{TGL3}.3-A] \tabA{$SL(3,q)$  with the condition that $3\nmid (q-1)$}

\begin{center}\begin{footnotesize}
 $\left.
\begin{array}{l|l||l|l}
h & \deg \chi_h & h & \deg \chi_h\\\hline 
1  & 1 & 
2  & q(q+1)\\
3  & q^3 & 
4  & q^2+q+1\\
5  & q(q^2+q+1) & 
6  & (q+1)(q^2+q+1)\\
7  & (q-1)(q^2+q+1) & 
8  & (q-1)^2(q+1)\\
\end{array}
\right.\quad
\left. \begin{array}{l|l|l}
h & I_h & {\rm exceptions} \\ \hline
1,2,3 & - & \\
4,5 & n\in \ZZ_{q-1} & (q-1)\mid n\\
6 & (n,m)\in \ZZ_{q-1}^2 & (q-1)\mid (n-m)\\ 
7 & n\in \ZZ_{q^2-1} & (q+1)\mid n \\
8 & n\in \ZZ_{q^2+q+1} & (q^2+q+1)\mid n\\
\end{array}\right.$
\end{footnotesize}
\end{center}
\end{tab}

\begin{tab}[\ref{TGL3}.3-B] \tabB{$SL(3,q)$ with the condition that $3\nmid (q-1)$}

\begin{center}\begin{footnotesize}
$\left.\begin{array}{c|c||c|c||c|c||c|c||c|c}
h & \mathbf{v} & h & \mathbf{v} & h & \mathbf{v} & h & \mathbf{v} & h & \mathbf{v} \\ \hline 
3 &  (1, 7) & 
4 &  (1, 2) & 
5 &  (2, 3) & 
5 &  (4, 7) & 
6 &  (4, 5) \\
\end{array}\right.$
\end{footnotesize}
\end{center}
\end{tab}
\newpage

\begin{tab}[\ref{TGL3}.3-C] \tabC{$SL(3,q)$  with the condition that $3\nmid (q-1)$}

\begin{center}\begin{footnotesize}
 $\left.\begin{array}{c|c||c|c}
j & \mathbf{v} & j & \mathbf{v}  \\ \hline 
1 &  (1, 7) & 
2 &  (3) \\
\end{array}\right.$
\end{footnotesize}
\end{center}
\end{tab}

\subsection{Groups $U(4,q)$}\label{TU4}

\begin{tab}[\ref{TU4}-A] \tabA{$U(4,q)$}

\smallskip
\begin{center}\begin{footnotesize}
$\left.
\begin{array}{l|l||l|l}
h & \deg \chi_h & h & \deg \chi_h\\\hline 
1 &    q^6 & 2 &  q^3  (q^2-q+1)\\
3 &  q^2  (q^2+1) & 4 &  q  (q^2-q+1)\\
5 &  1 & 6 &  (q-1)  (q^2+1)  q^3\\
7 &  q  (q^2+1)  (q-1)^2 & 8 &  (q-1)  (q^2+1)\\
9 &  q^2  (q^2+1)  (q^2-q+1) & 10 &  q  (q^2-q+1)  (q^2+1)\\
11 &  q  (q^2-q+1)  (q^2+1) & 12 &  (q^2+1)  (q^2-q+1)\\
13 &  q^2  (q^2-q+1)  (q+1)^2 & 14 &  (q^2-q+1)  (q+1)^2\\
15 &  q  (q-1)  (q^2+1)  (q^2-q+1) & 16 & (q-1)  (q^2+1)  (q^2-q+1)\\
17 &  q  (q+1)  (q^2-q+1)  (q^2+1) & 18 &  (q+1)  (q^2-q+1)  (q^2+1)\\
19 &  (q^2-q+1)  (q^2+1)  (q-1)^2 & 20 &  (q-1)  (q+1)  (q^2+1)  (q^2-q+1)\\
21 &  (q^2+1)  (q^2-q+1)  (q+1)^2 & 22 &  (q^2+1)  (q-1)^2  (q+1)^2\\
23 &  (q-1)  (q^2-q+1)  (q+1)^3 & & \\
\end{array}
\right.$

\smallskip
{}
\med
$\left.\begin{array}{l|l|l}
h & I_h & {\rm exceptions} \\ \hline
1\ldots,5 & k_1\in \ZZ_{q+1} & \\
6,\ldots,12 & (k_1,k_2)\in \ZZ_{q+1}\times \ZZ_{q+1}  & (q+1)\mid (k_1-k_2)\\
13,14 & k_1\in \ZZ_{q^2-1} & (q-1)\mid k_1 \\
15,16 &  (k_1,k_2,k_3)\in \ZZ_{q+1}\times \ZZ_{q+1}\times\ZZ_{q+1} & (q+1)\mid (k_i-k_j),
i\neq j\\
17,18 & (k_1,k_2)\in \ZZ_{q+1}\times \ZZ_{q^2-1} & (q^2-1)\mid (q-1)k_1-k_2\\
{} & {} & (q-1)\mid k_2 \\ 
19 & (k_1,k_2,k_3,k_4)\in \ZZ_{q+1}^4 &  (q+1)\mid (k_i-k_j), i\neq j\\
20 & (k_1,k_2,k_3)\in \ZZ_{q^2-1}\times \ZZ_{q+1}\times \ZZ_{q+1} & (q-1)\mid k_1\\
{} & {} & (q+1)\mid (k_2-k_3)\\
21 & (k_1,k_2)\in \ZZ_{q^2-1}\times \ZZ_{q^2-1} &  (q-1)\mid k_i\\
 {} & {} & (q^2-1) \mid qk_1+k_2\\
{} & {} &  (q^2-1)\mid (k_1-k_2)\\
22 & (k_1,k_2)\in \ZZ_{q^3+1}\times \ZZ_{q+1} & (q^2-q+1)\mid k_1 \\
23 & k_1\in \ZZ_{q^4-1} & (q^3+q-q^2-1),(q^2+1)\mid k_1\\
\end{array}\right.$
\end{footnotesize}
\end{center}
\end{tab}

\begin{tab}[\ref{TU4}-B] \tabB{$U(4,q)$}

\begin{center}\begin{footnotesize}
$\left.\begin{array}{c|c||c|c||c|c||c|c||c|c||c|c}
h & \mathbf{v} & h & \mathbf{v} & h & \mathbf{v}  & h & \mathbf{v} & h & \mathbf{v}& h & \mathbf{v}\\ \hline 
10 &  (11) & 
1 &  (2, 6) & 
3 &  (8, 12) & 
4 &  (5, 8) & 
6 &  (7, 15) &
9 &  (3, 6) \\
9 &  (11, 15) & 
9 &  (12, 20) & 
11 &  (2, 4) & 
11 &  (3, 7) &
11 &  (12, 16) & 
13 &  (14, 23) \\ 
15 &  (16, 19) & 
16 &  (7, 8) & 
17 &  (9, 11) & 
17 &  (18, 20) &
18 &  (11, 12) & 
20 &  (6, 8) \\
20 &  (15, 16) & 
21 &  (13, 14) & 
21 &  (17, 18) & 
22 &  (7, 7, 7, 19) & 
14 &  (5, 5, 8, 8, 12) \\
\end{array}\right.$
\end{footnotesize}
\end{center}
\end{tab}

\begin{tab}[\ref{TU4}-C] \tabC{$U(4,q)$}

\begin{center}\begin{footnotesize}
$\left.\begin{array}{c|c||c|c||c|c||c|c||c|c}
j & \mathbf{v} & j & \mathbf{v} & j & \mathbf{v}& j & \mathbf{v} & j & \mathbf{v} \\
\hline 
1 &  (1) & 
2 &  (2, 6) & 
3 &  (2, 7, 7, 8, 19) &
4 &  (2, 7, 15) & 
5 &  (2, 7, 16, 19)   \\
\end{array}\right.$
\end{footnotesize}
\end{center}
\end{tab}

\newpage
\subsection{Groups $U(5,q)$}\label{TU5}

\begin{tab}[\ref{TU5}-A] \tabA{$U(5,q)$}

\begin{center}\begin{footnotesize}
$\left.
\begin{array}{l|l||l|l}
h & \deg \chi_h & h & \deg \chi_h \\\hline 
1  & q^{10} &
2  & q^6(q-1)(q^2+1)\\
5  & q^2(q^4-q^3+q^2-q+1)&
8  & q^6(q^4-q^3+q^2-q+1) \\
9  & q^3(q^2-q+1)(q^4-q^3+q^2-q+1)&
15 &  q^2(q-1)(q^2+1)(q^4-q^3+q^2-q+1) \\
16 &  q(q-1)(q^2+1)(q^4-q^3+q^2-q+1)&
19 &  q^3(q-1)(q^2+1)(q^4-q^3+q^2-q+1) \\
20 &  q(q^2+1)(q^4-q^3+q^2-q+1)(q-1)^2 &
21 &  (q-1)(q^2+1)(q^4-q^3+q^2-q+1) \\
31 &  q(q-1)(q^2-q+1)(q^4-q^3+q^2-q+1)(q^2+1)&
32 &  (q-1)(q^2-q+1)(q^4-q^3+q^2-q+1)(q^2+1) \\
37 &  (q^2-q+1)(q^4-q^3+q^2-q+1)(q^2+1)(q-1)^2\\
\end{array} \right.$
\med

$\left.\begin{array}{l|l|l}
h & I_h & {\rm exceptions} \\ \hline
1,2,5 & k_1\in \ZZ_{+1} & \\
8,9,15,16 & (k_1,k_2)\in \ZZ_{q+1}^2 & (q+1)\mid k_1-k_2\\
19,20,21 & (k_1,k_2,k_3)\in \ZZ_{q+1}^3 & (q+1)\mid k_i-k_j, i\neq j\\ 
31, 32 & (k_1,k_2,k_3,k_4)\in \ZZ_{q+1}^4 & (q+1)\mid k_i-k_j, i\neq j\\ 
37 & (k_1,k_2,k_3,k_4,k_5)\in \ZZ_{q+1}^5 & (q+1)\mid k_i-k_j, i\neq j\\ 
\end{array}\right.$
\end{footnotesize}
\end{center}
\end{tab}

\begin{tab}[\ref{TU5}-B] \tabB{$U(5,q)$}

\begin{center}\begin{footnotesize}
$\left.\begin{array}{c|c||c|c||c|c||c|c||c|c||c|c}
h & \mathbf{v} & h & \mathbf{v} & h & \mathbf{v}  & h & \mathbf{v} & h & \mathbf{v}& h & \mathbf{v}\\ \hline 
26 &  (27) & 
1 &  (2, 8) & 
3 &  (11, 16) & 
4 &  (6, 17) & 
8 &  (9, 19) & 
9 &  (5, 15) \\ 
10 &  (3, 5) & 
10 &  (16, 17) & 
10 &  (21, 28) & 
11 &  (12, 21) & 
12 &  (6, 7) & 
13 &  (3, 8) \\
13 &  (14, 19) & 
13 &  (15, 25) & 
14 &  (2, 4) &
14 &  (3, 9) & 
14 &  (10, 15) & 
14 &  (16, 27) \\ 
14 &  (17, 23) & 
15 &  (16, 20) & 
17 &  (5, 11) & 
17 &  (18, 21) & 
18 &  (5, 12) & 
19 &  (20, 31)\\ 
22 &  (13, 14) &
22 &  (23, 33) & 
23 &  (2, 6) & 
23 &  (15, 16) & 
24 &  (17, 18) & 
25 &  (10, 19) \\
25 &  (27, 31) & 
25 &  (28, 38) & 
27 &  (9, 11) & 
27 &  (10, 20) & 
27 &  (15, 17) & 
27 &  (28, 32) \\ 
28 &  (16, 18) & 
29 &  (30, 42) & 
31 &  (32, 37) & 
32 &  (20, 21)& 
33 &  (13, 17) & 
33 &  (25, 27) \\ 
33 &  (34, 38) & 
34 &  (14, 18) & 
34 &  (23, 24) & 
34 &  (27, 28) & 
35 &  (36, 40) & 
38 &  (19, 21) \\ 
38 &  (31, 32) & 
39 &  (22, 24) & 
39 &  (29, 30) & 
39 &  (33, 34) & 
41 &  (35, 36) & 
36 &  (2, 6, 16, 21) \\
40 &  (20, 20, 20, 37) \\
\end{array}\right.$
\end{footnotesize}
\end{center}
\end{tab}

\begin{tab}[\ref{TU5}-C] \tabC{$U(5,q)$}

\begin{center}\begin{footnotesize}
$\left.\begin{array}{c|c||c|c||c|c}
j & \mathbf{v} & j & \mathbf{v} & j & \mathbf{v} \\ \hline 
1 &  (2, 5, 15, 19) & 
2 &  (2, 5, 15, 20, 20, 21, 37) &
3 &  (2, 5, 15, 20, 31) \\ 
4 &  (2, 5, 15, 20, 32, 37) & 
5 &  (2, 5, 16, 19, 20) & 
6 &  (2, 5, 16, 20, 20, 20, 21, 37) \\
7 &  (2, 5, 16, 20, 20, 31) & 
8 &  (2, 5, 16, 20, 20, 32, 37) & 
9 &  (2, 5, 16, 21, 40) \\
10 &  (2, 9, 19) & 
11 &  (2, 9, 20, 20, 21, 37) &
12 &  (2, 9, 20, 31) \\ 
13 &  (2, 9, 20, 32, 37) &
14 &  (2, 8) & 
15 &  (1) \\
\end{array}\right.$
\end{footnotesize}
\end{center}
\end{tab}

\subsection{Groups $Sp(4,q)$}\label{TSp4}

\begin{tab}[\ref{TSp4}.1-A] \tabA{$Sp(4,q)$ with $q$ even}

\smallskip
\begin{center}\begin{footnotesize}
 $\left.
\begin{array}{l|l||l|l||l|l}
h & \deg \chi_h  & h & \deg \chi_h   & h & \deg \chi_h  \\\hline 
1 &  1 &
2 &  \frac{1}{2} q(q+1)^2&
3,4 &  \frac{1}{2}q(q^2+1) \\
5 &  \frac{1}{2}q(q-1)^2&
6 &  q^4 & 
7 &  (q+1)(q^2+1)\\
8 &  q(q+1)(q^2+1)&
9 &  (q-1)(q^2+1)&
10&   q(q-1)(q^2+1)\\
11 &  (q+1)(q^2+1)&
12 &  q(q+1)(q^2+1)&
13 &  (q-1)(q^2+1)\\
14 &  q(q-1)(q^2+1)&
15 &  (q+1)^2(q^2+1)&
16,17  & (q-1)(q+1)(q^2+1)\\
18 &  (q-1)^2(q+1)^2&
19 &  (q-1)^2(q^2+1)\\
\end{array}
\right.$

\med
$\left.\begin{array}{l|l|l}
h & I_h & {\rm exceptions} \\ \hline
1,\ldots,6 & - & \\
7,8,11,12 & k\in \ZZ_{q-1} & (q-1)\mid k\\
9,10,13,14 & k\in \ZZ_{q+1} & (q+1)\mid k \\
15 & (k,l)\in \ZZ_{q-1}^2 & (q-1)\mid k\pm l \\
16 & k\in \ZZ_{q^2-1} & (q\pm 1)\mid k \\
17 & (k,l)\in \ZZ_{q-1}\times \ZZ_{q+1} & (q-1)\mid k, (q+1)\mid l\\
18 & k\in \ZZ_{q^2+1} & (q^2+1)\mid k \\
19 & (k,l)\in \ZZ_{q+1}^2 & (q+1)\mid k\pm l\\
\end{array}\right.$
\end{footnotesize}
\end{center}
\end{tab}

\begin{tab}[\ref{TSp4}.1-B] \tabB{$Sp(4,q)$ with $q$ even}

\begin{center}
\begin{footnotesize}
 $\left.\begin{array}{c|c||c|c||c|c||c|c||c|c}
h & \mathbf{v} & h & \mathbf{v} & h & \mathbf{v}  & h & \mathbf{v} & h & \mathbf{v}\\ \hline 
8 &  (7, 17) & 
10 &  (9, 19) & 
12 &  (11, 16) & 
14 &  (13, 19) & 
15 &  (7, 8) \\
15 &  (11, 12) & 
16 &  (13, 14) & 
17 &  (9, 10) & 
7 &  (1, 2, 3) & 
11 &  (1, 2, 4) \\
6 &  (1, 9, 13, 19) & 
6 &  (3, 5, 9, 19) & 
6 &  (4, 5, 13, 19) \\
\end{array}\right.$
\end{footnotesize}
\end{center}
\end{tab}

\begin{tab}[\ref{TSp4}.1-C] \tabC{$Sp(4,q)$ with $q$ even}

\begin{center}
\begin{footnotesize}
$\left.\begin{array}{c|c||c|c||c|c||c|c}
j & \mathbf{v} & j & \mathbf{v} & j & \mathbf{v}  & j & \mathbf{v} \\ \hline 
1 &  (1, 9, 13, 19) & 
2 &  (1, 9, 14) & 
3 &  (1, 10, 13) &
4 &  (3, 5, 9, 19) \\
5 &  (3, 5, 10) & 
6 &  (4, 5, 13, 19) &
7 &  (4, 5, 14) & 
8 &  (6) \\
\end{array}\right.$
\end{footnotesize}
\end{center}
\end{tab}

\begin{tab}[\ref{TSp4}.2-A] \textsc{chevie} sets ${\bf X}_h$ and their degrees for
$Sp(4,q)$ with $q$ odd.
\medskip

In this table, for  reader's sake,  we additionally provide the notation of
\cite{Sri} for the \ir characters of $G$.

\begin{center}\begin{footnotesize}
 $\left.
\begin{array}{l|l|l||l|l|l}
h & \textrm{\cite{Sri}} & \deg \chi_h & h & \textrm{\cite{Sri}} & \deg \chi_h \\ \hline 
  1 &                   \chi_1  &               q^4-2q^2+1 &
  2 &                   -\chi_2 &                     q^4-1 \\
  3 &                    \chi_3 &     q^4+2q^3+2q^2+2q+1 &
  4 &                    \chi_4 &     q^4-2q^3+2q^2-2q+1 \\
  5 &                   -\chi_5 &                     q^4-1 & 
  6 &                   -\chi_6 &               q^3-q^2+q-1 \\
  7 &                    \chi_7 &             q^4-q^3+q^2-q  & 
  8 &                    \chi_8 &               q^3+q^2+q+1 \\
  9 &                    \chi_9 &             q^4+q^3+q^2+q &
 10 &                    -\xi_1 &               q^3-q^2+q-1 \\
 11 &                   -\xi_1' &             q^4-q^3+q^2-q &
 12 &                     \xi_3 &               q^3+q^2+q+1 \\
 13 &                    \xi_3' &             q^4+q^3+q^2+q &
 14 &                 -\xi_{21} &               (q^4-1)/2 \\
\end{array}
\right.$\\
\end{footnotesize}
\end{center}
\end{tab}

\newpage

\begin{tab}[\ref{TSp4}.2-B] \tabB{$Sp(4,q)$ with $q$ odd}

\begin{center}\begin{footnotesize}
$\left.\begin{array}{c|c||c|c||c|c||c|c||c|c||c|c}
h & \mathbf{v} & h & \mathbf{v} & h & \mathbf{v}  & h & \mathbf{v}& h & \mathbf{v} & h & \mathbf{v}  \\ \hline 
2 &  (6, 7) & 
3 &  (8, 9) & 
3 &  (12, 13) & 
3 &  (18, 19) &
4 &  (16, 17) & 
5 &  (10, 11) \\
5 &  (14, 15)  &
5 &  (20, 21) & 
7 &  (4, 6) & 
9 &  (2, 8) &
10 &  (22, 23) & 
11 &  (4, 10) \\ 
11 &  (14, 16) & 
11 &  (15, 17) & 
11 &  (24, 25) & 
12 &  (26, 27) & 
13 &  (5, 12) &
13 &  (18, 20) \\ 
13 &  (19, 21) & 
13 &  (28, 29)& 
14 &  (10, 17) & 
14 &  (23, 24) & 
14 &  (35, 38) & 
15 &  (10, 16) \\ 
15 &  (22, 25) & 
15 &  (36, 37) & 
18 &  (12, 21) & 
18 &  (26, 28) & 
19 &  (12, 20) & 
19 &  (27, 29) \\ 
20 &  (22, 24) &
20 &  (35, 37) & 
21 &  (23, 25) &
21 &  (36, 38) & 
22 &  (37, 40)&
23 &  (38, 40) \\ 
24 &  (17, 22) & 
25 &  (16, 23) & 
26 &  (34, 42) & 
27 &  (33, 42) & 
28 &  (15, 26) & 
28 &  (21, 27) \\ 
28 &  (25, 30) & 
28 &  (31, 41) & 
28 &  (36, 39) & 
29 &  (14, 27) & 
29 &  (20, 26) & 
29 &  (24, 30) \\ 
29 &  (32, 41) & 
29 &  (35, 39) & 
30 &  (22, 26) & 
30 &  (23, 27) & 
30 &  (39, 40) & 
30 &  (41, 42) \\ 
31 &  (25, 42) & 
32 &  (24, 42) & 
33 &  (37, 44) & 
34 &  (38, 44) & 
35 &  (24, 40) & 
36 &  (25, 40) \\ 
39 &  (26, 37) & 
39 &  (27, 38) & 
41 &  (22, 34) & 
41 &  (23, 33) & 
43 &  (31, 35) &
43 &  (32, 36) \\
\end{array}\right.$
\end{footnotesize}

\medskip
\begin{footnotesize}
$\left.\begin{array}{c|c}
h & \mathbf{v} \\\hline
1 &  (16, 17, 40, 40, 40, 40) \\
\end{array}\right.$
\end{footnotesize}
\end{center}

\end{tab}

\begin{tab}[\ref{TSp4}.2-C] \tabC{$Sp(4,q)$ with $q$ odd}

\begin{center}
\begin{footnotesize}
$\left.\begin{array}{c|c||c|c||c|c}
j & \mathbf{v} & j & \mathbf{v} & j & \mathbf{v} \\ \hline 
1 &  (4, 6, 10, 44) & 
2 &  (4, 6, 22, 23, 44) & 
3 &  (4, 6, 22, 34, 40) \\
4 &  (4, 6, 22, 38, 40, 44) & 
5 &  (4, 6, 23, 33, 40) &
6 &  (4, 6, 23, 37, 40, 44) \\
7 &  (4, 6, 33, 38, 40, 40) & 
8 &  (4, 6, 34, 37, 40, 40) & 
9 &  (4, 6, 37, 38, 40, 40, 44) \\
10 &  (4, 6, 40, 41) & 
11 &  (4, 10, 40, 42) & 
12 &  (4, 22, 23, 40, 42) \\
13 &  (4, 22, 38, 40, 40, 42) & 
14 &  (4, 23, 37, 40, 40, 42) & 
15 &  (4, 37, 38, 40, 40, 40, 42) \\
16 &  (6, 10, 16, 17, 44) & 
17 &  (6, 11, 44) & 
18 &  (6, 14, 16, 44) \\
19 &  (6, 15, 17, 44) & 
20 &  (6, 16, 17, 22, 23, 44) &
21 &  (6, 16, 17, 22, 34, 40) \\
%
22 &  (6, 16, 17, 22, 38, 40, 44) & 
23 &  (6, 16, 17, 23, 33, 40) & 
24 &  (6, 16, 17, 23, 37, 40, 44) \\
25 &  (6, 16, 17, 33, 38, 40, 40) & 
26 &  (6, 16, 17, 34, 37, 40, 40) & 
27 &  (6, 16, 17, 37, 38, 40, 40, 44) \\
28 &  (6, 16, 17, 40, 41) & 
29 &  (6, 16, 23, 24, 44) & 
30 &  (6, 16, 24, 34, 40) \\
31 &  (6, 16, 24, 38, 40, 44) & 
32 &  (6, 16, 34, 35) & 
33 &  (6, 16, 35, 38, 44) \\
34 &  (6, 17, 22, 25, 44) & 
35 &  (6, 17, 25, 33, 40) & 
36 &  (6, 17, 25, 37, 40, 44) \\
37 &  (6, 17, 33, 36) & 
38 &  (6, 17, 36, 37, 44) & 
39 &  (6, 24, 25, 44) \\
40 &  (7, 10, 44) & 
41 &  (7, 22, 23, 44) & 
42 &  (7, 22, 34, 40) \\
43 &  (7, 22, 38, 40, 44) & 
44 &  (7, 23, 33, 40) & 
45 &  (7, 23, 37, 40, 44) \\
46 &  (7, 33, 38, 40, 40) & 
47 &  (7, 34, 37, 40, 40) & 
48 &  (7, 37, 38, 40, 40, 44) \\
49 &  (7, 40, 41) & 
50 &  (10, 16, 17, 40, 42) & 
51 &  (11, 40, 42) \\
52 &  (14, 16, 40, 42) & 
53 &  (15, 17, 40, 42) & 
54 &  (16, 17, 22, 23, 40, 42) \\
55 &  (16, 17, 22, 38, 40, 40, 42) & 
56 &  (16, 17, 23, 37, 40, 40, 42) & 
57 &  (16, 17, 37, 38, 40, 40, 40, 42) \\
58 &  (16, 23, 24, 40, 42) & 
59 &  (16, 23, 32, 40) & 
60 &  (16, 23, 35, 42) \\
61 &  (16, 24, 38, 40, 40, 42) & 
62 &  (16, 32, 38, 40, 40) & 
63 &  (16, 35, 38, 40, 42) \\
64 &  (17, 22, 25, 40, 42) & 
65 &  (17, 22, 31, 40) & 
66 &  (17, 22, 36, 42) \\
67 &  (17, 25, 37, 40, 40, 42) & 
68 &  (17, 31, 37, 40, 40) & 
69 &  (17, 36, 37, 40, 42) \\
70 &  (24, 25, 40, 42) & 
71 &  (24, 31, 40) &
72 &  (24, 36, 42) \\ 
73 &  (25, 32, 40) & 
74 &  (25, 35, 42) & 
75 &  (31, 35) \\
76 &  (32, 36) & 
77 &  (43) \\
\end{array}\right.$ 
\end{footnotesize}\end{center}
\end{tab}

\newpage
\begin{tab}[\ref{TSp4}.2-D] \tabD{$Sp(4,q)$ with $q$ odd and $7\mid (q+1)$}
\medskip

In view of Lemma \ref{LSp4.1} we have
\tabDD{$\psi=1_G+\chi_4(a,b)+\chi_6(c)+\chi_{10}(d)$} 

\begin{center}
\begin{footnotesize}
$\left.
\begin{array}{l||l|l}
C_i & St & \psi\\\hline 
B_2(i) & -1 & 1-\beta_{ic}\\
B_4(i,j) & 1&  1+\beta_{ia}\beta_{jb}+\beta_{ib}\beta_{ja}-\beta_{ic}\beta_{jc}-\beta_{id}-\beta_{jd}\\
B_5(i)& -1 & 1-\beta_{id} \\
B_6(i)& -q & q(1-\beta_{2ic})\\
C_1(i) & -q & q(1-\beta_{id}) \\
\end{array}
 \right .$
\end{footnotesize}\end{center}
\noindent Here, $\beta_{t}=2Re(\xi_1^{t})$.
\end{tab}

\begin{tab}[\ref{TSp4}.3-A] \textsc{chevie} sets ${\bf X}_h$ and their degrees for
$CSp(4,q)$ with $q$ odd.\\

Here, for the reader's sake, we also recall  the notation of \cite{Sh} for \ir characters
of $H=CSp(4,q)$. 

\smallskip
\begin{center}\begin{footnotesize}
 $\left.
\begin{array}{l|l|l||l|l|l}
h & \textrm{\cite{Sh}} & \deg \chi_h & h & \textrm{\cite{Sh}} & \deg \chi_h \\ \hline 
1 & X_1(\lambda,\mu)  & (q+1)^2(q^2+1) &
2 & X_2(\Lambda)  & q^4-1 \\
3 & X_3(\Lambda,\nu)  & q^4-1 &
4 & X_4(\theta)  & (q^2-1)^2 \\
5 & X_5(\Lambda,\omega)  & (q^2+1)(q-1)^2 & 
6 & X_1(\lambda)  & (q+1)(q^2+1) \\
7 & X_2(\lambda)  & q(q+1)(q^2+1) &
8 & X_3(\lambda)  & (q+1)(q^2+1)\\
9 & X_4(\lambda)  & q(q+1)(q^2+1) &
10 & X_5(\omega)  & (q-1)(q^2+1)\\
11 & X_6(\omega)  & q(q-1)(q^2+1)&
12 & X_7(\Lambda)  & (q-1)(q^2+1)\\
13 & X_8(\Lambda)  & q(q-1)(q^2+1) &
14 & \tau_1  & q^2+1\\
15 & \tau_2  & q(q^2+1) &
16 & \tau_3  & q^2(q^2+1) \\
17 & \tau_4(\lambda')  & q^2-1 &
18 & \tau_5(\lambda')  & q^2(q^2-1) \\
19 & \theta_1  & q(q+1)^2/2 &
20 & \theta_2  & q(q-1)^2/2 \\
21 & \theta_3  & q(q^2+1)/2 &
22 & \theta_4  & q(q^2+1)/2 \\
23 & \theta_5  & q^4 &
24 & \theta_0(\lambda) & 1 \\
\end{array}
\right.$\\
\end{footnotesize}
\end{center}
\end{tab}

\begin{tab}[\ref{TSp4}.3-B] \tabB{$CSp(4,q)$ with $q$ odd}

\begin{center}\begin{footnotesize}
 $\left.\begin{array}{c|c||c|c||c|c||c|c||c|c}
h & \mathbf{v} & h & \mathbf{v} & h & \mathbf{v}  & h & \mathbf{v} & h & \mathbf{v}\\ \hline 
1 &  (6, 7) & 
1 &  (8, 9) & 
2 &  (10, 11) & 
3 &  (12, 13) & 
3 &  (17, 18) \\
7 &  (3, 6) & 
8 &  (14, 15) & 
9 &  (2, 8) & 
9 &  (15, 16) & 
11 &  (5, 10) \\ 
13 &  (5, 12) & 
15 &  (10, 14) & 
15 &  (19, 20) & 
15 &  (21, 22) & 
16 &  (2, 14) \\ 
16 &  (11, 15) & 
18 &  (4, 17) & 
6 &  (14, 22, 22) & 
6 &  (19, 22, 24) & 
12 &  (17, 20, 20) \\
14 &  (17, 24, 24) & 
19 &  (17, 22, 24) & 
21 &  (17, 20, 24) \\
\end{array}\right.$
\med

$\left.\begin{array}{c|c||c|c||c|c}
h & \mathbf{v} & h & \mathbf{v} & h & \mathbf{v} \\ \hline 
4 &  (5, 20, 20, 20, 20) &
23 &  (5, 10, 17, 20, 20, 24) & 
23 &  (5, 17, 20, 20, 20, 22) \\
\end{array}\right.$
\end{footnotesize}
\end{center}
\end{tab}

\begin{tab}[\ref{TSp4}.3-C] \tabC{$CSp(4,q)$ with $q$ odd}

\begin{center}\begin{footnotesize}
$\left. \begin{array}{c|c||c|c||c|c||c|c}
j & \mathbf{v} & j & \mathbf{v} & j & \mathbf{v}  & j & \mathbf{v} \\ \hline 
1 &  (5, 10, 12, 24) & 
2 &  (5, 10, 17, 20, 20, 24) & 
3 &  (5, 10, 20, 21) & 
4 &  (5, 12, 20, 22) \\
5 &  (5, 17, 20, 20, 20, 22) & 
6 &  (10, 13, 24) & 
7 &  (11, 12, 24) & 
8 &  (11, 17, 20, 20, 24) \\
9 &  (11, 20, 21) & 
10 &  (13, 20, 22) & 
11 &  (23) \\
\end{array}\right.$
\end{footnotesize}
\end{center}
\end{tab}

\newpage
\subsection{Groups ${}^3D_4(q)$}\label{T3D4}

\begin{tab}[\ref{T3D4}.1-A] \tabA{${}^3D_4(q)$ with $q$ even}

\begin{center}
\begin{footnotesize}
 $\left.
\begin{array}{l|l||l|l}
h & \deg \chi_h & h & \deg \chi_h\\\hline 
1 &  1 &
2 &  q(q^4-q^2+1) \\
3 &  q^3(q+1)^2(q^4-q^2+1)/2 & 
4 &  q^3(q+1)^2(q^2-q+1)^2/2\\
5 &  q^3(q-1)^2(q^2+q+1)^2/2 & 
6 &  q^3(q-1)^2(q^4-q^2+1)/2 \\
7 &  q^7(q^4-q^2+1) &
8 &  q^{12} \\
9 &  (q+1)(q^2+q+1)(q^2-q+1)(q^4-q^2+1) & 
10 &  q^3(q+1)(q^2-q+1)(q^2+q+1)(q^4-q^2+1)\\
11 &  (q+1)(q^4-q^2+1)(q^2-q+1)^2 & 
12 &  q(q^4-q^2+1)(q+1)^2(q^2-q+1)^2\\
13 &  q^3(q+1)(q^4-q^2+1)(q^2-q+1)^2 & 
14 &  (q+1)(q^2+q+1)(q^4-q^2+1)(q^2-q+1)^2\\
15 &  q(q+1)(q^2+q+1)(q^4-q^2+1)(q^2-q+1)^2 & 
16 &  (q^2+q+1)(q^4-q^2+1)(q+1)^2(q^2-q+1)^2\\
17 &  (q-1)(q^2+q+1)(q^2-q+1)(q^4-q^2+1) & 
18 &  q^3(q-1)(q^2-q+1)(q^2+q+1)(q^4-q^2+1)\\
19 &  (q-1)(q+1)(q^2+q+1)(q^2-q+1)^2(q^4-q^2+1) &
20 &  (q-1)(q^4-q^2+1)(q^2+q+1)^2\\
21 &  q(q^4-q^2+1)(q-1)^2(q^2+q+1)^2 & 
22 &  q^3(q-1)(q^4-q^2+1)(q^2+q+1)^2\\
23 &  (q-1)(q^2-q+1)(q^4-q^2+1)(q^2+q+1)^2 & 
24 &  q(q-1)(q^2-q+1)(q^4-q^2+1)(q^2+q+1)^2\\
25 &  (q-1)(q+1)(q^2+q+1)^2(q^2-q+1)(q^4-q^2+1) & 
26 &  (q-1)^2(q+1)^2(q^2-q+1)^2(q^4-q^2+1) \\
27 &  (q-1)^2(q+1)^2(q^2+q+1)^2(q^4-q^2+1) &
28 &  (q-1)^2(q+1)^2(q^2+q+1)^2(q^2-q+1)^2\\
29 &  (q^2-q+1)(q^4-q^2+1)(q-1)^2(q^2+q+1)^2 \\
\end{array}
\right.$
\bigskip

$\left.\begin{array}{l|l|l}
h & I_h & {\rm exceptions} \\ \hline
1,\ldots,8 & - & \\
9,10 & k\in \ZZ_{q-1} & (q-1)\mid k\\
11,12,13 & k\in (q^2+q+1) & (q^2+q+1)\mid k\\
14,15 & k\in \ZZ_{q^3-1} & (q-1) \mid k\\
16 & (k,l)\in \ZZ_{q^3-1}\times \ZZ_{q-1} & {\rm many} \\
17,18 & k\in \ZZ_{ q+1} & (q+1) \mid k \\
19 & k\in \ZZ_{(q+1)(q^3-1)} & (q+1),(q^3-1) \mid k\\
20,21,22 & k\in \ZZ_{q^2-q+1} & (q^2-q+1)\mid k\\
23,24 & k\in \ZZ_{q^3+1} & (q+1)\mid k\\
25 & k\in \ZZ_{(q-1)(q^3+1)} & (q-1)(q^3+1)\mid k\\
26 & (k,l)\in \ZZ_{q^2+q+1}^2 & {\rm many} \\
27 & (k,l)\in \ZZ_{q^2-q+1}^2 & {\rm many} \\
28 &  k\in \ZZ_{q^4-q^2+1} & (q^4-q^2+1)\mid k\\
29 & (k,l)\in \ZZ_{q^3+1}\times\ZZ_{q+1} & {\rm many}\\
\end{array}\right.$
\end{footnotesize}
\end{center}
\end{tab}

\begin{tab}[\ref{T3D4}.1-B] \tabB{${}^3D_4(q)$ with $q$ even}

\begin{center}\begin{footnotesize}
$\left.\begin{array}{c|c||c|c||c|c||c|c||c|c}
h & \mathbf{v} & h & \mathbf{v}& h & \mathbf{v}& h & \mathbf{v}& h & \mathbf{v} \\ \hline 
10 &  (9, 25) & 
13 &  (11, 19) & 
14 &  (11, 12) & 
15 &  (12, 13) & 
15 &  (14, 19) \\ 
16 &  (9, 10) & 
16 &  (14, 15) & 
18 &  (17, 29) & 
19 &  (17, 18) & 
22 &  (21, 24) \\ 
23 &  (20, 21) & 
24 &  (23, 29) & 
25 &  (20, 22) & 
25 &  (23, 24) & 
9 &  (1, 2, 3, 4) \\ 
27 &  (21, 21, 21, 29) & 
8 &  (1, 17, 20, 21, 29) & 
8 &  (5, 6, 7, 17, 29) \\ 
\end{array}\right.$
\end{footnotesize}
\end{center}
\end{tab}

\begin{tab}[\ref{T3D4}.1-C] \tabC{${}^3D_4(q)$ with $q$ even}

\begin{center}
\begin{footnotesize}
$\left.\begin{array}{c|c||c|c||c|c||c|c}
j & \mathbf{v} & j & \mathbf{v} & j & \mathbf{v}  & j & \mathbf{v}\\ \hline 
1 &  (1, 17, 20, 21, 29) & 
2 &  (1, 17, 23, 29) & 
3 &  (1, 17, 24) &
4 &  (1, 18, 20, 21) \\ 
5 &  (1, 18, 23) & 
6 &  (2, 5, 6, 20, 21, 29) &
7 &  (2, 5, 6, 23, 29) & 
8 &  (2, 5, 6, 24) \\ 
9 &  (5, 6, 7, 17, 29) & 
10 &  (5, 6, 7, 18) & 
11 &  (8) \\
\end{array}\right.$
\end{footnotesize}
\end{center}
\end{tab}
\newpage
\begin{tab}[\ref{T3D4}.2-A] \tabA{${}^3D_4(q)$ with $q$ odd}

\begin{center}
\begin{footnotesize}
 $\left.
\begin{array}{l|l||l|l}
h & \deg \chi_h & h & \deg \chi_h\\\hline 
1 &  1 & 
2 &  q(q^4-q^2+1) \\
3 &  q^3(q+1)^2(q^4-q^2+1)/2 & 
4 &  q^3(q+1)^2(q^2-q+1)^2/2\\
5 &  q^3(q-1)^2(q^2+q+1)^2/2 & 
6 &  q^3(q-1)^2(q^4-q^2+1)/2\\
7 &  q^7(q^4-q^2+1) & 
8 &  q^{12}\\
9  & (q^2+q+1)(q^2-q+1)(q^4-q^2+1) &
10  & q(q^2+q+1)(q^2-q+1)(q^4-q^2+1)\\
11 &  q^3(q^2+q+1)(q^2-q+1)(q^4-q^2+1) & 
12 &  q^4(q^2+q+1)(q^2-q+1)(q^4-q^2+1)\\
13 &  (q+1)(q^2+q+1)(q^2-q+1)(q^4-q^2+1) & 
14 &  q^3(q+1)(q^2-q+1)(q^2+q+1)(q^4-q^2+1)\\
15 &  (q+1)(q^4-q^2+1)(q^2-q+1)^2 &
16 &  q(q^4-q^2+1)(q+1)^2(q^2-q+1)^2 \\
17 &  q^3(q+1)(q^4-q^2+1)(q^2-q+1)^2 & 
18 &  (q+1)(q^2+q+1)(q^4-q^2+1)(q^2-q+1)^2\\
19 &  q(q+1)(q^2+q+1)(q^4-q^2+1)(q^2-q+1)^2 & 
20 &  (q^2+q+1)(q^4-q^2+1)(q+1)^2(q^2-q+1)^2\\
21 &  (q-1)(q^2+q+1)(q^2-q+1)(q^4-q^2+1) & 
22 &  q^3(q-1)(q^2-q+1)(q^2+q+1)(q^4-q^2+1)\\
23 &  (q-1)(q+1)(q^2+q+1)(q^2-q+1)^2(q^4-q^2+1) & 
24 &  (q-1)(q^4-q^2+1)(q^2+q+1)^2\\
25 &  q(q^4-q^2+1)(q-1)^2(q^2+q+1)^2 & 
26 &  q^3(q-1)(q^4-q^2+1)(q^2+q+1)^2\\
27 &  (q-1)(q^2-q+1)(q^4-q^2+1)(q^2+q+1)^2 & 
28 &  q(q-1)(q^2-q+1)(q^4-q^2+1)(q^2+q+1)^2\\
29 &  (q-1)(q+1)(q^2+q+1)^2(q^2-q+1)(q^4-q^2+1) & 
30 &  (q-1)^2(q+1)^2(q^2-q+1)^2(q^4-q^2+1)\\
31 &  (q-1)^2(q+1)^2(q^2+q+1)^2(q^4-q^2+1) &
32 &  (q-1)^2(q+1)^2(q^2+q+1)^2(q^2-q+1)^2 \\
33 &  (q^2-q+1)(q^4-q^2+1)(q-1)^2(q^2+q+1)^2 \\
\end{array}
\right.$

\bigskip
$\left.\begin{array}{l|l|l}
h & I_h & {\rm exceptions} \\ \hline
1,\ldots,12 & - & \\
13,14 & k\in \ZZ_{q-1} & (q-1)/2\mid k\\
15,16,17 & k\in \ZZ_{q^2+q+1} & (q^2+q+1)\mid k\\
18,19 & k\in \ZZ_{q^3-1} & (q-1),(q^3-1)/2 \mid k\\
20 & (k,l)\in \ZZ_{q^3-1}\times \ZZ_{q-1}   & {\rm many}\\
21,22 & k\in \ZZ_{q+1} & (q+1)/2\mid k \\
23 &  k\in \ZZ_{(q+1)(q^3-1)} &  (q+1),(q^3-1),(q+1)(q^3-1)/2 \mid k\\
24,25,26 &  k\in \ZZ_{q^2-q+1} & (q^2-q+1)\mid k \\
27,28  & k\in \ZZ_{q^3+1} &  (q+1),(q^3+1)/2 \mid k\\
29 & k\in \ZZ_{(q-1)(q^3+1)} & (q-1),(q^3+1), (q-1)(q^3+1)/2\mid k\\
30 & (k,l)\in \ZZ_{q^2+q+1} ^2 & {\rm many}\\
31 & (k,l)\in \ZZ_{q^2-q+1}^2 & {\rm many} \\
32 & k\in \ZZ_{q^4-q^2+1} & (q^4-q^2+1)\mid k\\
33 & (k,l)\in \ZZ_{q^3+1}\times \ZZ_{q+1} & {\rm many} \\
\end{array}\right.$
\end{footnotesize}
\end{center}
\end{tab}

\begin{tab}[\ref{T3D4}.2-B] \tabB{${}^3D_4(q)$ with $q$ odd}

\begin{center}
\begin{footnotesize}
$\left.\begin{array}{c|c||c|c||c|c||c|c}
h & \mathbf{v} & h & \mathbf{v}& h & \mathbf{v}& h & \mathbf{v}\\ \hline 
10 &  (9, 21) & 
11 &  (9, 27) & 
12 &  (10, 28) & 
12 &  (11, 22) \\ 
13 &  (9, 10) &
14 &  (11, 12) & 
14 &  (13, 29) & 
17 &  (15, 23) \\ 
18 &  (9, 11) & 
18 &  (15, 16) & 
19 &  (10, 12) & 
19 &  (16, 17) \\ 
19 &  (18, 23) & 
20 &  (13, 14) & 
20 &  (18, 19) & 
22 &  (21, 33) \\ 
23 &  (21, 22) & 
26 &  (25, 28) & 
27 &  (24, 25) & 
28 &  (27, 33) \\ 
29 &  (24, 26) & 
29 &  (27, 28) & 
31 &  (25, 25, 25, 33) & 
8 &  (1, 21, 24, 25, 33) \\
8 &  (5, 6, 7, 21, 33) &
32 &  (5, 5, 5, 5, 25, 33) \\
\end{array}\right.$
\end{footnotesize}
\end{center}
\end{tab}

\begin{tab}[\ref{T3D4}.2-C] \tabC{${}^3D_4(q)$ with $q$ odd}

\smallskip
\begin{center}
\begin{footnotesize}
 $\left.\begin{array}{c|c||c|c||c|c||c|c}
j & \mathbf{v} & j & \mathbf{v} & j & \mathbf{v}  & j & \mathbf{v}\\ \hline 
1 &  (1, 21, 24, 25, 33) & 
2 &  (5, 6, 7, 21, 33) & 
3 &  (2, 5, 6, 24, 25, 33) & 
4 &  (1, 21, 27, 33) \\
5 &  (1, 21, 28) & 
6 &  (1, 22, 24, 25) &
7 &  (1, 22, 27) &
8 &  (2, 5, 6, 27, 33) \\
9 &  (2, 5, 6, 28) &
10 &  (5, 6, 7, 22)  &
11 &  (8) \\
\end{array}\right.$
\end{footnotesize}
\end{center}
\end{tab}

\subsection{Groups $G_2(q)$}\label{TG2}

\begin{tab}[\ref{TG2}.1-A] \tabA{$G_2(q)$ with $q$ even and $q \equiv 1\pmod 3$}

\begin{center}\begin{footnotesize}
 $\left.
\begin{array}{l|l||l|l}
h & \deg \chi_h & h & \deg \chi_h\\\hline 
1 &  1  &
2 &  q(q+1)^2(q^2+q+1)/6 \\
3 &  q(q+1)^2(q^2-q+1)/2 &
4,5 &  q(q^2+q+1)(q^2-q+1)/3\\
6 &  q(q-1)^2(q^2-q+1)/6 &
7 &  q(q-1)^2(q^2+q+1)/2 \\
8,9 &  q(q-1)^2(q+1)^2/3&
10 &  q^6\\
11 &  (q+1)(q^2-q+1) &
12 &  q(q+1)^2(q^2-q+1)\\
13 &  q^3(q+1)(q^2-q+1) & 
14,16 &  (q^2+q+1)(q+1)(q^2-q+1)\\
15,17 &  q(q^2+q+1)(q+1)(q^2-q+1) &
18,20 &  (q^2-q+1)(q-1)(q^2+q+1)\\
19,21 &  q(q^2-q+1)(q-1)(q^2+q+1) &
22 &  (q+1)^2(q^2+q+1)(q^2-q+1)\\
23 &  (q-1)^2(q^2-q+1)(q^2+q+1) &
24,25 &  (q-1)(q+1)(q^2+q+1)(q^2-q+1)\\
26 &  (q-1)^2(q+1)^2(q^2-q+1) &
27 &  (q+1)^2(q-1)^2(q^2+q+1)\\
\end{array}
\right.$

\med
$\left.\begin{array}{l|l|l}
h & I_h & {\rm exceptions} \\ \hline
1,\ldots,13 & - & \\
14,15 & k\in\ZZ_{q-1} & (q-1)/3  \mid k \\
16,17 & k\in \ZZ_{q-1} & (q-1)\mid k\\
18,\ldots,21 & k\in \ZZ_{q+1} & (q+1)\mid k \\
22 & (k,l)\in \ZZ_{q-1} ^2 & (q-1)\mid k,l,(k\pm l), (2k-l), (k-2l) \\
23 & (k,l)\in \ZZ_{q+1} ^2 & (q+1)\mid k,l,(k\pm l), (2k-l), (k-2l) \\
24,25 & k\in \ZZ_{q^2-1} & (q\pm 1)\mid k\\
26 &k\in \ZZ_{q^2+q+1} &  (q^2+q+1)/3\mid k\\
27 & k\in \ZZ_{q^2-q+1} & (q^2-q+1)\mid k\\
\end{array}\right.$
\end{footnotesize}
\end{center}
\end{tab}

\begin{tab}[\ref{TG2}.1-B] \tabB{$G_2(q)$ with $q$ even and $q \equiv 1\pmod 3$}

\begin{center}
\begin{footnotesize}
$\left.\begin{array}{c|c||c|c||c|c||c|c||c|c}
h & \mathbf{v} & h & \mathbf{v}& h & \mathbf{v}& h & \mathbf{v}& h & \mathbf{v} \\ \hline 
8 &  (9) & 
13 &  (11, 24) & 
14 &  (11, 12) & 
15 &  (12, 13) & 
15 &  (14, 24) \\ 
17 &  (16, 25) & 
19 &  (18, 23) & 
21 &  (20, 23) & 
22 &  (14, 15) & 
22 &  (16, 17) \\ 
24 &  (18, 19) & 
25 &  (20, 21) & 
12 &  (2, 3, 9) & 
10 &  (1, 18, 20, 23) & 
16 &  (1, 2, 3, 5) \\ 
\end{array}\right.$
\end{footnotesize}
\end{center}
\end{tab}
 
\begin{tab}[\ref{TG2}.1-C] \tabC{$G_2(q)$ with $q$ even and $q \equiv 1\pmod 3$}

\begin{center}
\begin{footnotesize}
$\left.\begin{array}{c|c||c|c||c|c||c|c}
j & \mathbf{v} & j & \mathbf{v} & j & \mathbf{v}  & j & \mathbf{v}\\ \hline 
1 &  (1, 18, 20, 23) & 
2 &  (1, 18, 21) & 
3 &  (1, 19, 20) & 
4 &  (4, 6, 7, 18, 23) \\ 
5 &  (4, 6, 7, 19) & 
6 &  (5, 6, 7, 20, 23) & 
7 &  (5, 6, 7, 21) & 
8 &  (10) \\ 
\end{array}\right.$
\end{footnotesize}
\end{center}
\end{tab}

\begin{tab}[\ref{TG2}.2-A] \tabA{$G_2(q)$ with $q$ even and $q \equiv -1\pmod 3$}

\begin{center}
\begin{footnotesize}
 $\left.
\begin{array}{l|l||l|l}
h & \deg \chi_h & h & \deg \chi_h\\\hline 
1 &  1 &
2 &  q(q+1)^2(q^2+q+1)/6 \\
3 &  q(q+1)^2(q^2-q+1)/2 &
4,5 &  q(q^2+q+1)(q^2-q+1)/3\\
6 & q(q-1)^2(q^2-q+1)/6 &
7 &  q(q-1)^2(q^2+q+1)/2\\
8,9 &  q(q-1)^2(q+1)^2/3 &
10 &   q^6 \\
11 &  (q-1)(q^2+q+1) &
12 &  q(q-1)^2(q^2+q+1) \\
13 &  q^3(q-1)(q^2+q+1) & 
14,16 &  (q^2+q+1)(q+1)(q^2-q+1)\\
15,17 &  q(q^2+q+1)(q+1)(q^2-q+1) & 
18,20 &  (q^2-q+1)(q-1)(q^2+q+1) \\
19,21 &  q(q^2-q+1)(q-1)(q^2+q+1) & 
22 &  (q+1)^2(q^2+q+1)(q^2-q+1) \\
23 &  (q-1)^2(q^2-q+1)(q^2+q+1) & 
24,25 &  (q-1)(q+1)(q^2+q+1)(q^2-q+1) \\
26 &  (q-1)^2(q+1)^2(q^2-q+1) & 
27 &   (q+1)^2(q-1)^2(q^2+q+1)\\
\end{array}
\right.$

\med
$\left.\begin{array}{l|l|l}
h & I_h & {\rm exceptions} \\ \hline
1,\ldots,13  & - & \\
14,\ldots,17 & k\in \ZZ_{q-1} & (q-1)\mid k \\
18,19 & k\in \ZZ_{q+1} & (q+1)\mid k\\
20,21 & k \in \ZZ_{q+1} & (q+1)/3\mid k \\
22 & (k,l)\in \ZZ_{q-1}^2 & (q-1)\mid k,l, (k\pm l), (2k-l), (k-2l)\\
23 & (k,l)\in \ZZ_{q+1} ^2  & (q+1)\mid k,l,(k\pm l), (2k-l), (k-2l) \\
24,25 & k\in \ZZ_{q^2-1} &  (q\pm 1) \mid k \\
26 &  k\in \ZZ_{q^2+q+1} & (q^2+q+1)\mid k\\
27 &  \ZZ_{q^2-q+1} & (q^2-q+1)/3\mid k
\end{array}\right.$
\end{footnotesize}
\end{center}
\end{tab}

\begin{tab}[\ref{TG2}.2-B] \tabB{$G_2(q)$ with $q$ even and $q \equiv -1\pmod 3$}

\begin{center}
\begin{footnotesize}
$\left.\begin{array}{c|c||c|c||c|c||c|c||c|c}
h & \mathbf{v} & h & \mathbf{v}& h & \mathbf{v}& h & \mathbf{v}& h & \mathbf{v} \\ \hline 
8 &  (9) & 
13 &  (12, 21) & 
15 &  (14, 24) & 
17 &  (16, 25) & 
19 &  (18, 23) \\ 
20 &  (11, 12) & 
21 &  (20, 23) & 
22 &  (14, 15) & 
22 &  (16, 17) & 
24 &  (18, 19) \\ 
25 &  (11, 13) & 
25 &  (20, 21) & 
4 &  (1, 9, 11) & 
12 &  (6, 7, 9) & 
16 &  (1, 2, 3, 5)\\ 
\end{array}\right.$

\medskip
$\left.\begin{array}{c|c||c|c||c|c}
h & \mathbf{v} & h & \mathbf{v}& h & \mathbf{v} \\ \hline 
14 &  (1, 1, 2, 3, 9, 11) & 
10 &  (1, 6, 7, 9, 11, 18, 23) & 
27 &  (6, 6, 6, 7, 7, 7, 9, 9, 9, 23)\\
\end{array}\right.$
\end{footnotesize}
\end{center}
\end{tab}

\begin{tab}[\ref{TG2}.2-C] \tabC{$G_2(q)$ with $q$ even and $q \equiv -1\pmod 3$}

\begin{center}
\begin{footnotesize}
$\left. \begin{array}{c|c||c|c||c|c}
j & \mathbf{v} & j & \mathbf{v} & j & \mathbf{v}  \\ \hline 
1 &  (1, 6, 7, 8, 11, 18, 23) & 
2 &  (1, 6, 7, 8, 11, 19) & 
3 &  (1, 6, 7, 9, 11, 18, 23) \\ 
4 &  (1, 6, 7, 9, 11, 19) &
5 &  (1, 11, 12, 18, 23) & 
6 &  (1, 11, 12, 19) \\ 
7 &  (1, 18, 20, 23) & 
8 &  (1, 18, 21) &
9 &  (1, 19, 20) \\ 
10 &  (4, 6, 7, 18, 23) & 
11 &  (4, 6, 7, 19) & 
12 &  (5, 6, 6, 7, 7, 8, 11, 23) \\
13 &  (5, 6, 6, 7, 7, 9, 11, 23) &
14 &  (5, 6, 7, 11, 12, 23) & 
15 &  (5, 6, 7, 20, 23)\\ 
16 &  (5, 6, 7, 21) &
17 &  (10) \\ 
\end{array}\right.$
\end{footnotesize}
\end{center}
\end{tab}

\begin{tab}[\ref{TG2}.3-A] \tabA{$G_2(q)$ with $q$ odd and $q \equiv 0\pmod 3$}

\smallskip
\begin{center}
\begin{footnotesize}
$\left.
\begin{array}{l|l||l|l}
h & \deg \chi_h & h & \deg \chi_h\\\hline 
1  & 1  &
2  & q(q+1)^2(q^2+q+1)/6 \\
3  & q(q+1)^2(q^2-q+1)/2 & 
4,5  & q(q^2+q+1)(q^2-q+1)/3 \\
6  & q(q-1)^2(q^2-q+1)/6 & 
7  & q(q-1)^2(q^2+q+1)/2 \\
8,9  & q(q-1)^2(q+1)^2/3 & 
10 &  q^6  \\
11 &  (q^2+q+1)(q^2-q+1)  & 
12,13 &  q(q^2+q+1)(q^2-q+1) \\
14 &  q^2(q^2+q+1)(q^2-q+1) & 
15,17 &  (q^2+q+1)(q+1)(q^2-q+1)  \\
16,18 &  q(q^2+q+1)(q+1)(q^2-q+1) & 
19,21 &  (q^2-q+1)(q-1)(q^2+q+1) \\
20,22 &  q(q^2-q+1)(q-1)(q^2+q+1) & 
23 &  (q+1)^2(q^2+q+1)(q^2-q+1) \\
24 &  (q-1)^2(q^2-q+1)(q^2+q+1) & 
25,26 &  (q-1)(q+1)(q^2+q+1)(q^2-q+1) \\
27 &  (q-1)^2(q+1)^2(q^2-q+1) & 
28 &  (q+1)^2(q-1)^2(q^2+q+1)\\
\end{array}
\right.$

\smallskip
{}

\med
$\left.\begin{array}{l|l|l}
h & I_h & {\rm exceptions} \\ \hline
1,\ldots 14 & - & \\
15,\ldots,18 & k\in \ZZ_{q-1} & (q-1)/2\mid k\\
19,\ldots,22 & k\in \ZZ_{q+1} & (q+1)/2\mid k\\
23 & (k,l)\in \ZZ_{q-1}^2 & (q-1)\mid k,l,(k\pm l), (2k-l),(k-2l)\\
24 & (k,l)\in \ZZ_{q+1} ^2 & (q+1)\mid k,l,(k\pm l), (2k-l), (k-2l)  \\
25,26 & k\in \ZZ_{q^2-1} & (q\pm 1)\mid k\\
27 & k\in \ZZ_{q^2+q+1} & (q^2+q+1)\mid k  \\
28 & k\in \ZZ_{q^2-q+1} & (q^2-q+1)\mid k \\
\end{array}\right.$
\end{footnotesize}
\end{center}
\end{tab}

\newpage
\begin{tab}[\ref{TG2}.3-B] \tabB{$G_2(q)$ with $q$ odd and $q \equiv 0\pmod 3$}

\begin{center}
\begin{footnotesize}
$\left. \begin{array}{c|c||c|c||c|c||c|c||c|c}
h & \mathbf{v} & h & \mathbf{v}& h & \mathbf{v}& h & \mathbf{v}& h & \mathbf{v} \\ \hline 
12 &  (11, 19) & 
13 &  (11, 21) & 
14 &  (12, 22) & 
14 &  (13, 20) & 
15 &  (11, 13) \\ 
16 &  (12, 14) & 
16 &  (15, 25) & 
17 &  (11, 12) & 
18 &  (13, 14) & 
18 &  (17, 26) \\ 
20 &  (19, 24) & 
22 &  (21, 24) & 
23 &  (15, 16) & 
23 &  (17, 18) & 
25 &  (19, 20) \\ 
26 &  (21, 22) & 
10 &  (1, 19, 21, 24) \\
\end{array}\right.$
\end{footnotesize}\end{center}
\end{tab}

\begin{tab}[\ref{TG2}.3-C] \tabC{$G_2(q)$ with $q$ odd and $q \equiv 0\pmod 3$}

\begin{center}
\begin{footnotesize}
$\left. \begin{array}{c|c||c|c||c|c||c|c}
j & \mathbf{v} & j & \mathbf{v} & j & \mathbf{v}  & j & \mathbf{v}\\ \hline 
1 &  (1, 19, 21, 24) & 
2 &  (1, 19, 22) & 
3 &  (1, 20, 21) & 
4 &  (4, 6, 7, 19, 24) \\
5 &  (4, 6, 7, 20) & 
6 &  (5, 6, 7, 21, 24) & 
7 &  (5, 6, 7, 22) & 
8 &  (10) \\
\end{array}\right.$
\end{footnotesize}
\end{center}
\end{tab}

\begin{tab}[\ref{TG2}.4-A] \tabA{$G_2(q)$ with $q$ odd and $q \equiv 1\pmod 3$}

\smallskip
\begin{center}
\begin{footnotesize}
 $\left.
\begin{array}{l|l||l|l}
h & \deg \chi_h & h & \deg \chi_h\\\hline 
1 &  1 & 
2 &  q(q+1)^2(q^2+q+1)/6\\
3 &  q(q+1)^2(q^2-q+1)/2 & 
4,5 &  q(q^2+q+1)(q^2-q+1)/3\\
6 &  q(q-1)^2(q^2-q+1)/6 & 
7 &  q(q-1)^2(q^2+q+1)/2\\
8,9 &  q(q-1)^2(q+1)^2/3 & 
10 &  q^6\\
11 &  (q^2+q+1)(q^2-q+1) & 
12,13 &  q(q^2+q+1)(q^2-q+1)\\
14 &  q^2(q^2+q+1)(q^2-q+1) & 
15 &  (q+1)(q^2-q+1)\\
16 &  q(q+1)^2(q^2-q+1) & 
17 &  q^3(q+1)(q^2-q+1)\\
18,20 &  (q^2+q+1)(q+1)(q^2-q+1) & 
19,21 &  q(q^2+q+1)(q+1)(q^2-q+1)\\
22,24 &   (q^2-q+1)(q-1)(q^2+q+1) & 
23,25 &  q(q^2-q+1)(q-1)(q^2+q+1)\\
26 &  (q+1)^2(q^2+q+1)(q^2-q+1) & 
27 &  (q-1)^2(q^2-q+1)(q^2+q+1)\\
28,29 &  (q-1)(q+1)(q^2+q+1)(q^2-q+1) &
30 &  (q-1)^2(q+1)^2(q^2-q+1)\\
31 &  (q+1)^2(q-1)^2(q^2+q+1)\\
\end{array}
\right.$

\med
$\left. \begin{array}{l|l|l}
h & I_h & {\rm exceptions} \\ \hline
1,\ldots,17 & - & \\
18,19 & k\in \ZZ_{q-1} &  (q-1)/3, (q-1)/2\mid k \\
20,21 & k\in \ZZ_{q-1} & (q-1)/2\mid k  \\
22,\ldots,25 & k \in \ZZ_{q+1} & (q+1)/2\mid k\\
26 & (k,l)\in \ZZ_{q-1}^2 &  (q-1) \mid k,l,(k\pm l),(2k-l),(k-2l)\\
27 & (k,l)\in \ZZ_{q+1}^2 & (q+1)\mid k,l,(k\pm l),(2k-l), (k-2l)\\
28,29 & k \in \ZZ_{q^2-1} & (q\pm 1)\mid k\\
30 & k\in  \ZZ_{q^2+q+1} & (q^2+q+1)/3\mid k\\
31 & k\in \ZZ_{q^2-q+1} &  (q^2-q+1)\mid k\\
\end{array}\right.$
\end{footnotesize}
\end{center}
\end{tab}

\begin{tab}[\ref{TG2}.4-B] \tabB{$G_2(q)$ with $q$ odd and $q \equiv 1\pmod 3$}

\smallskip
\begin{center}
\begin{footnotesize}
$\left.\begin{array}{c|c||c|c||c|c||c|c||c|c}
h & \mathbf{v} & h & \mathbf{v}& h & \mathbf{v}& h & \mathbf{v}& h & \mathbf{v} \\ \hline 
8 &  (9) & 
12 &  (11, 22) & 
13 &  (11, 24) & 
14 &  (12, 25) &
14 &  (13, 23) \\ 
17 &  (15, 28) & 
18 &  (11, 13) & 
18 &  (15, 16) & 
19 &  (12, 14) & 
19 &  (16, 17) \\ 
19 &  (18, 28) & 
20 &  (11, 12) & 
21 &  (13, 14) & 
21 &  (20, 29) & 
23 &  (22, 27) \\ 
25 &  (24, 27) & 
26 &  (18, 19) & 
26 &  (20, 21) & 
28 &  (22, 23) & 
29 &  (24, 25) \\ 
16 &  (2, 3, 9) & 
10 &  (1, 22, 24, 27) \\
\end{array}\right.$
\end{footnotesize}

\end{center}
\end{tab}

\begin{tab}[\ref{TG2}.4-C] \tabC{~$~G_2(q)~$ with $~q~$ odd  and $q \equiv 1\pmod
3$}

\begin{center}
\begin{footnotesize}
$\left.\begin{array}{c|c||c|c||c|c||c|c}
j & \mathbf{v} & j & \mathbf{v} & j & \mathbf{v}  & j & \mathbf{v}\\ \hline 
1 &  (1, 22, 24, 27) & 
2 &  (1, 22, 25) & 
3 &  (1, 23, 24) & 
4 &  (4, 6, 7, 22, 27) \\
5 &  (4, 6, 7, 23) & 
6 &  (5, 6, 7, 24, 27) & 
7 &  (5, 6, 7, 25) & 
8 &  (10) \\
\end{array}\right.$
\end{footnotesize}
\end{center}
\end{tab}

\begin{tab}[\ref{TG2}.5-A] \tabA{$G_2(q)$ with $q$ odd and $q \equiv -1\pmod 3$}

\smallskip
\begin{center}
\begin{footnotesize}
 $\left.
\begin{array}{l|l||l|l}
h & \deg \chi_h & h & \deg \chi_h\\\hline 
1 &   1 &
2 &  q(q+1)^2(q^2+q+1)/6\\
3 &  q(q+1)^2(q^2-q+1)/2 & 
4,5 &  q(q^2+q+1)(q^2-q+1)/3\\
6 &  q(q-1)^2(q^2-q+1)/6 & 
7 &  q(q-1)^2(q^2+q+1)/2\\
8,9 &  q(q-1)^2(q+1)^2/3 & 
10 &  q^6\\
11 &  (q^2+q+1)(q^2-q+1) & 
12,13 &  q(q^2+q+1)(q^2-q+1)\\
14 &  q^2(q^2+q+1)(q^2-q+1) & 
15 &  (q-1)(q^2+q+1)\\
16 &  q(q-1)^2(q^2+q+1) & 
17 &  q^3(q-1)(q^2+q+1)\\
18,20 &  (q^2+q+1)(q+1)(q^2-q+1) & 
19,21 &  q(q^2+q+1)(q+1)(q^2-q+1)\\
22,24 &  (q^2-q+1)(q-1)(q^2+q+1) & 
23,25 &  q(q^2-q+1)(q-1)(q^2+q+1)\\
26 &  (q+1)^2(q^2+q+1)(q^2-q+1) &
27 &  (q-1)^2(q^2-q+1)(q^2+q+1)\\
28,29 &  (q-1)(q+1)(q^2+q+1)(q^2-q+1) & 
30 &  (q-1)^2(q+1)^2(q^2-q+1)\\
31 &  (q+1)^2(q-1)^2(q^2+q+1)\\
\end{array}
\right.$

\bigskip
$\left.\begin{array}{l|l|l}
h & I_h & {\rm exceptions} \\ \hline
1,\ldots,17 & -  &\\
18,\ldots,21 & k\in \ZZ_{q-1} &  (q-1)/2\mid k\\
22,23 & k\in \ZZ_{q+1} & (q+1)/2\mid k\\
24,25 & k\in \ZZ_{q+1} & (q+1)/2, (q+1)/3 \mid k\\
26 & (k,l)\in \ZZ_{q-1}^2 & (q-1)\mid k,l,(k\pm l), (2k-l),(k-2l)\\
27 &  (k,l)\in \ZZ_{q+1}^2 & (q+1)\mid k,l,(k\pm l), (2k-l), (k-2l)\\
28,29 & k\in \ZZ_{q^2-1}  & (q\pm 1)\mid k\\
30 &  k\in \ZZ_{q^2+q+1} &  (q^2+q+1)\mid k\\
31 & k\in \ZZ_{q^2-q+1} & (q^2-q+1)/3\mid k\\
\end{array}\right.$
\end{footnotesize}
\end{center}
\end{tab}

\begin{tab}[\ref{TG2}.5-B] \tabB{$G_2(q)$ with $q$ odd and $q \equiv -1\pmod 3$}

\begin{center}
\begin{footnotesize}
$\left. \begin{array}{c|c||c|c||c|c||c|c||c|c}
h & \mathbf{v} & h & \mathbf{v}& h & \mathbf{v}& h & \mathbf{v}& h & \mathbf{v} \\ \hline 
8 &  (9) & 
12 &  (11, 22) & 
13 &  (11, 24) & 
14 &  (12, 25) & 
14 &  (13, 23) \\ 
17 &  (16, 25) & 
18 &  (11, 13) & 
19 &  (12, 14) & 
19 &  (18, 28) & 
20 &  (11, 12) \\ 
21 &  (13, 14) & 
21 &  (20, 29) & 
23 &  (22, 27) & 
24 &  (15, 16) & 
25 &  (24, 27) \\ 
26 &  (18, 19) & 
26 &  (20, 21) & 
28 &  (22, 23) & 
29 &  (15, 17) & 
29 &  (24, 25) \\
\end{array}\right.$

\medskip
$\left.\begin{array}{c|c||c|c||c|c||c|c}
h & \mathbf{v} & h & \mathbf{v}& h & \mathbf{v}& h & \mathbf{v}\\ \hline 
4 &  (1, 9, 15) & 
16 &  (6, 7, 9) & 
10 &  (1, 6, 7, 9, 15, 22, 27) & 
31 &  (6, 6, 6, 7, 7, 7, 9, 9, 9, 27) \\
\end{array}\right.$
\end{footnotesize}
\end{center}
\end{tab}
\newpage

\begin{tab}[\ref{TG2}.5-C] \tabC{$G_2(q)$ with $q$ odd and $q \equiv -1\pmod 3$}

\begin{center}
\begin{footnotesize}
$\left. \begin{array}{c|c||c|c||c|c}
j & \mathbf{v} & j & \mathbf{v} & j & \mathbf{v} \\ \hline 
1 &  (5, 6, 6, 7, 7, 9, 15, 27) & 
2 &  (1, 6, 7, 9, 15, 22, 27) & 
3 &  (1, 6, 7, 8, 15, 22, 27) \\
4 &  (1, 6, 7, 8, 15, 23) &
5 &  (1, 6, 7, 9, 15, 23) &
6 &  (1, 15, 16, 22, 27) \\
7 &  (1, 15, 16, 23) & 
8 &  (1, 22, 24, 27) &
9 &  (1, 22, 25) \\
10 &  (1, 23, 24) & 
11 &  (4, 6, 7, 22, 27) & 
12 &  (4, 6, 7, 23) \\
13 &  (5, 6, 6, 7, 7, 8, 15, 27) &
14 &  (5, 6, 7, 15, 16, 27) & 
15 &  (5, 6, 7, 24, 27) \\
16 &  (5, 6, 7, 25) &
17 &  (10) \\
\end{array}\right.$
\end{footnotesize}
\end{center}
\end{tab}

\subsection{Groups ${}^2F_4(q^2)$}\label{T2F4}

\begin{tab}[\ref{T2F4}-A] Some \tabA{${}^2F_4(q^2)$}

\smallskip
\begin{center}\begin{footnotesize}
$\left.
\begin{array}{l|l}
h & \deg \chi_h \\\hline 
2,3 &  \frac{q\sqrt{2}}{2} (q+1)(q^4-q^2+1)(q-1)(q^2+1)^2\\
4 &  q^2(q^4-q^2+1)(q^8-q^4+1)\\
10 &  \frac{q^4}{6}(q-1)^2(q+1)^2(q^2+1)^2(q^8-q^4+1)\\
11,12 &  \frac{q^4}{4}
(q^4-q^2+1) (q^4-\sqrt{2}q^3+q^2-\sqrt{2}q+1)(q^2+1)^2(q+1)^2(q-1)^2\\
13,14 & 
\frac{q^4}{4}(q^4-q^2+1)(q^4+\sqrt{2}q^3+q^2+\sqrt{2}q+1)(q^2+1)^2(q+1)^2(q-1)^2\\
17 &  \frac{q^4}{3}(q-1)^2(q+1)^2(q^4-q^2+1)(q^8-q^4+1)\\
19,20 &  \frac{q^{13}\sqrt{2}}{2}(q+1)(q^4-q^2+1)(q-1)(q^2+1)^2\\
21 & q^{24}\\
46 &  (q+1)(q^4-q^2+1)(q^2-\sqrt{2}q+1)(q-1) (q^8-q^4+1)(q^2+1)^2\\
47,48 & 
\frac{q\sqrt{2}}{2}(q^4-q^2+1)(q^2-\sqrt{2}q+1)(q^8-q^4+1)(q^2+1)^2(q+1)^2(q-1)^2\\
49 & q^4(q+1)(q^4-q^2+1)(q^2-\sqrt{2}q+1)(q-1)(q^8-q^4+1)(q^2+1)^2\\
51 &  (q^4-q^2+1)(q^8-q^4+1)(q^2+1)^2(q+1)^2(q^2-\sqrt{2}q+1)^2(q-1)^2\\
\end{array}
\right.$

\medskip
$\left.\begin{array}{l|l|l}
h & I_h & {\rm exceptions} \\ \hline
2,3,4,10,11,12,13,14,17,19,20,21 & -\\
46,\ldots,49 & k\in \ZZ_{q^2+q\sqrt{2}+1} & (q^2+q\sqrt{2}+1)\mid k\\
51 & (k,l)\in (\ZZ_{q^2+q\sqrt{2}+1})^2 &  q^2+q\sqrt{2}+1\mid k,l,(k\pm l), (l\pm q^2k)\\
\end{array}\right.$
\end{footnotesize}
\end{center}
\end{tab}

\begin{tab}[\ref{T2F4}-B] \tabB{${}^2F_4(q^2)$}

\begin{center}\begin{footnotesize}
$\left. \begin{array}{c|c||c|c||c|c||c|c}
h & \mathbf{v} & h & \mathbf{v}& h & \mathbf{v}& h & \mathbf{v} \\ \hline 
15 &  (16) & 
23 &  (2, 19) & 
24 &  (3, 20) & 
27 &  (26, 34) \\ 
28 &  (22, 25) &
28 &  (26, 27) & 
31 &  (30, 33) & 
32 &  (29, 30) \\ 
33 &  (32, 38) & 
34 &  (29, 31) & 
34 &  (32, 33) & 
35 &  (42, 45) \\ 
36 &  (46, 49) & 
39 &  (30, 30, 30, 38) & 
45 &  (37, 42, 43, 44) &
49 &  (46, 47, 48, 51) \\ 
22 &  (1, 4, 5, 6, 7) & 
25 &  (5, 6, 7, 18, 21) & 
26 &  (1, 5, 6, 7, 18) & 
37 &  (47, 47, 48, 48, 51) \\ 
50 &  (37, 43, 43, 44, 44) \\
\end{array}\right.$
\end{footnotesize}
\end{center}
\end{tab}

\begin{tab}[\ref{T2F4}-C] \tabC{${}^2F_4(q^2)$}

\begin{center}\begin{footnotesize}
$\left. \begin{array}{c|c||c|c}
j & \mathbf{v} & j & \mathbf{v}\\ \hline 
1 &  (2, 4, 10, 11, 13, 17, 20, 46, 47, 47, 48, 51) & 
2 &  (3, 4, 10, 12, 14, 17, 19, 46, 47, 48, 48, 51) \\
3 &  (2, 4, 10, 11, 13, 17, 20, 47, 49) &
4 &  (3, 4, 10, 12, 14, 17, 19, 48, 49) \\
5 & (21)\\
\end{array}\right.$
\end{footnotesize}
\end{center}
\end{tab}
\newpage
\subsection{Groups $GL(4,q)$}\label{TGL4}

\begin{tab}[\ref{TGL4}-A] \tabA{$GL(4,q)$}

\smallskip
\begin{center}\begin{footnotesize}
$\left.
\begin{array}{l|l||l|l}
h & \deg \chi_h & h & \deg \chi_h\\\hline 
1&   q^6 & 2&   q^3 (q^2+q+1)\\
3&   q^2 (q^2+1) & 4&   q (q^2+q+1)\\
5&   1 & 6&   (q+1) (q^2+1) q^3\\
7&   q (q^2+1) (q+1)^2 & 8&   (q^2+1) (q+1)\\
9&   q^2 (q^2+1) (q^2+q+1)& 10&   q (q^2+1) (q^2+q+1)\\
11&   q (q^2+1) (q^2+q+1)& 12&   (q^2+1) (q^2+q+1)\\
13&   q^2 (q^2+q+1) (q-1)^2& 14&   (q^2+q+1) (q-1)^2\\
15&   q (q+1) (q^2+q+1) (q^2+1) & 16&   (q+1) (q^2+1) (q^2+q+1)\\
17&   q (q-1) (q^2+q+1) ( q^2+1) & 18&   (q-1) (q^2+q+1) (q^2+1)\\
19&   (q^2+1) (q^2+q+1) (q+1)^2 & 20&   (q-1) (q+1) (q^2+q+1) (q^2+1)\\
21&   (q^2+1) (q^2+q+1) (q-1)^2 & 22&   (q^2+1) (q-1)^2 (q+1)^2\\
23&   (q+1) (q^2+q+1) (q-1)^3\\
\end{array}
\right.$

\medskip
$\left.\begin{array}{l|l|l}
h & I_h & {\rm exceptions} \\ \hline
1\ldots,5 & k_1\in \ZZ_{q-1} & \\
6,\ldots,12 & (k_1,k_2)\in \ZZ_{q-1}\times \ZZ_{q-1} & (q-1)\mid (k_1-k_2)\\
13,14 & k_1\in \ZZ_{q^2-1} & (q+1)\mid k_1 \\
15,16 &  (k_1,k_2,k_3)\in \ZZ_{q-1}^3 & (q-1)\mid (k_i-k_j), i\neq j\\
17,18 & (k_1,k_2)\in \ZZ_{q-1}\times \ZZ_{q^2-1} & (q^2-1)\mid (q+1)k_1-k_2\\
{} & {} & (q+1)\mid k_2 \\ 
19 & (k_1,k_2,k_3,k_4)\in \ZZ_{q-1}^4 &  (q-1)\mid (k_i-k_j), i\neq j\\
20 & (k_1,k_2,k_3)\in \ZZ_{q^2-1}\times \ZZ_{q-1}\times \ZZ_{q-1} & (q+1)\mid k_1\\
{} & {} & (q-1)\mid (k_2-k_3)\\
21 & (k_1,k_2)\in \ZZ_{q^2-1}\times \ZZ_{q^2-1} &  (q+1)\mid k_i\\
 {} & {} & (q^2-1) \mid qk_1-k_2\\
{} & {} &  (q^2-1)\mid (k_1-k_2)\\
22 & (k_1,k_2)\in \ZZ_{q^3-1}\times \ZZ_{q-1} & (q^2+q+1)\mid (q+1)k_1 \\
23 & k_1\in \ZZ_{q^4-1} & (q^3+q+q^2+1) \mid (q^2+q+1)k_1\\
{} & {} & (q^2+1)\mid k_1\\
\end{array}\right.$
\end{footnotesize}
\end{center}
\end{tab}

\begin{tab}[\ref{TGL4}-B] \tabB{$GL(4,q)$}

\begin{center}\begin{footnotesize}
$\left. \begin{array}{c|c||c|c||c|c||c|c||c|c}
h & \mathbf{v} & h & \mathbf{v}& h & \mathbf{v}& h & \mathbf{v}& h & \mathbf{v} \\ \hline 
10 &  (11) & 
6 &  (1, 2) & 
6 &  (8, 20) & 
7 &  (3, 11) & 
8 &  (4, 5) \\
9 &  (3, 6) & 
9 &  (11, 17) & 
9 &  (12, 20) & 
11 &  (2, 4) & 
11 &  (12, 18) \\ 
12 &  (3, 8) & 
13 &  (14, 23) & 
15 &  (6, 7) & 
15 &  (9, 11) & 
15 &  (16, 20) \\
16 &  (7, 8) & 
16 &  (11, 12) & 
17 &  (18, 21) & 
19 &  (15, 16) & 
20 &  (17, 18) \\ 
21 &  (13, 14) & 
2 &  (3, 5, 18) & 
2 &  (4, 14, 18) & 
1 &  (5, 14, 18, 23) \\
\end{array}\right.$
\end{footnotesize}
\end{center}
\end{tab}

\begin{tab}[\ref{TGL4}-C] \tabC{$GL(4,q)$}

\begin{center}\begin{footnotesize}
$\left.\begin{array}{c|c||c|c||c|c}
j & \mathbf{v} & j & \mathbf{v} & j & \mathbf{v} \\ \hline 
1 &  (5, 14, 18, 23) &
2 &  (5, 13, 18) & 
3 &  (1) \\
\end{array}\right.$
\end{footnotesize}
\end{center}
\end{tab}

\begin{tab}[\ref{TGL4}-D] \tabD{$GL(4,q)$ with $3\mid (q+1)$}
\med

In view of Lemma \ref{LGL4} we have \tabDD{$\psi=1_G+\chi_{14}(a)+ \chi_{18}(0,a)
+\chi_{23}((q^2-1)b)$} 

\smallskip
\begin{center}
\begin{footnotesize}
$\left.
%
 \right.$
\end{footnotesize}
\end{center}
\end{tab}

\subsection{Groups $GL(5,q)$}\label{TGL5}

\begin{tab}[\ref{TGL5}-A] \tabA{$GL(5,q)$}

\smallskip
\begin{center}
\begin{footnotesize}
$\left.
%
\right.$
\end{footnotesize}
\end{center}
\end{tab}
\newpage

\begin{tab}[\ref{TGL5}-D] \tabD{$GL(5,q)$ with $3\mid (q+1)$}

\smallskip
In view of Lemma \ref{LGL5} we have 
\tabDD{$\psi=1_G+\chi_{24}(0,a)+\chi_{30}(a, 0)+\chi_{42}((q^2-1)b, 0)$} 

\begin{center} 
\begin{footnotesize}
$\left.
%
 \right.$
\end{footnotesize}
\end{center}
\end{tab}

\subsection{Groups $GL(6,q)$}\label{TGL6}

\begin{tab}[\ref{TGL6}-A] Some \tabA{$GL(6,q)$}

\smallskip
\begin{center}
\begin{footnotesize}
$\left.
%
\right.$
\end{footnotesize}
\end{center}
\end{tab}

\subsection{Groups $CSp(6,q)$}\label{TSp6}

\begin{tab}[\ref{TSp6}-A] Some \tabA{$CSp(6,q)$ with $q$ odd}

\begin{center}
\begin{footnotesize}
$\left.
%
\right.$
\end{footnotesize}
\end{center}
\end{tab}

\end{document}